\documentclass[psfig,amssymb,amsfonts]{amsart}
\usepackage{amssymb}
\usepackage{latexsym}
\usepackage{graphicx}

\newcommand{\be}{{\bf{e}}}

\newcommand{\lN}{{\mathcal{N}}}
\newcommand{\lE}{{\mathcal{E}}}
\newcommand{\spt}{~{\mbox{spt}}}
\newcommand{\Q}{\mathbb Q}
\newcommand{\C}{{{\mathbb C}}}
\newcommand{\R}{{{\mathbb R}}}
\newcommand{\Z}{{{\mathbb Z}}}
\newcommand{\T}{{{\mathbb T}}}

\def\bigO{{\mathcal{O}}}

\newcommand\parent{{\operatorname{parent}}}
\newcommand\sibling{{\operatorname{sibling}}}
\newcommand\spouse{{\operatorname{spouse}}}
\newcommand\child{{\operatorname{child}}}

\def\be#1{\begin{equation} \label{#1}}
\def\bi{\begin{itemize}}
\def\bs{\begin{split}}

\def\es{\end{split}}

\def\ba{\begin{align}}
\def\bas{\begin{align*}}

\def\ea{\end{align}}
\def\eas{\end{align*}}
\def\covers{{\twoheadrightarrow}}
\def\Im{{\hbox{Im}}}

\def\Re{{\hbox{Re}}}

\def\C{{\hbox{\mathbb C}}}

\def\R{{\mathbb R}}
\def\rad{{\mathcal{R}}} 
\def\Z{{\mathbb Z}}
\def\T{{{\mathbb T}}}
\def\C{{{\mathbb C}}}

\def\eps{\varepsilon}

\def\id{{\operatorname{id}}}

\newtheorem{theorem}{Theorem}
\theoremstyle{definition}
\newtheorem{definition}{Definition}
\theoremstyle{remark}
\newtheorem{remark}{Remark}
\theoremstyle{proposition}
\newtheorem{proposition}{Proposition}
\theoremstyle{lemma}
\newtheorem{lemma}{Lemma}
\theoremstyle{corollary}
\newtheorem{corollary}{Corollary}

\numberwithin{equation}{section}
\numberwithin{lemma}{section}
\numberwithin{remark}{section}
\numberwithin{theorem}{section}
\numberwithin{corollary}{section}
\numberwithin{proposition}{section}
\numberwithin{definition}{section}

\begin{document}

\title[Weakly  turbulent solutions  for NLS]{Weakly turbulent solutions
for the cubic defocusing
  nonlinear Schr\"odinger equation}


\vspace{-0.3in}

\author{J. Colliander}
\thanks{J.C. is supported in part by N.S.E.R.C. grant RGPIN 250233-07.}
\address{\small University of Toronto}

\author{M. Keel}
\thanks{M.K. is supported in part by the Sloan Foundation and N.S.F. Grant DMS0602792.}
\address{\small University of Minnesota, Minneapolis}

\author{G. Staffilani}
\thanks{G.S. is supported in part by N.S.F. Grant
DMS 0602678.}
\address{\small Massachusetts Institute of Technology}

\author{H. Takaoka}
\address{\small Kobe University}
\thanks{H.T. is supported in part by J.S.P.S. Grant No. 13740087}

\author{T. Tao}
\thanks{T.T. is supported by NSF grant DMS-0649473 and a grant from the
Macarthur Foundation.}
\address{\small University of California, Los Angeles}

\date{12 August 2008}


\begin{abstract}

We consider the cubic defocusing nonlinear Schr\"odinger equation on the
two dimensional torus.  We exhibit smooth solutions for which
the  support of the conserved energy moves to higher Fourier modes.
This weakly turbulent behavior is quantified by the growth of
higher Sobolev norms:  given any $\delta \ll 1, K \gg 1, s >1$, we construct
smooth initial data $u_0$ with $\|u_0\|_{{H}^s} < \delta$, so that
the corresponding time evolution $u$ satisfies $\|u(T)\|_{{H}^s} > K$
at some time $T$.  This growth occurs despite the Hamiltonian's bound on $\|u(t)\|_{\dot{H}^1}$ and
despite the conservation of the quantity $\|u(t)\|_{L^2}$.

The proof contains two arguments which may be of interest beyond
the particular result described above.  The first is a
construction of the solution's frequency support that
simplifies the system of ODE's describing each Fourier mode's evolution.
The second is a construction of solutions to these simpler systems of
ODE's which begin near one invariant manifold and ricochet from arbitrarily small
neighborhoods of an arbitrarily large number of other invariant manifolds.
The techniques used here are related to but are distinct
from those traditionally used
to prove Arnold Diffusion in perturbations of Hamiltonian systems.


\end{abstract}

\maketitle

\tableofcontents

\section{Introduction}


We consider  the periodic defocusing cubic nonlinear Schr\"odinger
(NLS) equation
\begin{equation}\label{nls}
\left\{
\begin{matrix}
-i \partial_t u + \Delta u  = |u|^2 u \\
u(0,x)  := u_0(x) \\
\end{matrix}
\right.
\end{equation}
where $u(t,x)$ is a complex valued function with the spatial variable $x$ lying in the torus $\T^2 := \R^2 / (2\pi \Z)^2$.
Equations such as \eqref{nls} arise as models in various  physical settings, including the description of the envelope
of a general dispersive wave in a weakly nonlinear medium, and more specifically in some models of surface water waves.
(See e.g. the survey in \cite{sulemsulem}, Chapter 1.)

We shall always take the initial data $u_0(x)$ to be smooth.
Recall (see e.g. \cite{cazenave, taotext}) that smooth solutions to \eqref{nls} exhibit both conservation of the Hamiltonian,
\begin{align}
E[u](t) & := \int_{\T^2} \frac{1}{2} |\nabla u|^2 + \frac{1}{4} |u|^4 dx (t) \nonumber \\
  & = E[u](0) \label{energyconservation},
  \end{align}
and conservation of mass, or $L^2(\T^2)$ norm,
\begin{align}
\int_{\T^2} |u|^2 dx (t) & = \int_{\T^2} |u|^2 dx (0) \label{massconservation},
\end{align}
for all $t > 0$.  The local-in-time well-posededness result of Bourgain \cite{BourgainXsb}
for data $u_0 \in H^s(\T^2), s > 0$, together with these conservation
laws, immediately gives the existence of a global smooth solution to
\eqref{nls} from smooth initial data.

We are interested in whether there
exist solutions of \eqref{nls} which initially oscillate only on
scales comparable to the spatial period
and eventually oscillate on arbitrarily short spatial
scales. One can quantify
such motion in terms of the growth in time of higher Sobolev norms
$\|u(t) \|_{H^s(\T^2)}, s > 1$,  defined using the Fourier transform
by,
\begin{align}
\|u(t)\|_{H^s(\T^2)} & :=\|u(t, \cdot)\|_{H^s(\T^2)} \; := \;
\left( \sum_{n \in \Z^2} \langle n \rangle^s |\hat u(t,n)| \right)^{\frac{1}{2}}  \label{hsnorm}
\end{align}
where $\langle n \rangle := (1 + |n|^2)^\frac{1}{2}$ 
and\footnote{In what follows, we omit the factors of $2 \pi$ arising in definitions of  
 the Fourier transform and its inverse, as these play no role in our analysis.},
\begin{align*}
\hat u(t,n) & := \int_{\T^2} u(t,x) e^{-in \cdot x} dx.
\end{align*}
For example, together  \eqref{energyconservation} and \eqref{massconservation} give a uniform
bound
\begin{align}
\label{h1bound}
\|u(t)\|_{H^1(\T^2)} & =  \left( \sum_{n \in \Z^2} \langle n \rangle^2 |\hat u(t,n)|^2 \right)^{\frac{1}{2}}
\; \leq \; C
\end{align}
 on the solution's $H^1$ norm.
Hence, for fixed $s > 1$, $\|u(t)\|_{H^s(\T^2)}$ could grow in time if the terms contributing
substantially to the sum on the left hand side of \eqref{h1bound} correspond,
loosely speaking, to ever higher $|n|$.    From this point of view then, we are interested
in whether the energy of a solution to \eqref{nls} can be carried by higher and higher Fourier modes.

The one space-dimensional analog of \eqref{nls} is completely integrable \cite{ZakharovShabat}, and
the higher conservations laws in that case imply $\|u(t)\|_{H^s(\T^1)} \leq C(\|u(0)\|_{H^s(\T^1)}), s \geq  1$ for all $t > 0$.
It is unknown (see e.g. \cite{Bourgain_GAFA_Survey}) whether unbounded growth in $H^s$,  $s>1$, is possible  in dimensions 2 and higher, let
alone whether such growth is generic.  The
main result of this paper is the construction of solutions to \eqref{nls} with arbitrarily large growth in
higher Sobolev norms,

\begin{theorem}
\label{main}  Let $1 < s $, $K \gg 1$, and $0 < \delta \ll 1$ be given parameters.  Then there
exists a global smooth solution $u(t,x)$ to \eqref{nls} and a time $T  > 0$ with
$$ \| u(0) \|_{H^s} \leq \delta$$
and
$$ \| u(T) \|_{ H^s} \geq K .$$
\end{theorem}

 Note that, in view of
\eqref{energyconservation}, \eqref{massconservation}, the growth
constructed here must involve both
movement of energy to higher frequencies, and movement of mass to
lower frequencies.  (The mass associated to the higher and higher frequency energy must
be decreasing by energy conservation.  This must be balanced, by mass conservation,
by more and more mass at low frequencies.)  Recall again that any smooth data in \eqref{nls} evolves
globally in time.  While finite and infinite time blowup results are known for focusing analogs of
\eqref{nls} (e.g. \cite{BourgainSmackDown},
\cite{BourgainBook}, \cite{kavian}, \cite{MR5}, 
\cite{Planchon-Raphael}), the mechanism responsible for the $H^s$ norm growth in Theorem \ref{main}
is distinct from these blowup dynamics.

Using the conservation laws
  \eqref{energyconservation}, \eqref{massconservation} and the
  Sobolev embedding theorem, we observe
  that, for $s = 1$, we have a stability property near zero,
$$\left( \limsup_{|t| \rightarrow \infty}
    \left[ \sup_{\|u_0\|_{H^1} \leq \delta} \| u(t) \|_{H^1} \right]
  \right) \leq C \delta.
$$
Theorem  \ref{main} implies a different behavior in the range $s>1$.
Since $\delta$ may be
chosen to be arbitrarily small and $K$ may be chosen arbitrarily
large in Theorem \ref{main}, we observe the following:

\begin{corollary}[$H^s$ instability of zero solution]
  \label{zerounstable}
The global-in-time solution map taking the initial data $u_0$ to the
associated solution $u$ of \eqref{nls} is strongly unstable in $H^s$
near zero for all $s>1$:
\begin{equation}
  \label{zerounstableequation}
  \inf_{\delta >0} \left( \limsup_{|t| \rightarrow \infty}
    \left[ \sup_{\|u_0\|_{H^s} \leq \delta} \| u(t) \|_{H^s} \right]
  \right) = \infty.
\end{equation}
\end{corollary}

It does not follow from \eqref{zerounstableequation}, nor directly from Theorem
\ref{main}, that there exists initial data $u_0 \in H^s$ for some $s>1$
which evolves globally in time and satisfies $\limsup_{|t| \rightarrow
  \infty} \|u(t) \|_{H^s} = \infty$. As remarked above, this remains an interesting open
question \cite{Bourgain_GAFA_Survey}.  In Section \ref{appendix} below,  we do prove that
\eqref{nls} has no nontrivial solutions which scatter - i.e. which approach a solution to the
linear equation at time $t = \infty$.  Note  that Theorem \ref{main} and Corollary \ref{zerounstable} also hold true for higher dimensional
generalizations of (1.1) by
considering solutions which are only dependent upon two of the
spatial coordinates.




Theorem \ref{main} is motivated by a diverse body of literature which we briefly review now, including upper bounds on
the possible growth of Sobolev norms for \eqref{nls}, lower bounds for
related models, and extensive work on the so-called
\emph{weak-turbulence theory} of related wave models.

A  straightforward  iteration argument based on the local theory
\cite{BourgainXsb} shows that high Sobolev norms of solutions of
\eqref{nls} can grow no faster than exponential-in-time. Bourgain
used refinements \cite{bourg1} of the Strichartz
inequality  to prove  polynomial-in-time
upper bounds \cite{BourgainHighSobolev} on Sobolev norm growth.  These results were sharpened,
using a normal forms reduction, in \cite{Bourgainsharpestyet}. (See
also \cite{CDKS},  \cite{StaffilaniDuke}.)

Previous examples of growth in the higher Sobolev norms
of solutions to \eqref{nls} are, to the best of our knowledge, found only in the work of
Kuksin  \cite{KuksinOscillationsGAFA1997} (see also related work in \cite{KuksinCMP1996},
 \cite{KuksinTurbulenceGAFA1997}, \cite{Kuksinwave}, \cite{KuksinGAFA1999}),  
where the following small dispersion nonlinear Schr\"odinger equation is 
considered\footnote{In these papers, Kuksin
    considers a variety of deterministic and random nonlinear Schr\"odinger equations in addition
    to \eqref{NLSZeroDisp}.  In
    \cite{Kuksinwave} the analog of
    \eqref{NLSZeroDisp} for the second order wave equation is considered.  Note too that we write (1.7) with the 
    sign convention 
on the time variable as in \eqref{nls}, rather than the convention used in 
\cite{KuksinOscillationsGAFA1997}, but this makes no substantial difference in describing the results.},
\begin{equation}
  \label{NLSZeroDisp}
 - i \partial_t w + \delta \Delta w = |w|^2 w
\end{equation}
with odd, periodic boundary conditions, 
where $\delta$ is taken to be a small
parameter. Kuksin shows, among other results and in various
formulations, that smooth norms of solutions of \eqref{NLSZeroDisp}
evolving from relatively generic data with unit $L^2$ norm
eventually grow larger than a negative power of $\delta$.
This result can be compared with Theorem \ref{main} as follows.
Suppose $w$ is a solution of \eqref{NLSZeroDisp}. The
rescaled function $u_\delta (t,x) = \delta^{-\frac{1}{2}} w( \delta^{-1} t, x)$
satisfies the PDE \eqref{nls}, the same equation as \eqref{NLSZeroDisp}
with $\delta =1$. Note that $\| D_x^s u_\delta (0, \cdot) \|_{L^2} =
\delta^{-\frac{1}{2}} \| D_x^s w(0, \cdot) \|_{L^2_x}$. Hence,
in the context of \eqref{nls}, the above described result in \cite{KuksinOscillationsGAFA1997}
addresses growth of solutions emerging from sufficiently large initial data, where the size of the data
depends on the amount of growth desired. In contrast, Theorem
\ref{main} concerns growth  from  small, or  order one,  data\footnote{Theorem \ref{main} is stated
for the case $\delta \ll 1$.  However as we hope is clear from the discussion in section
\ref{scalingargument} below, the construction allows any choice of $\delta > 0$.}
and its proof involves a strong interplay between the equation's dispersion and
nonlinearity.

In contrast to the question for the nonlinear problem \eqref{nls}, there are several
results on the growth of Sobolev norms in linear Schr\"odinger equations with potentials,
\begin{align}
\label{nlspotential}
- i \partial_t v + \Delta v + V(t,x) v & = 0,  \quad \quad x \in \T^2.
\end{align}
For example, for $V$ smooth in space, and random in time (resp.
smooth and periodic in $x$ and $t$), Bourgain proved that
generic data grows polynomially \cite{BourgainSmooth} (resp. logarithmically
\cite{BourgainQuasiPeriodic}) in time.


Bourgain has constructed \cite{BourgainPeriodicKdV},
\cite{BourgainAspects} Hamiltonian PDEs with solutions with divergent
high Sobolev norms. These constructions are based on perturbation
arguments off linear equations with spectrally defined Laplacian and
also involve
somewhat artificial{\footnote{See Remark 2 on p. 303 of
    \cite{BourgainHighSobolev} for further discussion. 
}} choices of
nonlinearities. In \cite{BourgainHighSobolev},
solutions with divergent Sobolev norms are constructed for a wave
equation with a natural cubic nonlinearity but still involving a
spectrally defined Laplacian.

Bourgain has also shown
\cite{BQuasiPeriodicAnnals}, \cite{BParkCity} that there is an abundance of
time quasi-periodic
solutions of \eqref{nls}.

As mentioned  at the outset, the growth of higher Sobolev norms is just
one way to quantify the diffusion of energy to higher and  higher modes - also called
a forward cascade, or a direct cascade.  Furthur motivation for Theorem \ref{main}, and indeed for much of
the work cited above, is the literature of analysis, physics, numerics,
heuristics and conjectures regarding this phenomenon in related wave models.   For example, since the early 1960's a 
so-called \emph{weak turbulence theory} (alternatively
\emph{wave turbulence theory})  has been developed which gives a statistical description of the
forward cascade in various ``weakly interacting'' dispersive wave models, mainly based on 
the analysis of associated kinetic equations for the evolution of the Fourier modes \cite{benney_newell, benney_saffman, hasselmann,  litvak,  
 zakharov_early}.   There are a great many subsequent  works 
both within and outside the framework of weak turbulence theory which address the passage of energy to higher modes in dispersive wave models.  
We are not prepared to give here a representative survey, but see   
\cite{ deville_et_al,  ZakharovOptical, kramer, majda_mclaughlin_tabak, newell_notes} and references therein for examples.

In broadest outline, the proof of Theorem \ref{main} proceeds by first
viewing \eqref{nls} as an infinite dimensional system of O.D.E.'s in
$\{a_n(t)\}_{n \in \Z^2}$, where $a_n(t)$ is closely related to the
Fourier mode $\hat{u}(t,n)$ of the solution.  We identify a related
system, which we call the {\emph{resonant system}}, that we use
as an approximation to the full system.  The goal then will be to build
a solution $\{r_n(t)\}_{n \in \Z^2}$ to the resonant system
which grows in time.  This is accomplished by choosing the initial
data $\{r_n(0)\}_{n \in \Z^2}$ to be supported on a certain frequency set
$\Lambda \subset \Z^2$ in such a way that the resonant system of O.D.E.'s
collapses to an even simpler, finite dimensional system that we call
{\emph{the Toy Model System}}, and whose solution we denote by
$b(t) \equiv \{b_1(t), b_2(t), \ldots , b_N(t)\}$.  Each variable $b_i(t)$
will represent how a certain subset of the $\{r_n(t)\}_{n \in \Z^2}$
evolves in time.  There are two independent but related
ingredients which complete the proof of the main Theorem.  First, we show
the existence of the frequency set $\Lambda$, which is defined in
terms of the desired Sobolev
norm growth and according to a wish-list of geometric and
combinatorial properties aimed at simplifying the resonant system. Second,
we show that the Toy Model System exhibits unstable orbits that travel
from an arbitrarily small neighborhood of one invariant manifold to near
a distant invariant manifold.  It is this instability which is
ultimately responsible for the support of the solution's energy moving to
higher frequencies.  Instabilities like this have been remarked on
at least as far back as \cite{poincare}, but have been studied with
increasing interest  since the paper of Arnold \cite{arnold}.
Our construction has similarities with
previous work on so-called ``Arnold Diffusion'' and more
general instabilities in Hamiltonian systems (see e.g.
\cite{alekseev, arnold, bertibolle,  bessiI, bessi_upper_bounds,
   bolotin_treschev, bourgainkaloshin, cheng_yan, contreras,
 craig_book,  dls00, dls03,  llava_memoirs, douady_first,
douady_le_calvez, fontich_martin, kaloshin_levi, mather_announce, moeckel_02,
 moeckel_expository,
treschevsolo}).  However, the particular approach taken here seems
to be different than arguments presently in the literature and might
be of independent interest.  There are several different definitions of
``Arnold Diffusion'' (see e.g. the remarks in the introduction of the survey
by Delshams, Gidea, and de la Llave in \cite{craig_book}).  Our use of this
term for the orbits described above,  which travel close to a sequence
of invariant manifolds and lead to a change in the action variable of order 1,
is in line with some, but not all of these conventions.
We emphasize also that ``diffusion'' is used here  as in much
of the literature, in a nontechnical - even colloquial - sense following
\cite{arnold}.

The paper is organized as follows.  Section \ref{overview} provides
a more detailed overview of the argument, giving  the
proof of Theorem \ref{main} modulo some intermediate claims.  Section
\ref{multi-hop} contains a proof that the Toy Model System exhibits the
unstable orbits described above.  Section \ref{NumberTheory} constructs
the frequency set $\Lambda$.  Finally, we prove as a postscript in Section \ref{appendix} that no nontrivial
solutions to \eqref{nls} scatter to linear solutions.

\section{Overview and Proof of Main Theorem} \label{overview}

\subsection{Preliminary Reductions:  NLS as an Infinite System of ODE's}

The  equation \eqref{nls} has gauge freedom: upon writing
\begin{equation}
  \label{vfromu}
  v(t,x) = e^{iGt} u (t,x), ~ G\in \R ~~\mbox(constant),
\end{equation}
the $NLS$ equation \eqref{nls} becomes the following equation for $v$,
\begin{equation}
  \label{NLSG}
  (-i\partial_t + \Delta) v = (G + |v|^2) v
\end{equation}
with the same initial data. (We will soon choose the constant $G$ to
achieve a cancellation.)

We make the {\it{Fourier Ansatz}}, motivated by the explicit solution
of the linear problem associated to $NLS$.  For solutions of the $NLS_G$
equation \eqref{NLSG}, we write
\begin{equation}
  \label{FourierAnsatz}
  v(t,x) = \sum_{n \in \Z^2} a_n (t) e^{i(n \cdot x + |n|^2 t)}.
\end{equation}
We consider in this paper only smooth solutions,  so the series on the right of
\eqref{FourierAnsatz} is absolutely summable.
Substituting \eqref{FourierAnsatz} into the equation \eqref{nls} and equating Fourier coefficients
 for both sides gives the following infinite system of equations for $a_n(t)$,
\begin{equation}
  \label{FNLSG}
  -i \partial_t a_n = G a_n + \sum_{\begin{matrix}
                                       \scriptstyle{n_1, n_2 , n_3 \in
                                       \Z^2} \\
                                       \scriptstyle{n_1 - n_2 + n_3 =
                                       n} \\
                                      \end{matrix}}
a_{n_1}  \overline{a_{n_2}} a_{n_3} e^{i \omega_4 t},
\end{equation}
where
\begin{equation}
  \label{omegafour}
  \omega_4 = |n_1|^2 - |n_2|^2 + |n_3|^2 - |n|^2.
\end{equation}
Certain terms on the right hand side of \eqref{FNLSG} can be removed
by correctly choosing the gauge parameter $G$.
We describe this cancellation this now.  Split the sum on the right hand side of \eqref{FNLSG} into the following terms,
  \begin{align}
 \label{sumdecomposition}
    \sum_{\begin{matrix}
                                        \scriptstyle{n_1, n_2 , n_3 \in
                                        \Z^2} \\
                                        \scriptstyle{n_1 - n_2 + n_3 =
                                        n} \\
                                       \end{matrix}}
=\sum_{\begin{matrix}
                                       \scriptstyle{n_1, n_2 , n_3 \in
                                       \Z^2} \\
                                       \scriptstyle{n_1 - n_2 + n_3 =
                                       n} \\
                                       \scriptstyle{n_1, n_3 \neq n} \\
                                     \end{matrix}}
+
 \sum_{\begin{matrix}
                                       \scriptstyle{n_1, n_2 , n_3 \in
                                       \Z^2} \\
                                       \scriptstyle{n_1 - n_2 + n_3 =
                                       n} \\
                                     \scriptstyle{n_1 = n}
                                      \end{matrix}}
+
\sum_{\begin{matrix}
                                       \scriptstyle{n_1, n_2 , n_3 \in
                                       \Z^2} \\
                                       \scriptstyle{n_1 - n_2 + n_3 =
                                       n} \\
                                     \scriptstyle{n_3 = n}
                                      \end{matrix}}
- \sum_{\begin{matrix}
                                       \scriptstyle{n_1, n_2 , n_3 \in
                                       \Z^2} \\
                                       \scriptstyle{n_1 - n_2 + n_3 =
                                       n} \\
                                     \scriptstyle{n_3 = n_1 = n}
                                      \end{matrix}}
             & := \mbox{Term I} +  \mbox{Term II} + \mbox{Term III} + \mbox{Term IV}.
\end{align}
Term IV here is not a sum at all, but rather $- a_n(t) |a_n(t)|^2$.  Terms II and III are single sums which by Plancherel's theorem and \eqref{massconservation} total,
\begin{align*}
2a_n(t) \cdot \sum_{m \in \Z^2} |a_m(t)|^2 & = 2 a_n(t) \cdot \|u(t)\|^2_{L^2(\T^2)} \\
& = 2 a_n(t) M^2,
\end{align*}
where we've written  $M := \|u(t)\|_{L^2(\T^2)}$.
We can cancel this with the first term on the right side of
  \eqref{FNLSG} by choosing $G = - 2M$ in \eqref{vfromu}.
Equation \eqref{FNLSG} takes then the following useful form, which we denote $\mathcal{F}NLS$,
\begin{equation}
  \label{FNLS}
  -i \partial_t a_n = - a_n |a_n |^2 + \sum_{n_1, n_2 , n_3 \in \Gamma
   (n) } a_{n_1} \overline{a_{n_2}} a_{n_3} e^{i \omega_4 t},
\end{equation}
where
\begin{equation}
  \label{Gamman}
  \Gamma (n) = \{ (n_1, n_2, n_3) \in (\Z^2)^3 : n_1 - n_2 + n_3 = n, n_1
  \neq n, n_3 \neq n \}.
\end{equation}
Note that at each time $u$ is easily recoverable from $v$, and both functions have identical Sobolev
norms.

One can easily show that  $\mathcal{F}NLS$ is locally well-posed  in $l^1 (\Z^2)$, and for
completeness we sketch the argument here. Define the trilinear operator
\begin{equation*}
\lN(t): l^1 (\Z^2) \times l^1 (\Z^2)
\times l^1 (\Z^2) \longmapsto l^1(\Z^2)
\end{equation*}
by
\begin{equation}
  \label{Ndefined}
  (\lN (t) (a,b,c) )_n = -a_n {\overline{b_n}} c_n + \sum_{n_1, n_2,
  n_3 \in \Gamma (n)} a_{n_1} {\overline{b_{n_2}}} c_{n_3} e^{i
  \omega_4 t}.
\end{equation}
With this notation, we can reexpress $\mathcal{F}NLS$ as $-i \partial_t a_n = (\lN (t) (a,a,a))_n.$

\begin{lemma}
  \label{trilinearlemma}
  \begin{equation}
    \label{triest}
    {{\| (\lN (t) (a,b,c))_n \|}_{l^1 (\Z^2 )} } \lesssim {{\| a
    \|}_{l^1 (Z^2)}}  {{\| b
    \|}_{l^1 (Z^2)}}  {{\| c
    \|}_{l^1 (Z^2)}}
  \end{equation}
\end{lemma}

\begin{proof}
The $l^1$ norm of the first term in \eqref{Ndefined} is
bounded by $\| a \|_{l^\infty} \|b \|_{L^\infty} \|c \|_{l^1}$, which
is bounded as claimed. For
the second term, take absolute values inside the $\sum_{n_1 , n_2 ,
  n_3 \in \Gamma (n)}$ and then replace the $n_3$ sum by the $n$ sum
using the $\Gamma(n)$ defining relation to observe the bound \eqref{triest}.
\end{proof}

\begin{remark}
  \label{lifetimeremark}
From Lemma \ref{trilinearlemma} and standard Picard iteration arguments one obtains local well-posedness in $l^1 (\Z^2)$. The local well-posedness result is valid on $[0,T]$ with
\begin{equation*}
  T \thicksim {{\| a(0)\|}_{l^1 (\Z^2)}^{-2}}.
\end{equation*}
The equation $-i \partial_t a = \lN (t) (a,a,a)$ behaves roughly like
the ODE $\partial_t a = a^3$ for the purposes of local existence theory.
\end{remark}

\subsection{Resonant and Finite Dimensional Truncations of $\mathcal{F}NLS$}

We define the {\it{set of all resonant non-self interactions}} $\Gamma_{res} (n) \subset \Gamma (n)$ by
\begin{equation}
  \label{GammaRes}
  \Gamma_{res}(n) = \{ (n_1, n_2, n_3) \in \Gamma (n) : \omega_4 =
  |n_1|^2 - |n_2|^2 + |n_3|^2 - |n|^2 = 0 \}.
\end{equation}
Note that $(n_1, n_2, n_3) \in \Gamma_{res} (n)$ precisely when  $(n_1 , n_2 , n_3 , n)$
form four corners of a nondegenerate rectangle with $n_2$ and  $n$
opposing each other, and similarly for $n_1 $ and $n_3$.
One way to justify this claim is to
first note the following symmetry:  the two conditions \eqref{Gamman},
\eqref{GammaRes} defining $\Gamma_{res}(n)$ imply directly that
$(n_1, n_2, n_3) \in
\Gamma_{res}(n)$ if and only if $(n_1 - n_0, n_2 - n_0, n_3 - n_0) \in \Gamma_{res}(n - n_0)$ for
any $n_0 \in \Z^2$.  Choosing $n_0 = n$,  it suffices to prove the above
geometric interpretation for $\Gamma_{res}(n)$
in the case $n = 0$, and this follows immediately from the law of cosines.

Heuristically, the resonant interactions dominate in \eqref{FNLS}
because they do not
contain the $e^{i \omega_4 t}$ factor that oscillates in time.
We approximate solutions
of \eqref{FNLS} by simply discarding the nonresonant interactions - and
we make this approximation rigorous with Lemma
\ref{approx} below.
For now, we simply define the {\it{resonant truncation}} $R\mathcal{F}NLS$
of $\mathcal{F}NLS$ by,
\begin{equation}
  \label{RFNLS}
-i \partial_t r_n = - r_n |r_n|^2 + \sum_{(n_1 , n_2 , n_3 ) \in
 \Gamma_{res} (n)} r_{n_1} {\overline{r_{n_2}}} r_{n_3}.
\end{equation}

Even after making the resonant approximation, we still have an infinite ODE to work with - $n$ can range freely over $\Z^2$.  Our strategy is to choose initial data for which the the system simplifies
in several ways.

Suppose we have some finite set of frequencies $\Lambda$ that satisfies the  following two properties:
\begin{itemize}
\item ($\mbox{Property I}_\Lambda$:  Initial data) The initial data $r_n(0)$ is entirely supported in $\Lambda$ (i.e. $r_n(0) = 0$ whenever $n \not \in \Lambda$).

\item ($\mbox{Property II}_\Lambda$:  Closure) Whenever $(n_1,n_2,n_3,n_4)$ is a rectangle in $\Z^2$ such that three of the corners lie in $\Lambda$, then the fourth
corner must necessarily lie in $\Lambda$.  In terms of previously defined notation, this can
be expressed as follows,
\begin{equation}
\label{closed}
  (n_1 , n_2 , n_3 ) \in \Gamma_{res} (n) , n_1 , n_2 , n_3  \in
  \Lambda  \implies n \in \Lambda.
\end{equation}
\end{itemize}

Then one can show that $r_n(t)$ stays supported in $\Lambda$ \emph{for all time}.  Intuitively, this is because the non-linearity in \eqref{RFNLS} cannot excite any modes outside of $\Lambda$ if one only starts with modes inside $\Lambda$.

For completeness, we include here the short argument
that condition \eqref{closed} guarantees a finite dimensional model
if the finite set $\Lambda$ contains the support of the initial data
$r(0) = \{r_n(0)\}_{n \in \Z^2}$ for \eqref{RFNLS}.
\begin{lemma}
  \label{resonantclosurelemma}
If $\Lambda$ is a finite set satisfying   $\mbox{Property I}_\Lambda$, $\mbox{Property II}_\Lambda$
above, and $r(0) \longmapsto r(t)$
solves $R \mathcal{F} NLS$ \eqref{RFNLS} on $[0,T]$ then for all $t \in [0,T]
\spt [r(t)] \subset \Lambda$.
\end{lemma}

\begin{proof}
Define,
$$ B(t) := \sum_{n \not \in \Lambda} |r_n(t)|^2.$$
Thus $B(0) = 0$, and from the closure property and the
boundedness\footnote{The argument sketched previously
for local well-posedness in
$l^1$ of \eqref{FNLS} carries over to \eqref{RFNLS} without change.}
of $r_n(t)$ we get the differential inequality
$$ |B'(t)| \leq C |B(t)|.$$
Thus, by the Gronwall estimate, $B(t) = 0$ for all $t$.  Hence,
none of the modes outside of $\Lambda$ are excited.
\end{proof}

Our strategy thus far is to choose initial data for \eqref{nls} with Fourier support in
such a set $\Lambda$ - so that the resonant system \eqref{RFNLS} reduces to a finite dimensional
system.  We now place more conditions on the set $\Lambda$ and on the initial
data which bring about
further simplifications.

We demand that for some positive integer $N$ (to be specified later), the set
$\Lambda$ splits into $N$ disjoint \emph{generations}
$\Lambda = \Lambda_1 \cup \ldots \cup \Lambda_N$
which satisfy the properties we specify below, after first introducing
necessary terminology.  Define a \emph{nuclear family} to
be a rectangle $(n_1, n_2, n_3, n_4)$ where the frequencies
$n_1, n_3$ (known as the ``parents'') live in a generation $\Lambda_j$, and
the frequencies $n_2, n_4$ (known as the ``children'') live
in the next generation $\Lambda_{j+1}$.  Note that if $(n_1,n_2,n_3,n_4)$ is
a nuclear family, then so is $(n_1,n_4,n_3,n_2)$,
$(n_3,n_2,n_1,n_4)$, and $(n_3,n_4,n_1,n_2)$; we shall call these
the \emph{trivial permutations} of the nuclear family.  We require the following
properties (in addition to the initial data and closure hypotheses described above):

\begin{itemize}

\item ($\mbox{Property III}_\Lambda$: Existence and uniqueness of spouse and children) For any $1 \leq j < N$ and any $n_1 \in \Lambda_j$ there exists a unique nuclear family
$(n_1,n_2,n_3,n_4)$ (up to trivial permutations) such that $n_1$ is a parent of this family.  In particular each $n_1 \in \Lambda_j$ has a unique spouse
$n_3 \in \Lambda_j$ and has two unique children $n_2, n_4 \in \Lambda_{j+1}$ (up to permutation).

\item ($\mbox{Property IV}_\Lambda$: Existence and uniqueness of sibling and parents) For any $1 \leq j < N$ and any $n_2 \in \Lambda_{j+1}$ there exists a unique
nuclear family $(n_1,n_2,n_3,n_4)$ (up to trivial permutations) such that $n_2$ is a child of this family.  In particular each $n_2 \in \Lambda_{j+1}$
has a unique sibling $n_4 \in \Lambda_{j+1}$ and two unique parents $n_1, n_3 \in \Lambda_j$ (up to permutation).

\item ($\mbox{Property V}_\Lambda$:  Nondegeneracy) The sibling of a frequency $n$ is never equal to its spouse.

\item ($\mbox{Property VI}_\Lambda$:  Faithfulness) Apart from the nuclear families, $\Lambda$ contains no other rectangles.  (Indeed, from the Closure hypothesis, it does not
even contain any right-angled triangles which are not coming from a nuclear family.)

\end{itemize}

Despite the genealogical analogies, we will not assign a gender to any individual frequency (one could do so, but it is somewhat artificial);
these are asexual families.  Since every pair of parents in one generation corresponds to exactly one pair of
children in the next, a simple counting argument now shows that
each generation must have exactly the same number of frequencies.

At present it is not at all clear that such a $\Lambda$ even exists for any given $N$.
But assuming this for the moment, we can simplify
the equation \eqref{RFNLS}.  It now becomes
\begin{equation}\label{gamma-rectangle-generations}
-i \partial_t r_n(t) = -|r_n(t)|^2 r_n(t) + 2 r_{n_{\child-1}}(t) r_{n_{\child-2}}(t) \overline{r_{n_{\spouse}}(t)}
+ 2 r_{n_{\parent-1}}(t) r_{n_{\parent-2}}(t) \overline{r_{n_{\sibling}}(t)}
\end{equation}
where for each $n \in \Lambda_j$, $n_{\spouse} \in \Lambda_j$ is its spouse,
$n_{\child-1}$, $n_{\child-2} \in \Lambda_{j+1}$ are its two children,
$n_{\sibling} \in \Lambda_j$ is its sibling, and $n_{\parent-1}$, $n_{\parent-2} \in \Lambda_{j-1}$ are its parents.  If $n$ is in the last
generation $\Lambda_N$ then we omit the term involving spouse and children; if $n$ is in the first generation $\Lambda_1$ we omit
the term involving siblings and parents.  The factor ``2'' arises from the
trivial permutations of nuclear families.

We now simplify this ODE by making yet another assumption, this time again involving
 the initial data:

\begin{itemize}
\item ($\mbox{Property VII}_\Lambda$: Intragenerational equality) The function
$n \mapsto r_n(0)$ is constant on each generation $\Lambda_j$.  Thus $1 \leq j \leq N$
and $n, n' \in \Lambda_j$ imply $r_n(0) = r_{n'}(0)$.
\end{itemize}

It is easy to verify by another Gronwall argument that if one has intragenerational
equality at time 0 then one has intragenerational
equality at all later times.  This is basically because the
frequencies from each generation interact with that generation and with
its adjacent generations in exactly the same way (regardless of what the
combinatorics of sibling, spouse, children, and parents are).
Thus we may collapse the function $n \mapsto r_n(t)$, which is currently
a function on $\Lambda = \Lambda_1 \cup \ldots \cup \Lambda_N$,
to the function $j \mapsto b_j(t)$ on $\{1,\ldots,N\}$, where $b_j(t) \
;= r_n(t)$ whenever $n \in \Lambda_j$.  Thus we describe the evolution
by using a single complex scalar for each generation.
The ODE \eqref{gamma-rectangle-generations} now collapses to the following
system
that we call the {\emph{Toy Model System}}
\begin{equation}
\label{toymodel}
-i \partial_t b_j(t) =-|b_j(t)|^2 b_j(t) + 2 b_{j-1}(t)^2 \overline{b_j}(t) +
2 b_{j+1}(t)^2 \overline{b_j}(t),
\end{equation}
with the convention\footnote{We show rigorously later that the support of $b(t)$ is
 constant in time - so that $b_0, b_{N+1}$ will vanish as a consequence of
 our choice of initial data.  In other words, $b_0(t)$ and $b_{N+1}(t)$ remain zero
 for all time because of the dynamics induced by the resonant system, and not by
 a possibly artificial convention introduced alongside these new variables.}
 that $b_0(t) = b_{N+1}(t) = 0$.


There are three main ingredients in the proof of Theorem \ref{main}.
The first ingredient is the construction of the
finite set $\Lambda$ of frequencies.
The second ingredient is the proof that the Toy Model System described above exhibits a particular
instability:  we show that there exist
solutions of the Toy Model which thread through small neighborhoods of an arbitrary
number of distinct invariant
tori.  More specifically,
we show that we have a ``multi-hop'' solution to this ODE
in which the mass is initially concentrated at $b_3$ but eventually ends up at $b_{N-2}$.
In terms of the resonant system \eqref{RFNLS}, this instability corresponds to the growth of higher
Sobolev norms.  The third ingredient is an Approximation Lemma which gives
conditions under which  solutions of $R{\mathcal{F}}NLS$ \eqref{RFNLS},
and hence solutions corresponding to the Toy Model System, approximate
an actual solution of the original NLS equation.  When the conditions of this Approximation
Lemma are satisfied, it
is enough to construct  a solution evolving
according to the Toy Model which exhibits the desired growth in $H^s$. A scaling
argument shows that these conditions can indeed be satisfied, and glues the three ingredients
together to complete the proof.  We now detail the claims of the three ingredients and prove
Theorem \ref{main} modulo these intermediate claims.



\subsection{First Ingredient:  The Frequency Set $\Lambda$}


The approximation to \eqref{FNLS} that we ultimately study is the time evolution of very particular
data under the equation with the nonresonant part of the nonlinearity removed.  The very particular
initial data $u(0)$ that we construct has Fourier support on a set $\Lambda \subset \Z^2$ which
satisfies one more important condition in addition to those described above.

The construction of
this frequency set $\Lambda$ is carried out in detail in Section \ref{NumberTheory}. Here we
record the precise claim we make about the set.

\begin{proposition}[First Ingredient:  the frequency set $\Lambda$]  \label{frequencies}
Given parameters $\delta \ll 1, K \gg 1$, we can find an $N \gg 1$ and a set of frequencies
$\Lambda \subset \Z^2$ with,
\begin{align*}
\Lambda & = \Lambda_1 \cup \Lambda_2 \cup \ldots \cup \Lambda_N \quad \text{disjoint union}
\end{align*}
which satisfies $\mbox{Property II}_\Lambda$ -
$\mbox{Property VI}_\Lambda$\footnote{Note that $\mbox{Property I}_\Lambda$,
$\mbox{Property VII}_\Lambda$ will be easily satisfied when we choose our initial data.}
  and also,
\begin{align}
 \label{normexplosion}
 \frac{\sum_{n \in \Lambda_N} |n|^{2s}}{\sum_{n \in \Lambda_1} |n|^{2s}} & \gtrsim \frac{K^2}{\delta^2}.
 \end{align}
 In addition, given any $\rad \gg C(K, \delta)$, we can ensure that $\Lambda$ consists of $N \cdot 2^{N-1}$
 disjoint frequencies $n$ satisfying $|n| \geq \rad$.
 \end{proposition}
 We will use the term {\it generations} to describe the sets $\Lambda_j$ that make up $\Lambda$.
 The {\it norm explosion} condition \eqref{normexplosion} describes how in a sense,  generation $\Lambda_N$ has moved
 very far away from the frequency origin compared to generation $\Lambda_1$. We will use the
 term {\it inner radius} to denote the parameter $\rad$ which we are free to choose as large  as
 we wish.

 \subsection{Second Ingredient:  Diffusion in the Toy Model}

 The second main component of the proof is, in comparison with the first one,  considerably
 more difficult to prove.
Our claim is that we can construct initial data for the Toy Model System \eqref{toymodel}
so that at time zero, $b(0)$ is concentrated in its third
component $b_3$
and this concentration then
propagates from $b_3$ to  $b_4$, then to $b_5$ etc. until at some later time the
solution is concentrated at\footnote{One could in fact construct solutions that  diffuse all
the way from $b_1$ to $b_N$ by a simple modification of
the argument, but to avoid some (very minor) technical issues we shall only go from
$b_3$ to $b_{N-2}$.} $b_{N-2}$.
We will measure the extent to which the solution is concentrated with a parameter $\epsilon$.
More precisely,
\begin{proposition}[Second Ingredient:  Diffusion in the Toy Model]  \label{diffusion}
Given $N > 1, \epsilon \ll1$, there
is initial data $b(0) = (b_1(0), b_2(0), \ldots, b_N(0)) \in \C^N$ for \eqref{toymodel} and there is a time $T = T(N, \epsilon)$ so that
\begin{align*}
|b_3(0)| \geq 1 - \epsilon & \quad \quad |b_j(0)| \leq \epsilon \quad  j \neq 3, \\
|b_{N-2}(T)| \geq 1 - \epsilon & \quad \quad |b_j(T)| \leq \epsilon \quad j \neq N-2.
\end{align*}
In addition, the corresponding solution satisfies $\|b(t)\|_{l^\infty} \sim 1$ for all $0\leq t \leq T$.
\end{proposition}
This Proposition will be recast in a slightly different language  in Section \ref{multi-hop} as Theorem \ref{arnold}.

\subsection{ Third Ingredient: The Approximation Lemma} \label{approx-sec}

This ingredient concerns the extent to which we  can approximate the system of ODE's
corresponding to the full NLS equation \eqref{FNLS} by other systems of ODE's. The
approximate system we will ultimately study is that coming from removing the non-resonant part of the
Fourier transform of the equation's cubic nonlinearity (i.e. equation \eqref{RFNLS}). 
Here we write the approximation lemma
in a more general form.

\begin{lemma}[Third Ingredient:  Approximation Lemma]\label{approx}
Let $0 < \sigma < 1$ be an absolute constant (all implicit constants in this lemma may depend on $\sigma$).  
Let $B \gg 1$, and let $T \ll B^2 \log B$.  Let
\begin{align*}
g(t) & := \left\{g_n(t)\right\}_{n \in \Z^2}
\end{align*}
 be a solution to the equation
\begin{equation}\label{nls-error}
-i \partial_t g(t) =   (\lN (t) (g(t),g(t),g(t)) ) +\lE (t)
\end{equation}
for times $0 \leq t \leq T$, where $\lN (t)$ is defined in \eqref{Ndefined}, \eqref{omegafour} and where
the initial data $g(0)$ is compactly supported.  Assume also that the solution $g(t)$ and the error term $\lE(t)$ obey the  bounds of the form
\begin{align}
\label{csmall}
 \| g(t) \|_{l^1(\Z^2)} & \lesssim B^{-1} \\
\label{integrated}
 \left\| \int_0^t \lE(s)\ ds \right\|_{l^1(\Z^2)} &\lesssim B^{-1-\sigma}
\end{align}
for all $0 \leq t \leq T$.

We conclude that  if $a(t)$ denotes the solution to $\mathcal{F}NLS$ \eqref{FNLS} with initial data $a(0) = g(0)$, then we have
\begin{align} \| a(t) - g(t) \|_{l^1(\Z^2)} & \lesssim B^{-1-\sigma/2} \label{close}
\end{align}
for all $0 \leq t \leq T$.
\end{lemma}




\begin{proof}[Proof of Lemma \ref{approx}]
  First note that since $a(0) = g(0)$ is assumed to be compactly supported, the solution $a(t)$ to \eqref{FNLS} exists globally in time, is smooth with respect to time, and is in $l^1(\Z^2)$ in space.

Write
$$ F(t) := -i\int_0^t \lE(s)\ ds, \, \, \mbox{ and }\, \,
 d(t) := g(t) + F(t).$$
Observe that
$$ - i d_t = \lN(d-F,d-F,d-F),$$
where we have suppressed the explicit $t$ dependence for brevity.  Observe that $g = O_{l^1}(B^{-1})$ and $F = O_{l^1}(B^{-1-\sigma})$, where we use $O_{l^1}(X)$ to denote any quantity with an $l^1(\Z^2)$ norm of $O(X)$.  In particular we have $d = O_{l^1}(B^{-1})$.  By trilinearity and \eqref{triest} we thus have
$$ - i d_t = \lN(d,d,d) + O_{l^1}(B^{-3-\sigma}).$$
Now write $a := d + e$.  Then we have
$$ - i (d+e)_t = \lN(d+e,d+e,d+e),$$
which when subtracted from the previous equation gives (after more
trilinearity and \eqref{triest})
$$ i e_t = O_{l^1}(B^{-3-\sigma}) + O_{l^1}( B^{-2} \| e \|_1 ) + O_{l^1}( \| e \|_1^3),$$
and so by the differential form of Minkowski's inequality, we have
$$ \partial_t \| e \|_1 \lesssim B^{-3-\sigma} + B^{-2} \|e\|_1 + \|e\|_1^3.$$
If we assume temporarily (i.e. as a bootstrap assumption)  that $\|e\|_1 = O(B^{-1})$ for all $t \in [0,T]$, then one can absorb the third term on the right-hand side in the second.  Gronwall's inequality then gives
$$ \|e\|_1 = B^{-1 - \sigma} \exp( C B^{-2} t )$$
for all $t \in [0,T]$.  Since we have $T \ll B^2 \log B$, we thus have 
$\|e\|_1 \ll B^{-1-\sigma/2}$, and so we can remove the a priori hypothesis 
$\|e\|_1 = O(B^{-1})$ by continuity arguments, and conclude the proof.

\end{proof}


Lemma \ref{approx} gives us an approximation on a time interval
 of approximate length $B^2 \log B$, a factor $\log B$ larger than the
 interval $[0,B^2]$ for which the solution is controlled by a straightforward local-in-time
 argument.  The exponent $\sigma/2$ can be in fact replaced by
any exponent between 0 and $\sigma$, but we choose $\sigma/2$ for concreteness.

\subsection{The Scaling Argument and  the Proof of Theorem \ref{main}.}
\label{scalingargument}

Finally we present the relatively simple scaling argument that glues  the three main components
together to get Theorem \ref{main}.

Given $\delta, K$, construct $\Lambda$ as in Proposition \ref{frequencies}.
(This is done in Section
\ref{NumberTheory} below.)   Note that we
are free to specify $\rad$ (which measures the inner radius of the frequencies involved in $\Lambda$) as large as we wish and this will be done
shortly, with $\rad = \rad(\delta, K)$.

  With the number $N = N(\delta, K)$ from the construction of
  $\Lambda$ (recall $N$ represents
  the number of generations in the set of frequencies),
  and a number $\epsilon = \epsilon(K, \delta)$ which we will specify shortly, we
construct a traveling wave solution $b(t)$ to the toy model concentrated at scale $\epsilon$ according to 
Proposition \ref{diffusion} above.  This proposition also gives us a time $T_0 = T_0(K, \delta)$ at which 
the wave has traversed the $N$ generations of frequencies.    Note that the toy model
has the following scaling,
\begin{align*}
b^{(\lambda)}(t) & := \lambda^{-1} b(\frac{t}{\lambda^2}).
\end{align*}
We choose the initial data for \eqref{nls} by setting
\begin{align}
a_n(0)& =b_j^{(\lambda)}(0) \, \, \mbox{for all } n \in \Lambda_j, \label{datachosenhere}
\end{align}
and $a_n(t) = 0$ when $n \notin \Lambda$.  We specify first the scaling parameter
$\lambda$ and then the parameter $\rad$
according to the following considerations which we detail below  after presenting them
now in only the roughest form.  The parameter $\lambda$ is chosen
large enough to ensure the Approximation Lemma \ref{approx} applies, with $c_n$ the solution
of the resonant system of O.D.E.'s \eqref{RFNLS} also evolving from the
data \eqref{datachosenhere}, over the time interval $[0, \lambda^2 T_0]$ - which is the time
the rescaled solution $b^{(\lambda)}$ takes to travel through all the generations in
$\Lambda$.  In other words, we want to apply the Approximation Lemma \ref{approx}
with a parameter $B$ chosen large enough so that,
\begin{align}
\label{need_B_to_satisfy}
B^2 \log B & \gg \lambda^2 T_0.
\end{align}
With $\lambda$ and $B$ so chosen, we will be able to prove that $\|u(t)\|_{H^s(\T^2)}$ grows  by a factor
of $\frac{K}{\delta}$ on $[0, \lambda^2T_0]$.  We finally choose $\rad$ to ensure this quantity
starts at size approximately $\delta$, rather than a much smaller scale.

We detail now these general remarks.  The aim is to apply Lemma \ref{approx} with
$c(t) = \{c_n(t)\}_{n \in \Z^2}$ defined by,
\begin{align*}
c_n(t) & = b_j^{(\lambda)}(t),
\end{align*}
for $n \in \Lambda_j$, and $c_n(t) = 0$ when $n \notin \Lambda$.  Hence, we set
$\lE (t)$ to be the non-resonant part of the nonlinearity on the right hand side of 
\eqref{FNLS}.  That is, 
\begin{align}
\label{pickE}
\lE (t) & : = - \sum_{[\Gamma(n) \backslash \Gamma_{res} (n) ] \cap \Lambda^3}  c_{n_1} {\overline{c_{n_2}}} c_{n_3} 
e^{i \omega_4 t}
\end{align}
where $\omega_4$ is as in \eqref{omegafour}.  (We include the set $\Lambda$ in the description of the 
sum above to emphasize once more that the frequency support of $c(t)$
is always in this set.)
We choose $B = C(N) \lambda$ and then show that for large enough
$\lambda$ the required conditions \eqref{csmall} and
\eqref{integrated} hold true. Observe that \eqref{need_B_to_satisfy}
  holds true with this choice for large enough $\lambda$.
Note first that simply by considering
its support, the fact that $|\Lambda| = C(N)$, and the fact that $\|b(t)\|_{l^\infty} \sim 1$,
we can be sure that,
$\|b(t)\|_{l^1(\Z)} \sim C(N)$ and therefore
\begin{align}
\|b^{(\lambda)}(t)\|_{l^1(\Z)}, \|c(t) \|_{l^1(\Z^2)}  & \leq \lambda^{-1} C(N). \label{bounds}
\end{align}
Thus, $\eqref{csmall}$ holds with the choice $B = C(N) \lambda$.  
For the  second condition \eqref{integrated}, 
we claim
\begin{equation}
\label{secondcondition}
\left \|\int_0^t  \lE(s) \,ds\right\|_{l^1}
\lesssim C(N) (\lambda^{-3} + \lambda^{-5} T).
\end{equation}
Note that this is sufficient with our choices $B = \lambda \cdot C(N)$ and $T = \lambda^2 T_0$.  
It remains only to show \eqref{secondcondition}, but this follows quickly from the van der Corput lemma:  
since $\omega_4$ does not 
vanish in the set $\Gamma(n) \backslash
  \Gamma_{res} (n)$, we can replace $e^{i \omega_4 s}$ by
  $\frac{d}{ds} [ \frac{e^{i \omega_4 s}}{\omega_4}]$ and then
  integrate by parts. Three terms arise: the boundary terms at $s=0,T$ and
  the integral term involving $\frac{d}{ds}[ c_{n_1} (s)
  {\overline{c_{n_2}}} (s)  c_{n_3} (s)]$. For the boundary terms, 
 we use \eqref{bounds} to obtain an upper bound of $C(N) \lambda^{-3}$. For the
  integral term, the $s$ derivative falls on one of the $b$
  factors. We replace this differentiated term using the equation to get  
  an expression that is 5-linear in $c$ and bounded by $C(N) \lambda^{-5} T$.

Once $\lambda$ has been  chosen as above,
we choose $\rad$ sufficiently large so that the initial data $c(0)=a(0)$ has the right size:
\begin{align}
\left(\sum_{n \in \Lambda} |c_n(0)|^2 |n|^{2s} \right)^{\frac{1}{2}} & \sim \delta. \label{datasize}
\end{align}
This is possible since the quantity on the left scales like $\lambda^{-1}$ in $\lambda$, and
$\rad^s$ in the parameter $\rad$.
(The issue here is that our choice of frequencies $\Lambda$ only gives us a large factor (that is, $\frac{K}{\delta}$) by which
the Sobolev norm of the solution will grow.  If our data is much smaller than $\delta$ in size, the solution's Sobolev norm  will not grow to be larger than $K$.)

It remains to show that we can guarantee,
\begin{align}
\label{toshow}
\left( \sum_{n \in \Z^2} |a_n(\lambda^2 T_0)|^2 |n|^{2s} \right)^{\frac{1}{2}} & \geq K,
\end{align}
where $a(t)$ is the evolution of the data $c(0)$ under the full system \eqref{FNLS}.
We do this by first establishing,
\begin{align}
\label{showfirst}
\left( \sum_{n \in \Lambda} |c_n(\lambda^2 T_0)|^2 |n|^{2s}\right)^{\frac{1}{2}} & \gtrsim K,
\end{align}
and second that,
\begin{align}
\label{showsecond}
\sum_{n \in \Lambda} |c_n(\lambda^2T_0) - a_n(\lambda^2T_0)|^2 |n|^{2s}) & \lesssim 1.
\end{align}
As for \eqref{showfirst}, consider the ratio of this norm of the resonant evolution at time $\lambda^2 T_0$ to the same norm at time 0,
\begin{align*}
{\mathcal{Q}} & := \frac{\sum_{n\in \Z^2} |c_n(\lambda^2 T_0)|^2 |n|^{2s}}
{\sum_{n\in \Z^2} |c_n(0)|^2 |n|^{2s}} \\
& = \frac{ \sum_{i = 1}^N \sum_{n \in \Lambda_i} |b_i^{(\lambda)}(\lambda^2T_0)|^2 |n|^{2s}}
 { \sum_{i = 1}^N \sum_{n \in \Lambda_i} |b_i^{(\lambda)}(0)|^2 |n|^{2s}},
\end{align*}
since $c_n := 0$ when $n \notin \Lambda$.
We use now the notation $S_j := \sum_{n \in \Lambda_j} |n|^{2s}$,
\begin{align*}
{\mathcal{Q}} & = \frac{\sum_{i=1}^N |b_i^{(\lambda)}(\lambda^2T_0)|^2 S_i}
{\sum_{i=1}^N |b_i^{(\lambda)}(0)|^2 S_i} \\
& \gtrsim  \frac{S_{N-2} (1 - \epsilon)}{\epsilon S_1 + \epsilon S_2 +(1 - \epsilon)S_3
+ \epsilon S_4 + \cdots + \epsilon S_N} \\
& = \frac{S_{N-2} (1 - \epsilon)}{S_{N-2} \cdot [\epsilon \frac{S_1}{S_{N-2}} + \epsilon \frac{S_2}{S_{N-2}}
 + (1-\epsilon)\frac{S_3}{S_{N-2}} + \epsilon \frac{S_4}{S_{N-2}} + \cdots + \epsilon + \epsilon
 \frac{S_{N-1}}{S_{N-2}} + \epsilon \frac{S_N}{S_{N-2}}]} \\
& =  \frac{(1-\epsilon) }{(1 - \epsilon)\frac{S_3}{S_{N-2}} + O(\epsilon)}\\
& \gtrsim \frac{K^2}{\delta^2},
\end{align*}
where the last inequality is ensured by Proposition \ref{frequencies} and by choosing $\epsilon \lesssim C(N,K,\delta)$ sufficiently small.

As for the second inequality \eqref{showsecond}, using Approximation 
Lemma \ref{approx} we obtain that
\begin{align}
\label{almostdone}
\sum_{n \in \Lambda} |c_n(\lambda^2T_0) - a_n(\lambda^2T_0)|^2 |n|^{2s}) & \lesssim
\lambda^{-1 - \sigma} \left( \sum_{n \in \Lambda} |n|^{2s} \right)^{\frac{1}{2}}  \leq \frac{1}{2},
\end{align}
by possibly  increasing $\lambda$ and $\rad$, maintaining \eqref{datasize}.
Together, the inequalities \eqref{almostdone}, \eqref{close} give us immediately \eqref{showsecond}.


\section{Diffusion in the Toy Model}
\label{multi-hop}


In this section we prove Proposition \ref{diffusion} above, which claims a particular sort of instability
for the system which we call the Toy Model System,
\begin{equation}\label{eq:b-eq}
\partial_t b_j = -i|b_j|^2 b_j + 2i \overline{b_j} (b_{j-1}^2 + b_{j+1}^2).
\end{equation}
This system was derived in the discussion preceding equation \eqref{toymodel} above.
We write $b(t)$ for the vector $(b_j)_{j \in \Z}$ and begin with some general remarks about
the system \eqref{eq:b-eq}.


Note that  \eqref{eq:b-eq} is globally well-posed in $l^2(\Z)$.
To see local well-posedness, observe that the system is of the form
$\partial_t b = T(b,\overline{b},b)$ where $T: l^2 \times l^2 \times l^2 \to l^2$ is a
trilinear form which is bounded on $l^2$.  Local well-posedness then follows from the
Picard existence theorem.  The time of existence obtained by the Picard
theorem depends on the $l^2$ norm of $b$, but one quickly observes that this quantity is conserved.
Indeed, we have
\begin{equation}\label{eq:l2-current}
\begin{split}
 \partial_t |b_j|^2 &= 2 \Re \overline{b_j} \partial_t b_j\\
&= 4\Re i \left[ \overline{b_j}^2 b_{j-1}^2 + \overline{b_j}^2 b_{j+1}^2\right] \\
&= 4 \Im (b_j^2 \overline{b_{j-1}}^2 - b_{j+1}^2 \overline{b_j}^2)
\end{split}
\end{equation}
and hence by telescoping series we obtain the mass conservation law $\partial_t \sum_j |b_j|^2 = 0$.
Note that all these formal computations are justified if $b$ is in $l^2$, thanks
to the inclusion $l^2(\Z) \subset l^4(\Z)$.

Though our analysis won't use it explictly, we note next  that \eqref{eq:b-eq} also enjoys the
conservation of the Hamiltonian
$$ H(b) := \sum_j \frac{1}{4} |b_j|^4 - \Re( \overline{b_j}^2 b_{j-1}^2 ).$$
Indeed, if one rewrites $H(b)$ algebraically in terms of $b$ and $\overline{b}$ as
$$ H(b) = \sum_j \frac{1}{4} b_j^2 \overline{b_j}^2 - \frac{1}{2} \overline{b_j}^2 b_{j-1}^2 - \frac{1}{2} b_j^2 \overline{b_{j-1}}^2$$
then we see from \eqref{eq:b-eq} that
$$ \partial_t b = -2i \frac{\partial H}{\partial \overline{b}};
\quad \partial_t \overline{b} = 2i \frac{\partial H}{\partial b},$$
and thus
$$ \partial_t H(b) = \partial_t b \cdot \frac{\partial H}{\partial b} + \partial_t \overline{b} \cdot \frac{\partial H}{\partial \overline{b}} = 0.$$
Again, it is easily verified that these formal computations can be justified in $l^2(\Z)$.

The system \eqref{eq:b-eq} enjoys a number of symmetries.  Firstly, there is phase invariance
$$ b_j(t) \leftarrow e^{i\theta} b_j(t)$$
for any angle $\theta$; this symmetry corresponds to the $l^2$ conservation.
There is also scaling symmetry
$$ b_j(t) \leftarrow \lambda b(\lambda^2 t)$$
for any $\lambda > 0$, time translation symmetry
$$ b_j(t) \leftarrow b_j(t - \tau)$$
for any $\tau \in \R$ (corresponding to Hamiltonian conservation, of course), and space translation symmetry
$$ b_j(t) \leftarrow b_{j - j_0}(t)$$
for any $j_0 \in \Z$.  Finally, there is time reversal symmetry
$$ b_j(t) \leftarrow \overline{b_j}(-t)$$
and space reflection symmetry
$$ b_j(t) \leftarrow b_{-j}(t).$$
There is also a sign symmetry
$$ b_j(t) \leftarrow \epsilon_j b_j(t)$$
where for each $j$, $\epsilon_j = \pm 1$ is an arbitrary sign.

Next we show that the infinite system \eqref{eq:b-eq} evolving  from $l^2$ data reduces to a
finite system if the data
is supported on only a finite number of modes $b_i$.
Indeed \eqref{eq:l2-current} gives
the crude differential inequality
\begin{align*}
 \partial_t |b_j|^2 &= 2 \Re( \overline{b_j} \partial_t b_j)\\
&= 4 \Re(i \overline{b_j}^2 (b_{j-1}^2 + b_{j+1}^2)) \\
&\leq 4 |b_j|^2
\end{align*}
for any $j \in \Z$.
From Gronwall's inequality we conclude first that if $b$ is an $l^2(\Z)$ solution to \eqref{eq:b-eq}, then
the support of $b$ is nonincreasing in time, i.e. if $b(t_0)$ is  supported on  $I \subseteq \Z$
at some time $t_0 \in \R$, then
$b(t)$ is supported in $I$ for all time.    The time reversal symmetry
 allows us to conclude that in fact the support of $b$ is constant in time.

In particular, if $I$ is finite, then \eqref{eq:b-eq} collapses to a
finite-dimensional ODE, which is obtained from \eqref{eq:b-eq} by
setting $b_j = 0$ for all $j \not \in I$.

We conclude our general remarks here on the dynamics of \eqref{eq:b-eq}
by observing how this evolution, and the assumptions on $\Lambda$ that went into its derivation,  account
for the necessary balance in mass at high and low frequencies
dictated by  the conservation laws \eqref{energyconservation}
and \eqref{massconservation}.  As discussed immediately following Theorem \ref{main}, this balance - i.e. the presence
of both a forward and inverse cascade of mass - is an important complication to constructing any solutions to
\eqref{nls} that carry energy at higher and higher frequencies.
Recall that the Toy Model comes from imposing a host of assumptions on the initial data for the resonant truncation
system \eqref{RFNLS}, which also has conserved mass, energy, and momentum.  Under the assumptions that led to the
Toy Model System, the conserved mass becomes
essentially{\footnote{We are systematically ignoring a harmless multiplicative constant.}}
\begin{equation}
\label{mass} \sum_j |b_j(t)|^2,
\end{equation}
the conserved momentum essentially becomes
\begin{equation}
\label{momentum} \sum_j |b_j(t)|^2 (\sum_{n \in \Lambda_j} n),
\end{equation}
and the conserved energy now essentially becomes
\begin{equation}
\label{energy} \sum_j |b_j(t)|^2 (\sum_{n \in \Lambda_j} |n|^2) + \frac{1}{2} \sum_j |b_j(t)|^4
+ \sum_j |b_j(t)|^2 |b_{j+1}(t)|^2.
\end{equation}
It looks like these three quantities are quite different.  However, one observes the identities
$$ \sum_{n \in \Lambda_j} n = \sum_{n \in \Lambda_{j+1}} n$$
and
$$ \sum_{n \in \Lambda_j} |n|^2 = \sum_{n \in \Lambda_{j+1}} |n|^2$$
for all $1 \leq j < N$.  This simply reflects the fact that in any nuclear family $(n_1,n_2,n_3,n_4)$, one
has $n_1 + n_3 = n_2 + n_4$
and $|n_1|^2 + |n_3|^2 = |n_2|^2 + |n_4|^2$.  Thus the conservation of momentum follows trivially from the conservation of mass,
and the conservation of energy is now equivalent to the conservation of the quartic expression
$$ \frac{1}{2} \sum_j |b_j(t)|^4 + \sum_j |b_j(t)|^2 |b_{j+1}(t)|^2.$$

We turn now to the more specific, unstable  behavior in \eqref{eq:b-eq} which we aim to demonstrate.
Let  $N \geq 1$ be a fixed integer and  let $\Sigma \subset \C^N$ be the unit-mass sphere
$$ \Sigma := \{ x \in \C^N: |x|^2 = 1 \}.$$  Assigning initial data $b(t_0)$ with $b_j(t_0) = 0$ for
 $j \leq 0, j > N$, \eqref{eq:b-eq} generates a group $S(t): \Sigma \to \Sigma$ of smooth flows on the smooth $2N-1$-dimensional compact manifold $\Sigma$,
defined by $S(t) b(t_0) := b(t+t_0)$.
We shall refer to elements $x$ of $\Sigma$
as \emph{points}.




Define the circles $\T_1,\ldots,\T_N$ by
$$ \T_j := \{ (b_1,\ldots,b_N) \in \Sigma: |b_j| = 1; \quad b_k = 0 \hbox{ for all } k \neq j \}$$
then it is easy to see from the above arguments that the flows $S(t)$ leave
each of the circles $\T_j$ invariant:
$S(t) \T_j = \T_j$.  Indeed for each $j$ we have the following explicit oscillator solutions
to \eqref{eq:b-eq} that traverse the circle $\T_j$:
\begin{equation}\label{tj-soln}
b_j(t) := e^{-i(t+\theta)}; \quad b_k(t) := 0 \hbox{ for all } k \neq j.
\end{equation}
Here $\theta$ is an arbitrary phase.

The main result in this section  is that there is Arnold
diffusion between any of these two circles, for instance
between the third\footnote{As mentioned previously, one could in fact diffuse all
the way from $\T_1$ to $\T_N$ by a simple modification of
the argument, but to avoid some (very minor) notational
issues near the endpoints we shall only go from
$\T_3$ to $\T_{N-2}$.} circle $\T_3$ and the third-to-last circle
$\T_{N-2}$.

\begin{theorem}[Arnold diffusion for \eqref{eq:b-eq}]\label{arnold}  Let $N \geq 6$.  Given any $\eps > 0$, there exists a point $x_3$
within $\eps$ of $\T_3$ (using the usual metric on $\Sigma$), a point $x_{N-2}$ within $\eps$ of $\T_{N-2}$, and a time\footnote{We shall only flow forwards in time here.  However, the time reversal symmetry $b_j(t) \mapsto \overline{b_j}(-t)$ (or the spatial reflection symmetry $b_j(t) \mapsto b_{N+1-j}(t)$)  allows one to obtain analogues of all these results when one evolves backwards in time.  Of course, when doing so, the stable and unstable modes that we describe below shall exchange roles.}
 $t \geq 0$ such that $S(t) x_3 = x_{N-2}$.
\end{theorem}

To state this result more informally, there exist solutions to \eqref{eq:b-eq} of total mass $1$ which are arbitrarily concentrated at the mode $j=3$ at some time, and then arbitrarily concentrated at the mode $j=N-2$ at a later time.

To motivate the theorem let us first observe that when $N=2$ we can easily diffuse from $\T_1$ to $\T_2$.  Indeed in this case we have the explicit ``slider'' solution
\begin{equation}\label{eq:slider}
b_1(t) := \frac{e^{-it}\omega}{\sqrt{1 + e^{2\sqrt{3}t}}}; \quad
b_2(t) := \frac{e^{-it} \omega^2}{\sqrt{1 + e^{-2\sqrt{3}t}}}
\end{equation}
where $\omega := e^{2\pi i/3}$ is a cube root of unity.
This solution approaches $\T_1$ exponentially fast as $t \to -\infty$, and
approaches $\T_2$ exponentially fast as $t \to +\infty$.  One can translate
this solution in the $j$ parameter, and obtain solutions that ``slide''
from $\T_j$ to $\T_{j+1}$ (or back from $\T_{j+1}$ to $\T_j$, if we reverse
time or apply a spatial reflection). This for instance validates the
$N=6$ case of Theorem \ref{arnold}.  Intuitively, the proof of Theorem
\ref{arnold} for higher $N$ should then proceed by concatenating these
slider solutions.  Of course, this does not work directly,
because each solution requires an infinite amount of time to connect
one circle to the next, but it turns out that a suitably perturbed or
``fuzzy'' version of these slider solutions can in fact be glued together.

\begin{figure}[htp]
\centering
\includegraphics[width=12cm]{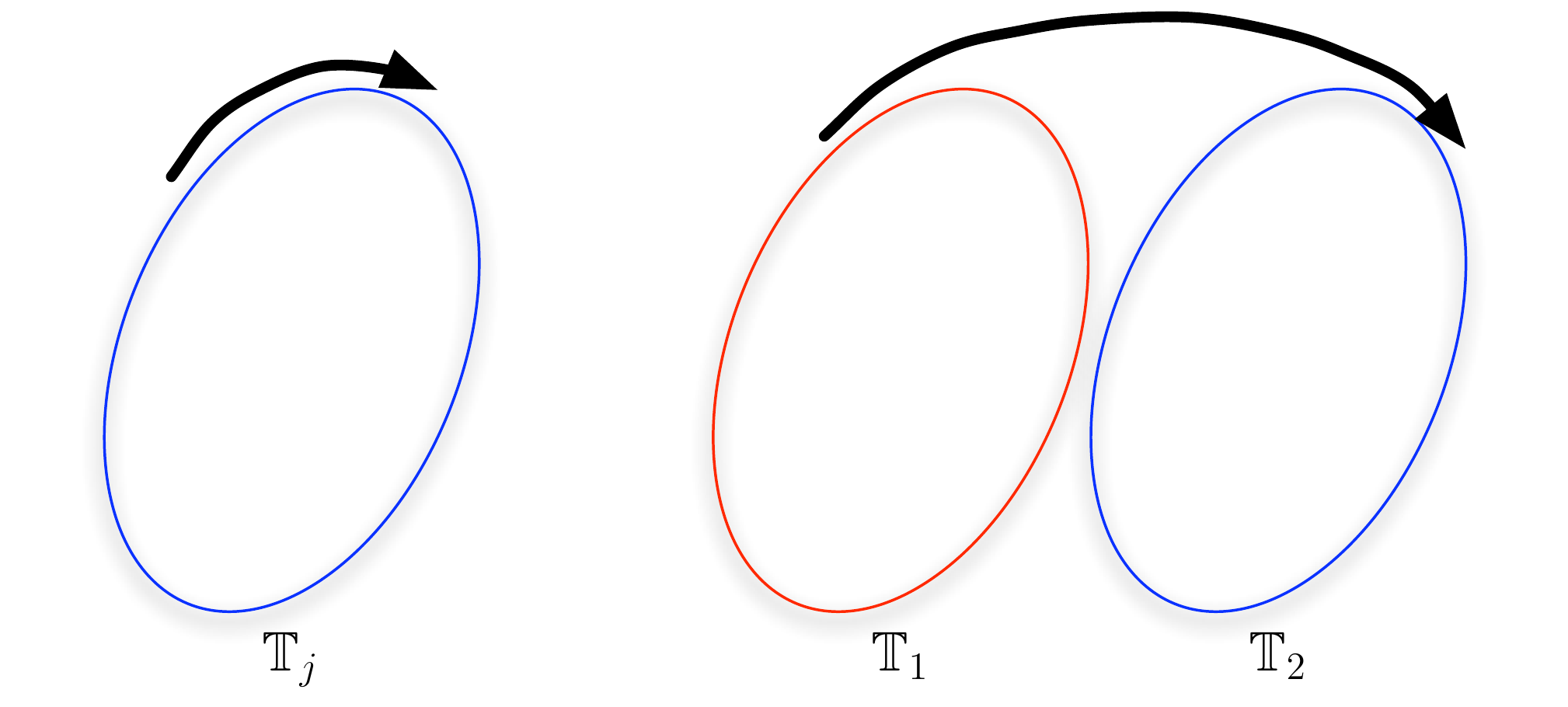}
\caption{Explicit oscillator solution around $\T_1$ and the slider solution from $\T_1$ to $\T_2$}\label{fig:ExplicitFamiles}
\end{figure}

\subsection{General Notation}

We shall use the usual $O()$ notation, but allow the constants in that notation to depend on $N$.  Hence $X = O(Y)$
 denotes $X \leq C Y$ for some constant $C$ that depends on $N$, but is otherwise
 universal.  We use $X \lesssim Y$ as a synonym for $X = O(Y)$.

We will also use \emph{schematic notation}, so that an expression such as $\bigO( f g h )$ will mean some linear combination of
quantities which \emph{resemble} $fgh$ up to the presence of constants and complex conjugation of the terms. Thus for instance $3f \overline{g} h + 2 \overline{fg} h - i fgh$
qualifies to be of the form $\bigO( fgh )$, and $|b|^2 b$ qualifies to be of the form $\bigO( b^3 )$.  We will extend this notation to the case when $f,g,h$ are vector-valued, and allow the linear combination to depend on $N$. Thus for instance if $c = (c_1,\ldots,c_N)$ then $\sum_{j=1}^{N-1} c_j \overline{c_{j+1}}$ would qualify to be of the form $\bigO( c^2 )$.

\subsection{Abstract Overview of Argument}

\begin{quotation}
\emph{Off the expressway, over the river, off the billboard, through the window, off the wall, nothin' but net.}  Michael Jordan \cite{JordanSuperBowl}
\end{quotation}

To prove Theorem \ref{arnold}, one has to engineer initial conditions near $\T_3$ that can ``hit'' the target $\T_{N-2}$ (or more precisely a small neighborhood of $\T_{N-2}$) after some long period of time.  This is difficult to do directly.  Instead, what we shall do is create a number of intermediate ``targets'' between $\T_3$ and $\T_{N-2}$, and show that in a certain sense one can hit any point on each target (to some specified accuracy) by some point on the previous target (allowing for some specified uncertainty in one's ``aim'').  These intermediate trajectories can then be chained together to achieve the original goal.  To describe the strategy more precisely, it is useful to set up some abstract notation.

\begin{definition}[Targets] A \emph{target} is a triple $(M, d, R)$, where $M$ is a subset of $\Sigma$, $d$
is a semi-metric\footnote{A semi-metric is the same as a metric, except that $d(x,y)$ is allowed to degenerate to
zero even when $x \neq y$. The reason we need to deal with semi-metrics is because of the phase symmetry
$x \mapsto e^{i\theta} x$ on $\Sigma$; one could quotient this out and then deal with nondegenerate
metrics if desired.} on $\Sigma$, and $R > 0$ is a radius.  We say that a point $x \in \Sigma$ is
\emph{within} a target $(M,d,R)$ if we have $d(x,y) < R$ for some $y \in M$.  Given two
points $x,y \in \Sigma$, we say that $x$ \emph{hits} $y$, and write $x \mapsto y$, if we have $y = S(t) x$ for some
$t \geq 0$.  Given an initial target $(M_1,d_1,R_1)$ and a final target $(M_2,d_2,R_2)$, we say that $(M_1,d_1,R_1)$
\emph{can cover} $(M_2,d_2,R_2)$, and write $(M_1,d_1,R_1) \covers (M_2,d_2,R_2)$, if for every
$x_2 \in M_2$ there exists an $x_1 \in M_1$, such that for any point $y_1 \in \Sigma$ with $d(x_1,y_1) < R_1$ there
exists a point $y_2 \in \Sigma$ with $d_2(x_2,y_2) < R_2$ such that $y_1$ hits $y_2$.
\end{definition}

\begin{remark} One could eliminate the radius parameter $R$ here by replacing each target $(M, d, R)$ with the equivalent target $(M,d/R,1)$, but this seems to make the concept slightly less intuitive (it is like replacing all metric balls $B_d(x,R)$ with unit balls $B_{d/R}(x,1)$ in a renormalized metric).
\end{remark}

\begin{remark}\label{shooting}
The notion of covering may seem somewhat complicated (involving no less than five quantifiers!), but it can be summarized as follows.  The assertion $(M_1,d_1,R_1)\covers
(M_2,d_2,R_2)$ is a guarantee that one can hit any point in the final target $M_2$ that one desires - within an accuracy of $R_2$ in the $d_2$ metric - by ``aiming'' at some well-chosen point in the initial target $M_1$ and then evolving by the flow, even if one's ``aim'' is a little uncertain (by an uncertainty of $R_1$ in the $d_1$ metric).
Thus the concept of covering is simultaneously a notion of surjectivity
(that the flowout of $M_1$ contains $M_2$ in some approximate sense) and a notion of
stability (that small perturbations of the initial state lead to small perturbations
in the final state). See also Figure \ref{fig:m1_covers_m2}.

\end{remark}

\begin{figure}[htp]
\centering
\includegraphics[width=12cm]{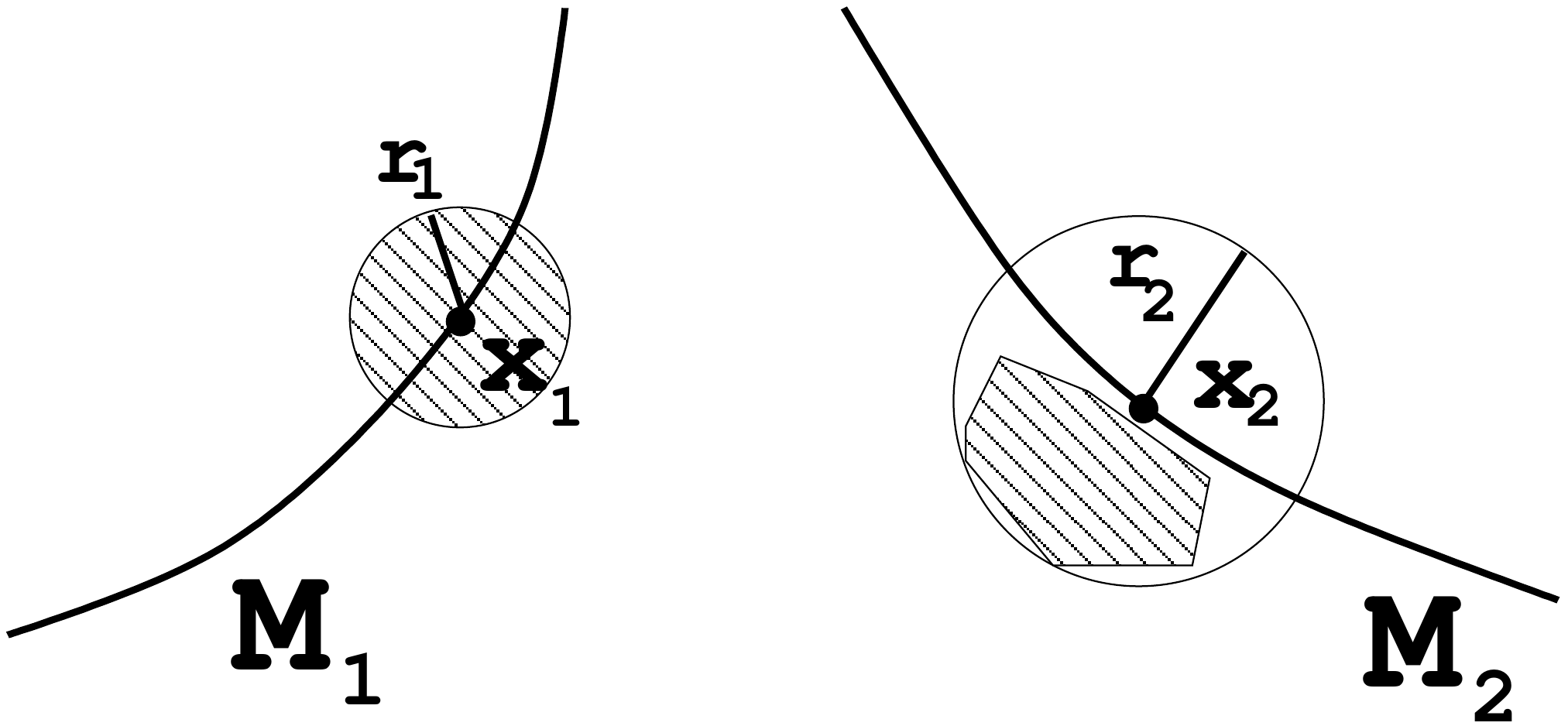}
\caption{$M_1 \covers M_2$. The shaded disk of radius $r_1$
  around $x_1$ flows out to the shaded area near $M_2$ inside a disk of
radius $r_2$ around $x_2$.}\label{fig:m1_covers_m2}
\end{figure}


Informally, one can think of a target $(M,d,R)$ as a ``fuzzy'' set $\{ x + O_d(R): x \in M \}$, where $O_d(R)$ represents some ``uncertainty'' of extent $R$ as measured in the semi-metric $d$.  Thus, for instance, if we were working in $\R^2$ with the usual metric $d$, and $M$ was a rectangle
$\{ (x_1,x_2): |x_1| \leq r_1, |x_2| \leq r_2 \}$ then one might describe the target $(M,d,R)$ somewhat schematically as
$$ \{ ( O(r_1) + O(R), O(r_2) + O(R)) \}$$
where the first term in each coordinate represents the parameters of the set (basically, they describe the points that one can ``aim'' at), and
the second term in each coordinate represents the uncertainty of the set (this describes the inevitable error that causes the actual location of
the solution to deviate from the point that one intended to hit).  This schematic notation may be somewhat confusing and we shall reserve it for informal discussions only.

One of the most important features of the covering concept for us is its transitivity.

\begin{lemma}[Transitivity]\label{trans} If $(M_1,d_1, R_1)$, $(M_2,d_2, R_2)$,
$(M_3,d_3, R_3)$ are targets such that $(M_1, d_1, R_1) \covers
(M_2,d_2, R_2)$ and $(M_2,d_2, R_2)  \covers(M_3,d_3, R_3)$,
then $(M_1,d_1, R_1)  \covers (M_3,d_3, R_3)$.
\end{lemma}

\begin{figure}[htp]
\centering
\includegraphics[height=8cm]{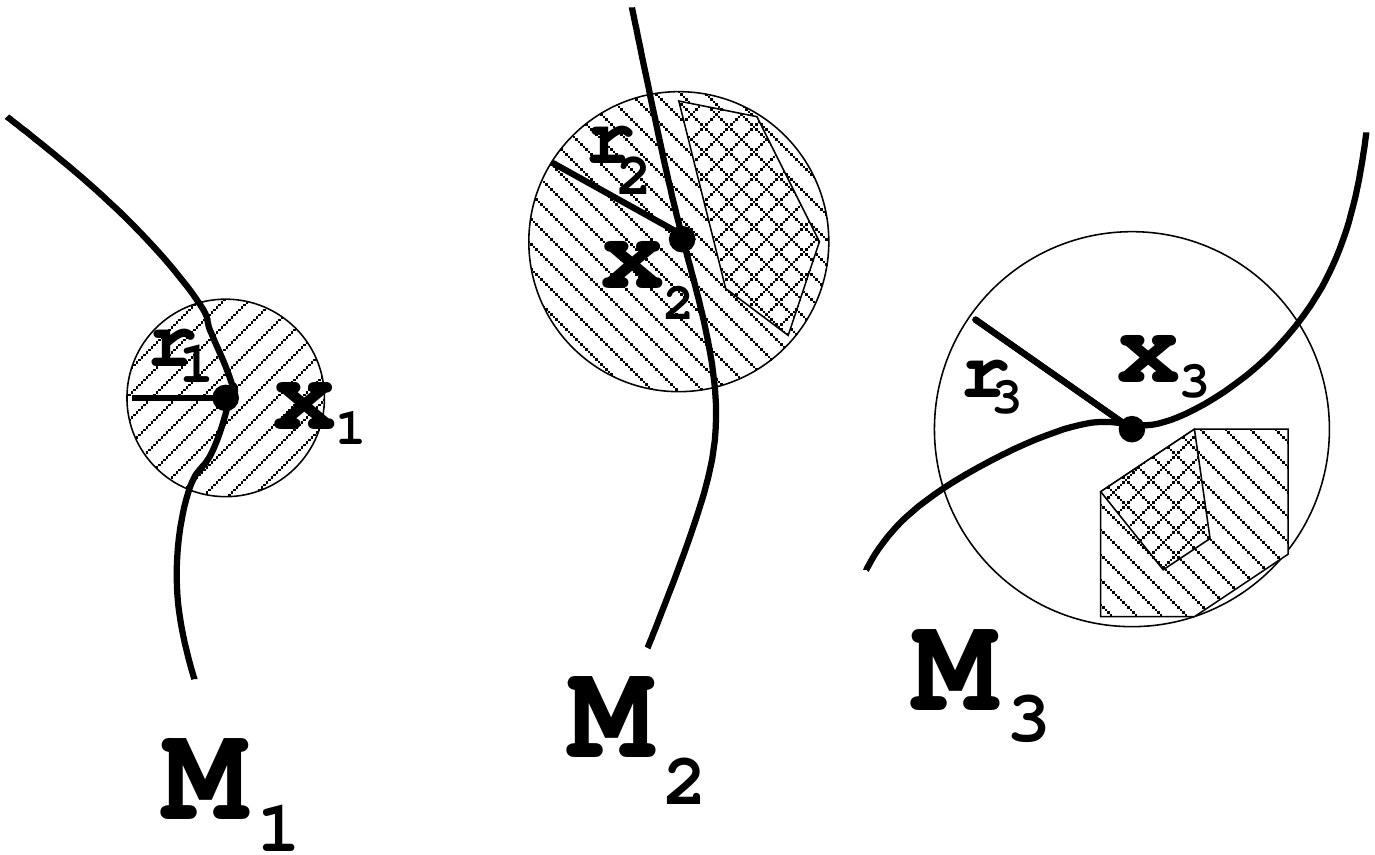}
\caption{$M_1 \covers M_2$, $M_2 \covers M_3$ imply
  $M_1 \covers M_3$. The shaded disk of radius $r_1$ centered $x_1$ flows out
  onto the double shaded portion of the disk of radius $r_2$ centered
  at $x_2$. The disk of radius $r_2$
  centered at $x_2$ maps into the disk of radius $r_3$ near $M_3$. }\label{fig:transitivity}
\end{figure}

\begin{proof} Let $x_3$ be any point in $M_3$.  Since $(M_2,d_2, R_2)  \covers (M_3,d_3, R_3)$, we can find $x_2 \in M_2$ such that for every $y_2$
with $d_2(x_2,y_2) < R_2$, there exists $y_3$ with $d_3(x_3,y_3) < R_3$ such
that $y_2 \mapsto y_3$.  Since $(M_1, d_1, R_1)  \covers (M_2,d_2, R_2)$,
we can find an $x_1 \in M_1$ such that for every $y_1$ with $d_1(x_1,y_1) < R_1$, we can find $y_2$ with $d_2(x_2,y_2) < R_2$ such that $y_1 \mapsto y_2$.  Putting these two together, we thus see that for every $y_1$ with $d_1(x_1,y_1) < R_1$, we can find $y_3$ with $d_2(x_3,y_3) < R_3$
such that $y_1 \mapsto y_2 \mapsto y_3$, which by the group properties of the flow imply that $y_1 \mapsto y_3$, and the claim follows.
\end{proof}

The reader is invited to see how the transitivity of the covering relation follows
intuitively from the interpretation of the concept
given in Remark \ref{shooting}.  See also figure \ref{fig:transitivity}.


We can now outline the idea of the proof.  For each $j=3,\ldots,N-2$, we will define three targets which lie fairly close to $\T_j$, namely
\begin{itemize}
\item An \emph{incoming target} $(M^-_j,d^-_j, R^-_j)$ (located near the stable manifold of $\T_j$),
\item A \emph{ricochet target} $(M^0_j, d^0_j,R^0_j)$ (located very near $\T_j$ itself), and
\item An \emph{outgoing target} $(M^+_j, d^+_j,R^+_j)$ (located near the unstable manifold of $\T_j$).
\end{itemize}


We will then prove the covering relations
\begin{align}
(M^-_j, d^-_j,R^-_j) & \covers (M^0_j, d^0_j,R^0_j) \hbox{ for all } 3 < j \leq N-2 \label{incoming} \\
(M^0_j, d^0_j,R^0_j) & \covers (M^+_j, d^+_j,R^+_j) \hbox{ for all } 3 \leq j < N-2 \label{outgoing} \\
(M^+_j, d^+_j,R^+_j) & \covers (M^-_{j+1}, d^-_{j+1}, R^-_{j+1}) \hbox{ for all } 3 \leq j < N-2 \label{transit}
\end{align}
which by many applications of Lemma \ref{trans} implies covering relation
\begin{equation}\label{long-chain}
(M^0_3, d^0_3,R^0_3)  \covers (M^0_{N-2}, d^0_{N-2}, R^0_{N-2}).
\end{equation}


As we shall construct $(M^0_3,d^0_3,R^0_3)$ to be close to $\T_3$ and $(M^0_{N-2},d^0_{N-2},R^0_{N-2})$ to be close to $\T_{N-2}$, Theorem \ref{arnold}
will follow very quickly\footnote{Readers familiar with the original paper of Arnold \cite{arnold} will see
strong parallels here; the notion of one set ``obstructing'' another in that paper is 
analogous to the notion of ``covering'' here, the targets $(M^-_j,d^-_j,R^-_j)$ are analogous
 to ``incoming whiskers'' for the torus $\T_j$, and the targets $(M^+_j,d^+_j,r^+_j)$ are 
 ``outgoing whiskers''.  Unfortunately, it seems that we cannot use the machinery from 
 \cite{arnold} directly, mainly because our invariant manifolds have too small 
 a dimension and so do not have the strong ``transversality'' properties required 
 in \cite{arnold}.  This requires a certain ``thickening'' of these sets using the 
 above machinery of covering of targets.} from 
\eqref{long-chain}.

Our exposition is structured as follows.  In
Section \ref{construct-sec}, we shall define the targets
$(M^-_j,d^-_j,R^-_j)$, $(M^0_j, d^0_j,R^0_j)$, $(M^+_j, d^+_j,R^+_j)$, after construction of some useful local coordinates near each circle $\T_j$ in Section \ref{local-sec}, and then in Section \ref{section:long-chain} see why \eqref{long-chain} implies Theorem \ref{arnold}.  In Section \ref{incoming-sec},
we establish the incoming covering estimate \eqref{incoming}, and in Section \ref{outgoing-sec} we establish the (very similar) outgoing covering estimate.  Finally in Section \ref{transit} we establish the (comparatively easy) transitory covering estimate
\eqref{transit} and hence \eqref{long-chain}.

\subsection{Local Coordinates near $\T_j$}\label{local-sec}

Fix $3 \leq j \leq N-2$.  In this section we shall devise a useful local coordinate system around the circle $\T_j$ that will clarify the
dynamics near that circle, and will motivate the choice of targets
$(M^-_j,d^-_j,R^-_j)$, $(M^0_j, d^0_j,R^0_j)$, $(M^+_j, d^+_j,R^+_j)$ involved.

We will assume here that $b_j \neq 0$; as mentioned in the introduction, this constraint is preserved by the flow.  We shall refer to the $b_j$ mode as the \emph{primary} mode, the mode $b_{j-1}$ as the \emph{trailing secondary mode}, the mode $b_{j+1}$ as the \emph{leading secondary mode}, the modes $b_1,\ldots,b_{j-2}$ as the \emph{trailing peripheral modes}, and the modes $b_{j+2},\ldots,b_N$ as the \emph{leading peripheral modes}.
In the vicinity of the circle $\T_j$, it is the primary mode that will have by far the most mass and will thus dominate the evolution.  From \eqref{eq:b-eq} we see that the secondary modes will be linearly forced by the primary mode.  The peripheral modes
will only be influenced by the primary mode indirectly (via its influence on the secondary modes) and their evolution will essentially be trivial.
At this stage there is a symmetry between the leading and trailing modes, but later on we shall break this symmetry when trying to construct
a solution that evolves from the $3$ mode to the $N-3$ mode (requiring one to be far more careful with the leading modes than the trailing ones).

Bearing in mind the phase rotation symmetry $x \mapsto x e^{i\theta}$ we shall select the ansatz
\begin{equation}\label{bctheta}
 b_j = r e^{i\theta}; \quad b_k = c_k e^{i\theta} \hbox{ for } k \neq j
 \end{equation}
where $r$, $\theta$ are real and the $c_k$ are allowed to be complex (again we assume $b_j \neq 0$). In other words, we are
conjugating the secondary and peripheral modes by the phase of the primary mode.
Substituting these equations into \eqref{eq:b-eq} gives
\begin{eqnarray}
&&\partial_t c_{j \pm 1} + i c_{j \pm 1} \partial_t \theta = -i|c_{j \pm 1}|^2 c_{j\pm 1} + 2i r^2
 \overline{c_{j\pm 1}}
+ 2 i c_{j \pm 2}^2 \overline{c_{j \pm 1}} \hbox{ for } \pm 1 = +1,-1 \label{sys-1}\\
&&\partial_t \theta = -r^2 + 2 \Re(c_{j-1}^2 + c_{j+1}^2) \label{sys-2}\\
&&\partial_t c_k + i c_k \partial_t \theta= -i|c_k|^2 c_k + 2i \overline{c_k} (c_{k-1}^2 + c_{k+1}^2)
 \hbox{ for } |j-k| \geq 2. \label{sys-3}
\end{eqnarray}
We have not stated the equation for $r$ (the magnitude of the primary mode)
explicitly, since on $\Sigma$ one can recover $r$ from the other coordinates by the conservation of mass which gives the formula
\begin{equation}\label{rform}
r^2 = 1 - \sum_{k \neq j} |c_k|^2.
\end{equation}
In particular we have the crude estimates
$$ r^2 = 1 - \bigO( c^2); \quad \partial_t \theta = -1 + \bigO( c^2 )$$
where we use \emph{schematic notation}, thus $\bigO(c^2)$ denotes any quadratic combination of the $c_k$ and their conjugates.
We can substitute this into \eqref{sys-3} to obtain an equation for the evolution of the peripheral modes
\begin{equation}\label{ck-eq}
 \partial_t c_k = i c_k + \bigO( c_k c^2 ).
\end{equation}
Now we turn to the secondary modes \eqref{sys-1}.  Fix a sign $\pm$.
From \eqref{rform}, \eqref{sys-2} we have
$$ r^2 = 1 - |c_{j\pm 1}|^2 - \bigO( c_{\neq j\pm 1}^2 ); \quad \partial_t \theta = -1 + |c_{j\pm 1}|^2 + 2 \Re(c_{j\pm 1}^2) + \bigO( c_{\neq j\pm 1}^2 )$$
where we write $c_{\neq j+1} := (c_1,\ldots,c_{j-1},c_{j+2},\ldots,c_N)$ and
$c_{\neq j-1} := (c_1,\ldots,c_{j-2},c_{j+1},\ldots,c_N)$.
Substituting this into \eqref{sys-1} we have
$$ \partial_t c_{j\pm 1} = i c_{j\pm 1} + 2i \overline{c_{j\pm 1}} - 2i |c_{j \pm 1}|^2 c_{j \pm 1} -
2 i \Re(c_{j\pm 1}^2) c_{j\pm 1} - 2i |c_{j\pm 1}|^2 \overline{c_{j\pm 1}} + \bigO( c_{j\pm 1} c_{\neq j\pm 1}^2 ).$$
It is then natural to diagonalize the linear component of this equation by introducing the coordinates
\begin{equation}\label{cj-coord}
c_{j\pm 1} = \omega c_{j \pm 1}^- + \omega^2  c_{j \pm 1}^+.
\end{equation}
where as in \eqref{eq:slider} $\omega := e^{2\pi i/3}$.  One then computes
$$ i c_{j\pm 1} + 2i \overline{c_{j\pm 1}} = - \sqrt{3} \omega c_{j \pm 1}^- + \sqrt{3} \omega^2 c_{j \pm 1}^+$$
and
$$ \Re(c_{j\pm 1}^2) = - \frac{1}{2}|c_{j \pm 1}|^2 + \bigO( c_{j \pm 1}^- c_{j \pm 1}^+ )$$
and thus
\begin{equation}\label{cpm-eq}
 \partial_t c_{j \pm 1} = (1 - |c_{j \pm 1}|^2) ( - \sqrt{3} \omega c_{j \pm 1}^- + \sqrt{3} \omega^2 c_{j \pm 1}^+ )
+ \bigO( c_{j \pm 1} c_{j \pm 1}^- c_{j \pm 1}^+ ) + \bigO( c_{j\pm 1} c_{\neq j\pm 1}^2 ).
\end{equation}
Taking components, we conclude

\begin{proposition}[Local coordinates near $\T_j$]\label{coord}  Let $3 \leq j \leq N-2$, and let $b(t)$ be a solution to \eqref{eq:b-eq} with $b_j(t) \neq 0$.
Define the coordinates $r, \theta$ (primary mode), $c_{j \pm 1}^-, c_{j \pm 1}^+$ (secondary modes), and $c_* := (c_1,\ldots,c_{j-2},c_{j+2},\ldots,c_N)$ (peripheral modes) by \eqref{bctheta}, \eqref{cj-coord}.  Then we have the system of equations
\begin{align}
\partial_t c_{j \pm 1}^- &= -\sqrt{3} c_{j \pm 1}^- + \bigO( c^2 c_{j\pm 1}^- ) + \bigO( c_{j \pm 1}^+ c_{\neq j \pm 1}^2 )\label{ceq-stable}\\
\partial_t c_{j \pm 1}^+ &= \sqrt{3} c_{j \pm 1}^+ + \bigO( c^2 c_{j\pm 1}^+ )  + \bigO( c_{j \pm 1}^- c_{\neq j \pm 1}^2 )\label{ceq-unstable}\\
\partial_t c_* &= i c_* + \bigO( c^2 c_* ).\label{ceq-neither}
\end{align}
Also, the constraint $b_j(t) \neq 0$ is equivalent (via \eqref{rform}) to the condition
\begin{equation}\label{cbound}
|c| < 1.
\end{equation}
\end{proposition}

\begin{remark}
Note the total disappearance of the primary mode coordinates $r, \theta$.  From a symplectic geometry viewpoint, we have effectively taken the \emph{symplectic quotient} of the state space with respect to the rotation symmetry $x \mapsto e^{i\theta} x$.  The elimination of these (very large) coordinates is conducive to analyzing the evolution of the (much smaller) secondary and peripheral modes accurately.
\end{remark}

When one is near the torus $\T_j$ (which in these coordinates is just the origin $c=0$), we expect the cubic terms $\bigO(c^3)$ in
the above proposition to be negligible.
From the equations \eqref{ceq-stable}, \eqref{ceq-unstable}, \eqref{ceq-neither} we conclude in this regime that the two real-valued modes $c_{j \pm 1}^-$ are linearly stable, the two real-valued modes $c_{j \pm 1}^+$ are linearly unstable (growing like $e^{\sqrt{3} t}$), and the remaining modes $c_*$ are oscillatory (behaving like $e^{it}$).  Observe also that most of the nonlinear interaction on these modes resembles a diagonal linear potential of magnitude $\bigO(c^2)$, indicating that the coupling between these modes is relatively weak (especially when $\bigO(c^2)$ is small); however there are some troublesome interactions
arising from the final terms in the right-hand sides of \eqref{ceq-stable}, \eqref{ceq-unstable} that allow the stable modes to influence the unstable
modes and vice versa (via coupling with other modes).  This causes a certain amount of mixing in the evolution which will require some care to handle, and is directly responsible for the rather inelegant appearance of polynomial factors of $T$ (in addition to the more natural exponential factors) in the analysis of later sections.

For the purpose of obtaining upper bounds for the magnitude of the various modes, the following useful lemma captures the stable nature
of the $c_{j \pm 1}^-$, the unstable nature of the $c_{j \pm 1}^+$, and the oscillatory nature of the $c_*$, provided that the solution stays
close to the circle $\T_j$ (which corresponds in these coordinates to $c=0$) in a certain $L^2$ sense.

\begin{lemma}[Upper bounds]\label{ubound}  Suppose that $[0,t]$ is a time interval on which we have the smallness condition
\begin{equation}\label{cl2-bound}
\int_0^t |c(s)|^2\ ds \lesssim 1.
\end{equation}
Then we have the estimates
\begin{align}
|c_{j \pm 1}^-(t)| &\lesssim e^{-\sqrt{3} t} |c_{j \pm 1}^-(0) | +
\int_0^t e^{-\sqrt{3}(t-s)} |c_{j \pm 1}^+(s)| |c_{\neq j \pm 1}^2(s)|\ ds \label{cminus-bound}\\
 |c_{j \pm 1}^+(t)| &\lesssim e^{\sqrt{3} t} | c_{j \pm 1}^+(0) |
+ \int_0^t e^{\sqrt{3}(t-s)} |c_{j \pm 1}^-(s)| |c_{\neq j \pm 1}(s)|^2\ ds \label{cplus-bound}\\
|c_{j \pm 1}(t)| &\lesssim e^{\sqrt{3} t} |c_{j \pm 1}(0)| \label{ctotal-bound}\\
|c_*(t)| &\lesssim |c_*(0)| \label{cstar-bound}.
\end{align}
\end{lemma}

\begin{proof} We take absolute values in \eqref{ceq-stable}, \eqref{ceq-unstable}, \eqref{ceq-neither} and obtain the differential inequalities\footnote{As the quantity being differentiated is only Lipschitz rather than smooth, these inequalities should be interpreted in the appropriate weak sense.}
\begin{align*}
\partial_t |e^{\sqrt{3} t} c_{j \pm 1}^-| &\lesssim  |c|^2 |e^{\sqrt{3} t} c_{j \pm 1}^-|  +
e^{\sqrt{3} t} |c_{j \pm 1}^+| |c_{\neq j \pm 1}| ^2\\
\partial_t |e^{-\sqrt{3} t} c_{j \pm 1}^+| &\lesssim  |c|^2 |e^{-\sqrt{3} t} c_{j\pm 1}^+| +
e^{-\sqrt{3} t} |c_{j \pm 1}^-| |c_{\neq j \pm 1}|^2 \\
\partial_t |c_*| &\lesssim |c|^2 |c_*|.
\end{align*}
The claims \eqref{cminus-bound}, \eqref{cplus-bound}, \eqref{cstar-bound} now follow from Gronwall's inequality.  To obtain \eqref{ctotal-bound},
we take absolute values of \eqref{ceq-stable}, \eqref{ceq-unstable} and sum to obtain
$$ \partial_t (|c_{j \pm 1}^-| + |c_{j \pm 1}^+|) \leq (\sqrt{3} + O( |c|^2 )) (|c_{j \pm 1}^-| + |c_{j \pm 1}^+|) $$
and the claim now follows from Gronwall's inequality again.
\end{proof}

\begin{remark}\label{slider-remark} The slider solution \eqref{eq:slider}, using the $j=1$ coordinates, is simply
$c_2^-(t) = 0, c_2^+(t) = (1 + e^{-2\sqrt{3} t})^{-1/2}$, thus escaping away from $\T_1$ using the unstable component of the $2$-mode.
Viewed instead in the $j=2$ coordinates, it becomes $c_1^-(t) = (1 + e^{2\sqrt{3}t})^{-1/2}, c_1^+(t) = 0$, thus collapsing into $\T_2$
using the stable component of the $1$-mode.  We shall use this solution  to transition from $(M^+_j,d^+_j,R^+_j)$ to $(M^-_{j+1},d^-_{j+1},R^-_{j+1})$.
\end{remark}

\subsection{Construction of the Targets}\label{construct-sec}

\begin{quotation}
\emph{Always be nice to people on the way up, because you'll meet the same people on the way down.} (Wilson Mizner)
\end{quotation}

We now are ready to construct the targets $(M^-_j,d^-_j,R^-_j)$, $(M^0_j, d^0_j,R^0_j)$, $(M^+_j, d^+_j,R^+_j)$.  As it turns out, these sets will lie close to $\T_j$ (and thus away from the coordinate singularity at $b_j = 0$), and will therefore be represented using the modes
$c_{j+1}^\pm, c_{j-1}^\pm, c_*$ from Proposition \ref{coord}.  (The modes $r,\theta$ are also present but will have no impact on our computations.)
Broadly speaking, these targets will demand a lot of control on the leading modes (in order to set up future ``ricochets'' off of subsequent tori $\T_{j+1}$, $\T_{j+2}$, etc.) but will be rather relaxed about the trailing modes (as they will become small and stay small for the remainder of the evolution).

We need a number of parameters.  First, we need for technical reasons an increasing set of exponents
$$ 1 \ll A^0_3 \ll A^+_3 \ll A^-_4 \ll \ldots \ll A^-_{N-2} \ll A^0_{N-2}$$
for sake of concreteness, we will take these to be consecutive powers of $10$ (thus $A^-_3 = 10$, $A^0_3 = 10^2$, and so forth up to
$A^0_{N-2} = 10^{3N-13}$).

Next, we shall need a small parameter
$$ 0 < \sigma \ll 1$$
depending on $N$ and the exponents $A$ (actually one could take $\sigma = 1/100$ quite safely).
This basically measures the distance to $\T_j$ at which the local coordinates become effective, and the linear terms in Proposition \ref{coord} dominate the cubic terms.

For technical reasons, we shall need a set of scale parameters
$$ 1  \ll r^0_{N-2} \ll r^-_{N-2} \ll r^+_{N-3} \ll \ldots \ll r^+_3 \ll r^0_3 $$
where each parameter is assumed to be sufficiently large depending on the preceding parameters and on $\sigma$ and the $A$'s; these parameters represent a certain shrinking of each target from the previous one (in order to guarantee that each target can be covered by the previous).

Finally, we need a very large time parameter
$$ T \gg 1$$
that we shall assume to be as large as necessary depending on all the previous parameters (in particular, we will obtain exponential gains in $T$ that will handle all losses arising from the $A$, $\sigma$, $r$ parameters.)

For each $3 \leq j \leq N-2$, we also need the concatenated coordinates
\begin{align*}
c_{\leq j-2} &:= (c_1,\ldots,c_{j-2}) \in \C^{j-2} \quad \hbox{ (trailing peripheral modes) } \\
c_{\leq j-1} &:= (c_1,\ldots,c_{j-1}) \in \C^{j-1} \quad \hbox{ (trailing modes) } \\
c_{\geq j+1} &:= (c_{j+1},\ldots,c_M) \in \C^{N-j} \quad \hbox{ (leading modes) }\\
c_{\geq j+2} &:= (c_{j+2},\ldots,c_M) \in \C^{N-j-1} \quad \hbox{ (leading peripheral modes) }.
\end{align*}

Very roughly, the targets $(M^-_j,d^-_j,R^-_j)$, $(M^0_j, d^0_j,R^0_j)$, $(M^+_j, d^+_j,R^+_j)$ can be defined in terms of the modes $c$ by Table \ref{roughtable}.  Thus for instance, the $c_{\geq j+2}$ mode of $(M^-_j,d^-_j,R^-_j)$
ranges over arbitrary values of magnitude $O( r_j^-  e^{-2\sqrt{3} T} )$, plus an unavoidable uncertainty of magnitude $O( T^{A_j^-} e^{-3\sqrt{3} T} )$.
As one advances from $(M^-_j,d^-_j,R^-_j)$, to $(M^0_j,d^0_j,R^0_j)$, to $(M^+_j,d^+_j,R^+_j)$, to $(M^-_{j+1},d^-_{j+1},R^-_{j+1})$, and so forth, the uncertainty will increase by a polynomial factor in $T$ (this will be where the $A$ exponents come in), while the size of the manifolds $M$ will shrink somewhat (this will be where the $r$ parameters come in).

It will take time $T$ to flow from $(M^-_j,d^-_j,R^-_j)$ to $(M^0_j,d^0_j,R^0_j)$, and time $T$ from $(M^0_j,d^0_j,R^0_j)$ to $(M^+_j, d^+_j,R^+_j)$.  (On the other hand, we will be able to flow from $(M^+_j,d^+_j,R^+_j)$ to $(M^-_{j+1}, d^-_{j+1},R^-_{j+1})$ in time $O(\log \frac{1}{\sigma})$.)  The reader may wish to verify that this is broadly consistent
with Table \ref{roughtable}, using the heuristics from the previous section that the stable modes $c_{j \pm 1}^-$ should decay by a factor of $e^{-\sqrt{3} T}$ over this time, the unstable modes $c_{j\pm 1}^+$ should grow by the same factor of $e^{\sqrt{3} T}$, and the remaining modes
$c_{\leq j-2}$, $c_{\geq j-2}$ should simply oscillate.

We now give more precise definitions of these objects.

\begin{table}[ht]
\caption{The targets $(M^-_j,d^-_j,R^-_j)$, $(M^0_j, d^0_j,R^0_j)$, $(M^+_j, d^+_j,R^+_j)$. It is the factors of $e^{-\sqrt{3} T}$ which are the most important feature here; the polynomial powers of $T$ in the uncertainties, and the scales $r$ in the main terms, are technical corrections which should be ignored at a first reading, and $\sigma$ should be thought of as a small constant (independent of $T$ or the $r$).  Note that the outgoing target $(M^+_j, d^+_j, R^+_j)$ resembles a shifted version of the incoming target $(M^-_j,d^-_j,R^-_j)$; this will be important in Section \ref{transit-sec}.}\label{roughtable}
\begin{tabular}{|l|l|l|l|}
\hline
Mode & $(M^-_j,d^-_j,R^-_j)$ & $(M^0_j,d^0_j,R^0_j)$ & $(M^+_j,d^+_j,R^+_j)$  \\
\hline
$c_{\leq j-2}$ & $0 + O( T^{A_j^-} e^{-2\sqrt{3} T} )$
               & $0 + O( T^{A_j^0} e^{-2\sqrt{3} T} )$
               & $0 + O( T^{A_j^+} e^{-2\sqrt{3} T} )$\\
\hline
$c_{j-1}^-$    & $\sigma + O( T^{A_j^-} e^{-\sqrt{3} T} )$
               & $0 + O( T^{A_j^0} e^{-\sqrt{3} T} )$
               & $0 + O( T^{A_j^+} e^{-2\sqrt{3} T} )$  \\
\hline
$c_{j-1}^+$    & $0 + O( T^{A_j^-} e^{-4\sqrt{3} T} )$
               & $0 + O( T^{A_j^0} e^{-3\sqrt{3} T} )$
               & $0 + O( T^{A_j^+} e^{-2\sqrt{3} T} )$\\
\hline
$c_{j+1}^-$    & $O( r_j^- e^{-2\sqrt{3} T} )$
               & $0 + O( T^{A_j^0} e^{-3\sqrt{3} T} )$
               & $0 + O( T^{A_j^+} e^{-4\sqrt{3} T} )$ \\
&$\quad + O( T^{A_j^-} e^{-3\sqrt{3} T} )$
&
&\\
\hline
$c_{j+1}^+$    & $O( r_j^- e^{-2\sqrt{3} T} )$
               & $O( \sigma e^{-\sqrt{3} T} )$
               & $\sigma + O( T^{A_j^+} e^{-\sqrt{3} T} )$                      \\
               &$\quad + O( T^{A_j^-} e^{-3\sqrt{3} T} )$
               &$\quad + O( T^{A_j^0} e^{-2\sqrt{3} T} )$
               &\\
\hline
$c_{\geq j+2}$ & $O( r_j^- e^{-2\sqrt{3} T} )$
               & $O( r_j^0 e^{-2\sqrt{3} T} )$
               & $O( r_j^+ e^{-2\sqrt{3} T} )$ \\
&$\quad + O( T^{A_j^-} e^{-3\sqrt{3} T} )$
&$\quad + O( T^{A_j^0} e^{-3\sqrt{3} T} )$
&$\quad + O( T^{A_j^+} e^{-3\sqrt{3} T} )$\\
\hline
$\theta$ & uncontrolled & uncontrolled & uncontrolled\\
\hline
\end{tabular}
\end{table}

\begin{figure}[htp]
\centering
\includegraphics[width=10cm]{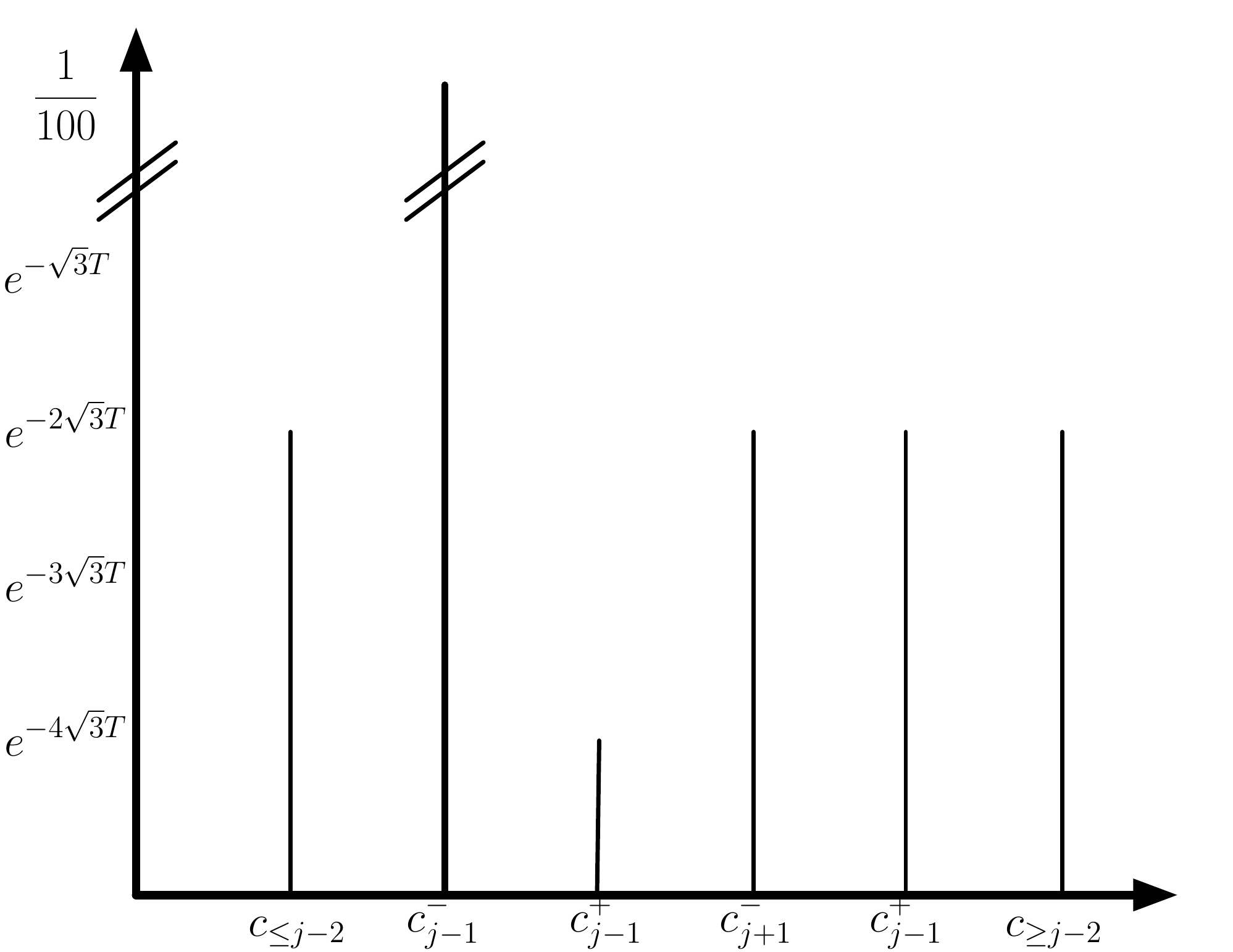}
\caption{An illustration of a configuration of modes within the incoming 
target $(M_j^-, d_j^-,
  R_j^-)$. The illustration ignores the tiny polynomial corrections. The diagonal lines highlight the vertical scale
  compression between the tiny exponential scales and the scale of
  $\sigma \thicksim \frac{1}{100}$.}\label{fig:MjMinusTarget}
\end{figure}

{\bf The Incoming target $(M^-_j,d^-_j,R^-_j)$:\\}

We define the \emph{incoming target} $(M^-_j,d^-_j,R^-_j)$ by setting $M^-_j$ to
be the set of all points whose coordinates obey the relations
\begin{align*}
c_{\leq j-2}, c_{j-1}^+ &= 0\\
c_{j-1}^- &= \sigma \\
|c_{\geq j+1}| &\leq r_j^- e^{-2\sqrt{3} T}
\end{align*}
with uncertainty $R^-_j := T^{A^-_j}$ and with semimetric $d^-_j(x, \tilde x)$ defined by
\begin{align*}
 d^-_j(x, \tilde x) &:= e^{2\sqrt{3} T} |c_{\leq j-2} - \tilde c_{\leq j-2}| + e^{\sqrt{3} T} |c_{j-1}^- - \tilde c_{j-1}^-| + e^{4\sqrt{3} T} |c_{j-1}^+ - \tilde c_{j-1}^+| \\
&+ e^{3\sqrt{3}T} |c_{\geq j+1} - \tilde c_{\geq j+1}|  )
\end{align*}
where $\tilde c$ of course denotes the coordinates of $\tilde x$.  This 
metric is not defined on the set $b_j=0$ (where the local coordinates break 
down), but this is not of importance to us because the 
metric is well defined for all points within  $(M^-_j, d^-_j, R^-_j)$, 
since $T$ is so large.  (If one wished, extend $d^-_j$ in some arbitrary 
fashion to be defined on the remaining portions of $\Sigma$, but this will 
have no impact on the argument.)

Informally, the incoming target has a significant presence (of size roughly $\sigma$) on the trailing stable mode, but is small elsewhere, except of course at the primary mode.

{\bf The Ricochet target $(M^0_j,d^0_j,R^0_j)$:\\}

We define the \emph{ricochet target} $(M^0_j,d^0_j,R^0_j)$ by setting $M^0_j$ to be the set of all points whose coordinates obey the relations
\begin{align*}
c_{\leq j-1},  c_{j+1}^- &= 0\\
|c_{j+1}^+| &\leq r_j^0 e^{-\sqrt{3} T} \\
|c_{\geq j+2}| &\leq r_j^0 e^{-2\sqrt{3} T}
\end{align*}
with uncertainty $R^0_j := T^{A^0_j}$ and with semimetric $d^0_j(x, \tilde x)$ defined by
\begin{align*}
 d^0_j(x, \tilde x) &:= e^{2\sqrt{3} T} |c_{\leq j-2} - \tilde c_{\leq j-2}| + e^{\sqrt{3} T} |c_{j-1}^- - \tilde c_{j-1}^-| +  e^{3\sqrt{3} T} |c_{j-1}^+ - \tilde c_{j-1}^+| + \\
&\quad  +  e^{3\sqrt{3} T} |c_{j+1}^- - \tilde c_{j+1}^-|
+  e^{2\sqrt{3} T} |c_{j+1}^+ - \tilde c_{j+1}^+| +
e^{3\sqrt{3} T} |c_{\geq j+2} - \tilde c_{\geq j+2}|.
\end{align*}

Informally, the ricochet target is small everywhere outside of the primary 
mode, but has its largest presence at the trailing stable mode and the leading unstable mode, with the latter just having overtaken the former.

\begin{figure}[htp]
\centering
\includegraphics[width=10cm]{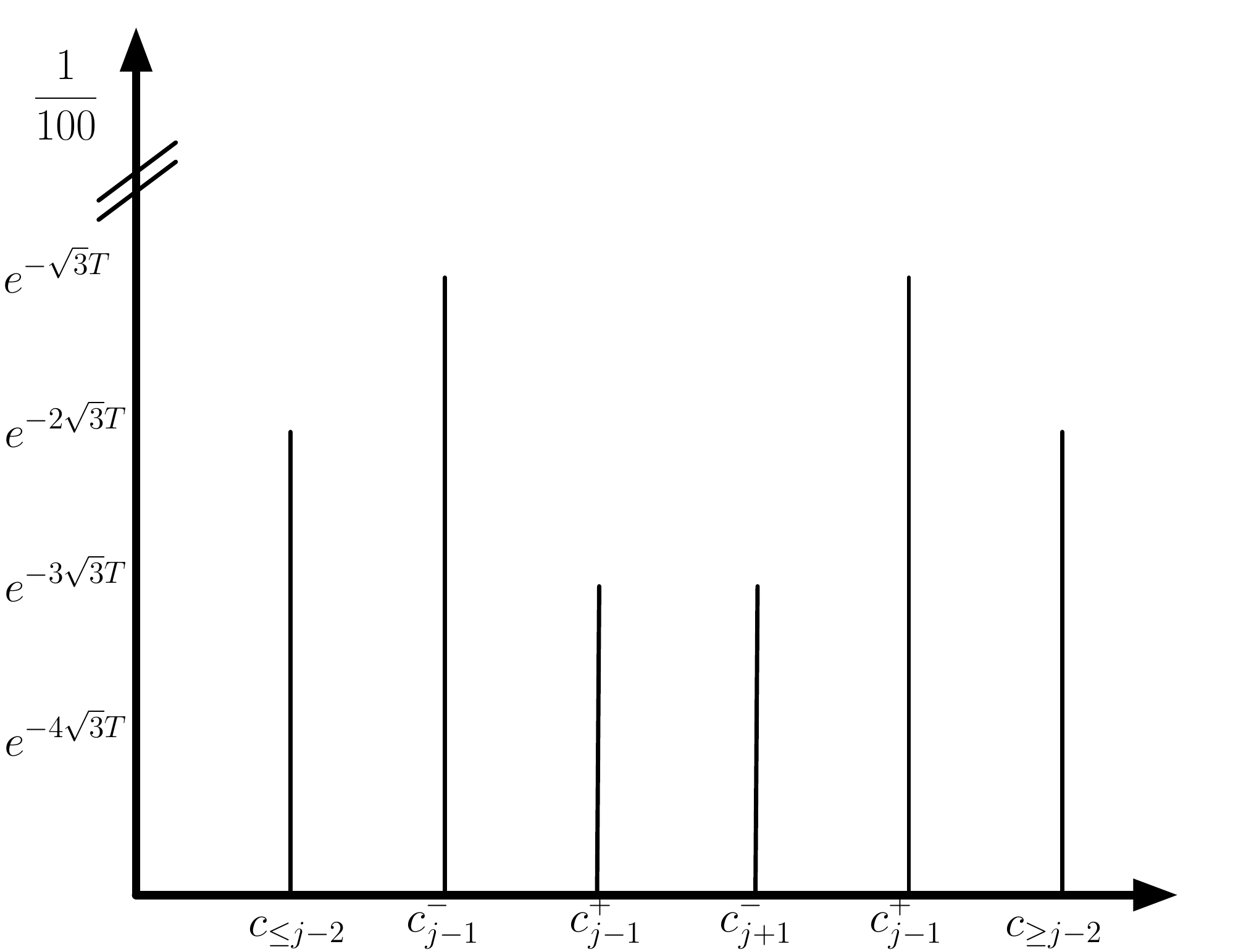}
\caption{An illustration of a configuration of modes  
within the ricochet target $(M_j^0, d_j^0,
  R_j^0)$. Note that all the modes displayed here are extremely small
  relative to $\sigma \thicksim \frac{1}{100}.$}\label{fig:MjRicochetTarget}
\end{figure}

{\bf The Outgoing target $(M^+_j,d^+_j,R^+_j)$:\\}

We define the \emph{outgoing target} $(M^+_j,d^+_j,R^+_j)$ by setting $M^+_j$ to be the set of all points whose coordinates obey the relations
\begin{align*}
c_{\leq j-1} = c_{j+1}^- &= 0\\
c_{j+1}^+ &= \sigma \\
|c_{\geq j+2}| &\leq r_j^+ e^{-2\sqrt{3} T}
\end{align*}
with uncertainty $R^+_j = T^{A^+_j}$
and with semimetric $d^+_j(x, \tilde x)$ defined by
\begin{align*}
d^+_j(x, \tilde x) &:= e^{2\sqrt{3} T} |c_{\leq j-1} - \tilde c_{\leq j-1}| + e^{4\sqrt{3} T} |c_{j+1}^- - \tilde c_{j+1}^-| \\
&\quad + e^{\sqrt{3} T} |c_{j+1}^+ - \tilde c_{j+1}^+|
+ e^{3\sqrt{3} T} |c_{\geq j+2} - \tilde c_{\geq j+2}| ).
\end{align*}

Informally, the outgoing target has a significant presence (of size $\sigma$) on the leading unstable mode and is small at all other secondary and peripheral modes.

\begin{figure}[htp]
\centering
\includegraphics[width=10cm]{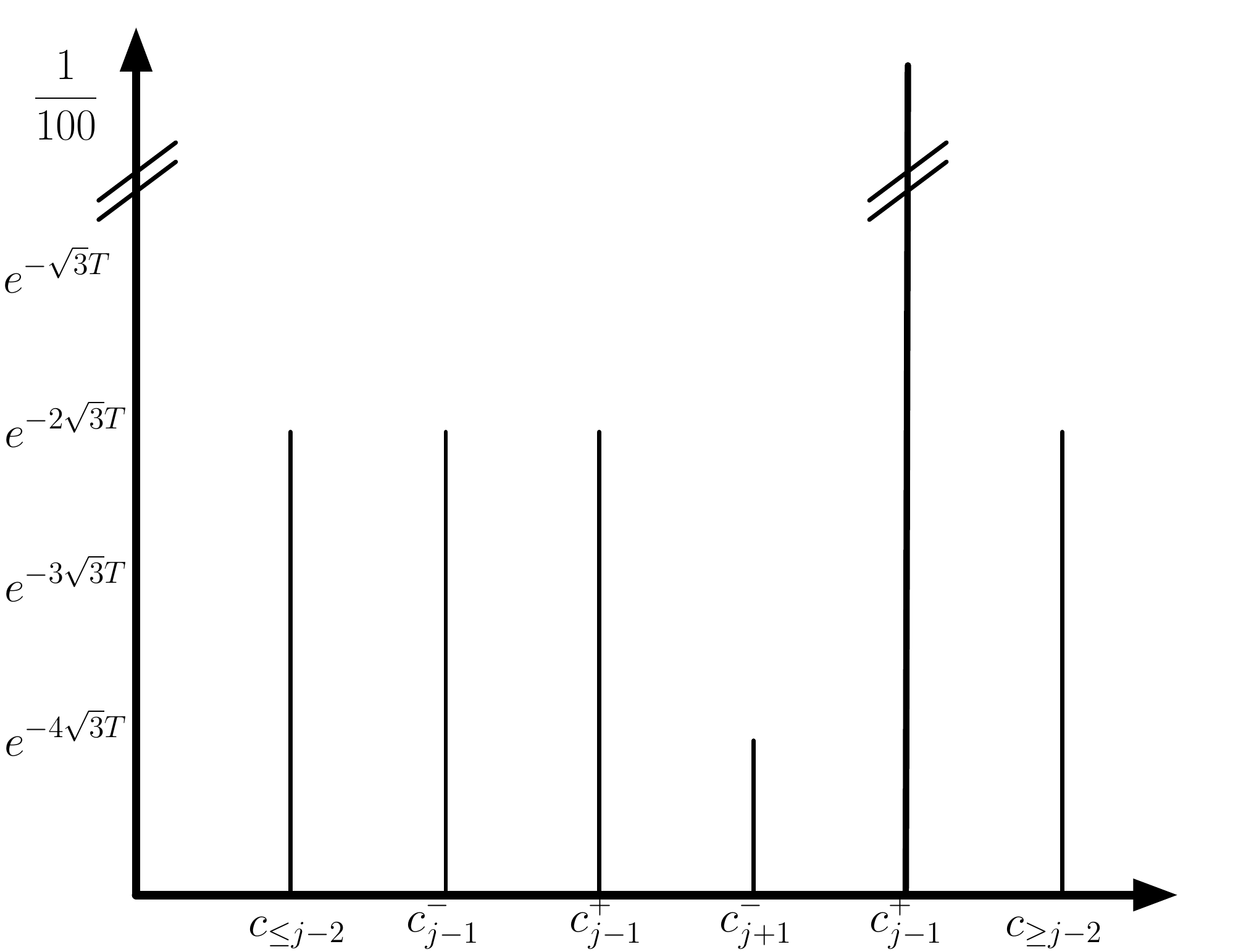}
\caption{An illustration of a configuration of modes within 
the outgoing target $(M_j^+, d_j^+,
  R_j^+)$. }\label{fig:MjOutgoingTarget}
\end{figure}

\subsection{Proof of Theorem \ref{arnold} assuming \eqref{long-chain}} \label{section:long-chain}

From \eqref{long-chain}, we see that there is at least one solution $b(t)$ to \eqref{eq:b-eq} which starts within 
the ricochet target $(M_3^0, d_3^0, R_3^0)$ at some time $t_0$ and ends up within the ricochet target
$(M_{N-2}^0, d_{N-2}^0, R_{N-2}^0)$ at some later time $t_1 > t_0$.  But from the definition of these targets, we thus see
that $b(t_0)$ lies within a distance $O(r_3^0 e^{-\sqrt{3} T} )$ of $\T_3$, while $b(t_1)$ lies within a distance
$O( r_{N-2}^0 e^{-\sqrt{3} T})$ of $\T_{N-2}$.  The claim follows.

It now remains to show the covering relationships \eqref{incoming}, \eqref{outgoing}, \eqref{transit}.  That is the purpose of the
next three sections; the main tool shall be repeated and careful applications of Gronwall-type inequalities.

\subsection{Flowing from the Incoming Target to the Ricochet Target}\label{incoming-sec}


Fix $3 \leq j \leq N-2$.
In this section we show that $(M^-_j,d^-_j,R^-_j)$ covers $(M^0_j,d^0_j,R^0_j)$; this is the lengthiest
and most delicate part of the argument.  We shall be able to flow from  $(M^-_j,d^-_j,R^-_j)$ to
$(M^0_j,d^0_j,R^0_j)$
for time exactly $T$.  To do this we will of course need good control on the evolution
of a solution starting within $(M^-_j,d^-_j,R^-_j)$ for this length of time.  The full evolution is summarized in Table \ref{intable}, but
to establish this behavior we shall need to proceed in stages.  Firstly, we need upper bounds on the flow; then we bootstrap these upper bounds to more precise asymptotics; then we use these asymptotics to hit arbitrary locations on the final target $(M^0_j, d^0_j, R^0_j)$ starting from a well-chosen
location on the initial target $(M^-_j, d^-_j, R^-_j)$.  This basic strategy will also be employed in the next two sections, though the technical details are slightly different in each case.

\begin{table}[ht]
\caption{The evolution of a solution which is within 
the incoming target $(M^-_j,d^-_j,R^-_j)$ at time $t=0$.  The $a_{\geq 2}, a_j^+, a_j^-$ are
explicit numbers defined below in \eqref{ajsdefinedhere}, \eqref{aj0}.
The function $g(t)$ and the transfer matrices $G_{j+1}(t)$, $G_{\geq j+2}(t)$ are
explicit objects (depending only on $\sigma$) that will be defined below. 
The configuration of the modes in the incoming target at $t=0$ is 
illustrated in Figure \ref{fig:MjMinusTarget}. Upon
arrival at $t=T$, the modes have adjusted and appear as in Figure
\ref{fig:MjRicochetTarget}. }\label{intable}
\begin{tabular}{|l|l|l|l|}
\hline
Mode & $t=0$ & $0 < t < T$ & $t=T$  \\
\hline
$c_{\leq j-2}$ & $O( T^{A_j^-} e^{-2\sqrt{3} T} )$           & $O( T^{A_j^-} e^{-2\sqrt{3} T} )$                & $O( T^{A_j^-} e^{-2\sqrt{3} T} )$ \\
\hline
$c_{j-1}^-$    & $\sigma + O( e^{-\sqrt{3} T} )$     & $g(t) + O( e^{-\sqrt{3} T} )$            & $O( e^{-\sqrt{3} T} )$  \\
\hline
$c_{j-1}^+$    & $O( T^{A_j^-} e^{-4\sqrt{3} T} )$           & $O( T^{2A_j^-+1} e^{-4\sqrt{3} T} e^{\sqrt{3} t})$  & $O( T^{2A_j^-+1} e^{-3\sqrt{3} T} )$ \\
\hline
$c_{j+1}^-$    & $e^{-2\sqrt{3} T} a^-_{j+1}$
               & $O( T^{A_j^-} e^{-2\sqrt{3} T} e^{-\sqrt{3} t})$
               & $O( T^{A_j^-} e^{-3\sqrt{3} T} )$ \\
               & \quad $+ O( T^{A_j^-}e^{-3\sqrt{3} T} )$ &&\\
\hline
$c_{j+1}^+$    & $e^{-2\sqrt{3} T} a^+_{j+1}$
               & $e^{-2\sqrt{3} T} e^{\sqrt{3} t} G_{j+1}^+(t) a^+_{j+1}$
               & $e^{-\sqrt{3} T} G_{j+1}^+(T) a^+_{j+1}$               \\
&\quad $+ O( T^{A_j^-}e^{-3\sqrt{3} T} )$ & \quad  $+ O( T^{A_j^-+2} e^{-3\sqrt{3} T} e^{\sqrt{3} t} )$ & \quad $+ O(  T^{A_j^-+2} e^{-2\sqrt{3} T} )$ \\
\hline
$c_{\geq j+2}$ & $e^{-2\sqrt{3} T} a_{\geq j+2}$
               & $e^{-2\sqrt{3} T} G_{\geq j+2}(t) a_{\geq j+2}$
               & $e^{-2\sqrt{3} T} G_{\geq j+2}(T) a_{\geq j+2}$ \\
              & \quad $+ O( T^{A_j^-} e^{-3\sqrt{3} T} )$ & \quad $+ O( T^{A_j^-} e^{-3\sqrt{3} T} )$ & \quad $+ O( T^{A_j^-} e^{-3\sqrt{3} T} )$\\
\hline
\end{tabular}
\end{table}

%


We begin with some basic upper bounds on the flow.

\begin{proposition}[Upper bounds, inbound leg]\label{growthprop}
Let $b(t)$ be a solution to \eqref{eq:b-eq} such that $b(0)$ is within $(M^-_j, d^-_j,R^-_j)$.
Let $c(t)$ denote the coordinates of $b(t)$ as in Proposition \ref{coord}.  Then we have the bounds
\begin{align*}
|c_*(t)|       &= O( T^{A_j^-} e^{-2\sqrt{3} T} ) \\
|c_{j-1}^-(t)| &= O( \sigma e^{-\sqrt{3} t} ) \\
|c_{j-1}^+(t)| &= O( T^{2A_j^-+1} e^{-4\sqrt{3} T} e^{\sqrt{3} t} ) \\
|c_{j+1}^-(t)| &= O( r_j^- (1+t) e^{-2\sqrt{3} T} e^{-\sqrt{3} t} ) \\
|c_{j+1}^+(t)| &= O( r_j^- e^{-2\sqrt{3} T} e^{\sqrt{3} t} )
\end{align*}
for $0 \leq t \leq T$.
\end{proposition}


\begin{proof}  From the hypothesis that $b(0)$ is within $(M^-_j, d^-_j,R^-_j)$, we easily verify that the bounds hold at time $t=0$.
To establish the bounds for later times, we can use the continuity method.  Let $C_0$ be a large constant (depending only on $N$)
to be chosen later, and suppose that
$0 \leq T' \leq T$ is a time for which the bounds
\begin{align*}
|c_*(t)|       &= O( C_0 T^{A_j^-} e^{-2\sqrt{3} T} ) \\
|c_{j-1}^-(t)| &= O( C_0 \sigma e^{-\sqrt{3} t} ) \\
|c_{j-1}^+(t)| &= O( C_0 T^{2A_j^-+1} e^{-4\sqrt{3} T} e^{\sqrt{3} t} ) \\
|c_{j+1}^-(t)| &= O( C_0 r_j^- (1+t) e^{-2\sqrt{3} T} e^{-\sqrt{3} t} ) \\
|c_{j+1}^+(t)| &= O( C_0 r_j^- e^{-2\sqrt{3} T} e^{\sqrt{3} t} )
\end{align*}
are known to hold for all $0 \leq t \leq T'$.  We will then prove the same bounds without the $C_0$ factor for the same 
range of times $t$,
implying that the set of times $T'$ for which the above statements hold is both open and closed; since that set contains $0$, it must then also contain $T$ and we are done.

Our main tool shall of course be Lemma \ref{ubound}.  From our bootstrap hypotheses (and the largeness of $T$) we can control the total magnitude of the modes:
\begin{equation}\label{ctb}
|c(t)| \lesssim C_0 \sigma e^{-\sqrt{3} t} + C_0 r_j^- e^{-2\sqrt{3} T} e^{\sqrt{3} t}
\end{equation}
for all $0 \leq t \leq T'$ (the trailing stable mode is almost always dominant, except near time $T$ where the leading unstable mode begins to compete)
and hence \eqref{cl2-bound} will hold for all $0 \leq t \leq T'$ (assuming $\sigma$ is small depending on $C_0$).  From \eqref{cstar-bound} we
thus have
\begin{equation}\label{cstar-good}
|c_*(t)| \lesssim |c_*(0)| \lesssim T^{A_j^-} e^{-2\sqrt{3} T}
\end{equation}
for all $0 \leq t \leq T'$, which is the desired bound on the peripheral modes $c_*(t)$.  Similarly from \eqref{ctotal-bound} we have
\begin{equation}\label{cj1-good}
|c_{j+1}(t)| \lesssim e^{\sqrt{3} t} |c_{j+1}(0)| \lesssim r_j^- e^{-2\sqrt{3} T} e^{\sqrt{3} t}
\end{equation}
which gives the desired bound on the leading unstable mode $c_{j+1}^+$.  From \eqref{cstar-good}, \eqref{cj1-good} we have
\begin{equation}\label{cje}
 |c_{\neq j-1}(t)| \lesssim e^{-2\sqrt{3} T} (T^{A_j^-} + r_j^- e^{\sqrt{3} t})
\end{equation}
and hence from \eqref{cminus-bound} we have
$$ |c_{j-1}^-(t)| \lesssim e^{-\sqrt{3} t} |c_{j-1}^-(0)| + \int_0^t e^{-\sqrt{3}(t-s)} |c_{j-1}^+(s)| e^{-4\sqrt{3} T} (T^{A_j^-} + r_j^- e^{\sqrt{3} s})^2\ ds.$$
Since $c_{j-1}^-(0) = O(\sigma)$ and $c_{j-1}^+(s) = O( C_0 T^{2A_j^-+1} e^{-4\sqrt{3} T} e^{\sqrt{3} t} )$, we conclude (using the largeness of $T$) that
\begin{equation}\label{cjp}
|c_{j-1}(t)^-| = O( \sigma e^{-\sqrt{3} t} )
\end{equation}
which is the desired bound on the trailing stable mode.  This and \eqref{cstar-good} imply that
$$ |c_{\neq j+1}(t)| \lesssim \sigma e^{-\sqrt{3} t} $$
and then from \eqref{cminus-bound} we have
$$ |c_{j+1}^-(t)| \lesssim e^{-\sqrt{3} t} |c_{j+1}^-(0)|
+ \int_0^t e^{-\sqrt{3}(t-s)} |c_{j+1}^+(s)| \sigma^2 e^{-2\sqrt{3} s}\ ds.$$
Using \eqref{cj1-good} to estimate $|c_{j+1}^+(s)|$, together with the initial bound $|c_{j+1}^-(0)| =  O( r_j^- e^{-2\sqrt{3} T} )$
we obtain
$$ |c_{j+1}^-(t)| \lesssim r_j^- (1+t) e^{-\sqrt{3} t} e^{-2\sqrt{3} T}$$
which is the desired bound on the leading stable mode.  Finally, from \eqref{cplus-bound} and \eqref{cje} we have
$$ |c_{j-1}^+(t)| \lesssim e^{\sqrt{3} t} |c_{j-1}^+(0)| + \int_0^t e^{\sqrt{3}(t-s)} |c_{j-1}^-(s)| e^{-4\sqrt{3} T} (T^{A_j^-} + r_j^- e^{\sqrt{3} s})^2\ ds$$
and thus from \eqref{cjp} and the initial bound $|c_{j-1}^+(0)| = O( T^{A_j^-} e^{-4\sqrt{3} T} )$ we have
$$ |c_{j-1}^+(t)| \lesssim T^{2A_j^-+1} e^{-4\sqrt{3} T} e^{\sqrt{3} t}$$
which is the desired bound on the trailing unstable mode.
\end{proof}

For the rest of this section, the time variable $t$ is assumed to lie in the range $0 \leq t \leq T$.  Let us make some more precise hypotheses on the initial data, namely
\begin{align}
c_{j+1}^-(0) &= e^{-2\sqrt{3} T} a_{j+1}^- + O( T^{A_j^-} e^{-3\sqrt{3} T} )\nonumber\\
c_{j+1}^+(0) &= e^{-2\sqrt{3} T} a_{j+1}^+ + O( T^{A_j^-} e^{-3\sqrt{3} T} )\label{ajsdefinedhere}\\
c_{\geq j+2}(0) &= e^{-2\sqrt{3} T} a_{\geq j+2} + O( T^{A_j^-} e^{-3\sqrt{3} T} ) \nonumber
\end{align}
for some data $a_{j+1}^\pm \in \R$, $a_{\geq j+2} \in \C^{N-j-1}$ of magnitude at most $r_j^-/2$ which we shall choose later.
In order to reach the turnaround set $(M^0_j, d^0_j,R^0_j)$, we will need slightly more precise bounds on the leading modes (and also the stable trailing mode, which is dominant) as follows.

\paragraph{{\bf Improved control on $c_{j-1}^-(t)$}} 

We first give a better bound on the trailing stable mode $c_{j-1}^-(t)$, which is the largest of all the modes (other than the primary one, of course).  Observe from Proposition \ref{growthprop} that any cubic term $\bigO( c^3 )$ splits as the sum of a main term of the form
$\bigO( (c_{j-1}^-)^3) )$, plus an error of size at most $O( T^{A_j^-} e^{-2\sqrt{3} T} )$ (say).  Thus we may rewrite \eqref{ceq-stable} somewhat crudely as
$$
\partial_t c_{j - 1}^- = -\sqrt{3} c_{j - 1}^- + \bigO( (c_{j - 1}^-)^3 ) + O( T^{A_j^-} e^{-2\sqrt{3} T} )$$
for some explicit cubic expression $\bigO( (c_{j - 1}^-)^3 )$ of $c_{j-1}^-$.
Now let $g$ be the solution\footnote{This function $g$ can in fact be computed explicitly, in fact it is essentially the function appearing in Remark \ref{slider-remark}, up to a translation in time.  However, we will not need to know the exact formula for it here; the only relevant features for us is that $g$ depends only on $\sigma$ (and possibly $N$) and obeys the bound \eqref{gbound}.} to the corresponding exact (scalar) equation
$$
\partial_t g = -\sqrt{3} g + \bigO( g^3 )$$
with the same initial data $g(0) = \sigma$.  Because $\sigma$ is small, it is easy to establish the decay bound
\begin{equation}\label{gbound}
g(t) = O( \sigma e^{-\sqrt{3} t} )
\end{equation}
(e.g. by the continuity method).  Writing
$c_{j-1}^- = g + E_{j-1}^-$, the error function $E_{j-1}^-(t)$ thus obeys the difference equation
$$ \partial_t E_{j-1}^- = - \sqrt{3} E_{j-1}^- + O( \sigma^2 e^{-2\sqrt{3} t} |E_{j-1}^-| ) + O( |E_{j-1}^-|^3 ) + O( T^{A_j^-} e^{-2\sqrt{3} T} )$$
with initial data $E_{j-1}^-(0) = 0$.  From this equation it is an easy matter (e.g. by the continuity method) to establish the bound
$$ E_{j-1}^-(t) = O( T^{A_j^-+1} e^{-2\sqrt{3} T} )$$
for all $0 \leq t \leq T$.  In other words we have the estimate
\begin{equation}\label{cj-precise}
 c_{j-1}^-(t) = g(t) + O( T^{A_j^-+1} e^{-2\sqrt{3} T} ).
\end{equation}
This and Proposition \ref{growthprop} allow us to refine our bound for $c^2$ and for $c_{\neq j+1}^2$:
\begin{equation}\label{c2-bound}
 \bigO( c^2 ) = \bigO( g^2 ) + O( T^{A_j^-+1} e^{-2\sqrt{3} T} )
\end{equation}
and
\begin{equation}\label{c2-bound-better}
 \bigO( c_{\neq j+1}^2 ) = \bigO( g^2 ) + O( T^{A_j^-+1} e^{-2\sqrt{3} T} e^{-\sqrt{3} t} ).
\end{equation}

\paragraph{\bf Improved control on $c_{\geq j+2}$}  

Now we control the leading peripheral modes.
Inserting \eqref{c2-bound} into \eqref{ck-eq} we see that
$$ \partial_t c_{\geq j+2} = i c_{\geq j+2} + \bigO( c_{\geq j+2} g^2 ) + O( T^{A_j^-+1} e^{-2\sqrt{3} T} |c_{\geq j+2}| ).$$
We approximate this by the corresponding linear equation
$$ \partial_t u = i u + \bigO( u g^2 )$$
where $u(t) \in \C^{N-j-1}$.  This equation has a fundamental solution
$G_{\geq 2}(t): \C^{N-j-1} \to \C^{N-j-1}$ for all $t \geq 0$, thus $G_{\geq 2}(t) u(0) = u(t)$.
(Again, this solution could be described explicitly since $g$ is itself explicit, but we will not need to do so here).
From \eqref{gbound} we have
\begin{equation}\label{gbound-2}
\int_0^T g^2(t)\ dt = O(1),
\end{equation}
and so an easy application of Gronwall's inequality shows that
\begin{equation}\label{gdither}
 | G_{\geq 2}(t) |, |G_{\geq 2}(t)^{-1} | = O(1).
 \end{equation}
Since $c_{\geq j+2}(0) =  e^{-2\sqrt{3} T} a_{\geq j+2} + O( T^{A_j^-}  e^{-3\sqrt{3} T} )$, we are motivated to use the ansatz
$$ c_{\geq j+2} = e^{-2\sqrt{3} T} G_{\geq 2}(t) a_{\geq j+2} + E_{\geq j+2}.$$
The error $E_{\geq j+2}$ then solves the equation
$$ \partial_t E_{\geq j+2} = iE_{\geq j+2} + \bigO( E_{\geq j+2} g^2 ) + O( T^{A_j^-+1} e^{-2\sqrt{3} T} |c_{\geq j+2}| )$$
with initial data $E_{\geq j+2}(0) = O( T^{A_j^-} e^{-3\sqrt{3} T} )$.  Applying the bound on $c_{\geq j+2}$ from Proposition \ref{growthprop}
we see that
$$ \partial_t |E_{\geq j+2}| = O( |E_{\geq j+2}| |g|^2 ) + O( T^{2A_j^-+1} e^{-4\sqrt{3} T} ).$$
From Gronwall's inequality and \eqref{gbound-2} we conclude that
$$ |E_{\geq j+2}(t)| = O( T^{A_j^-} e^{-3\sqrt{3} T} )$$
for all $0 \leq t \leq T$, and thus
\begin{equation}\label{cj2-bound}
c_{\geq j+2}(t) = e^{-2\sqrt{3} T} G_{\geq 2}(t) a_{\geq j+2} + O( T^{A_j^-} e^{-3\sqrt{3} T} ).
\end{equation}

\paragraph{\bf Improved control on $c_{j+1}$} 

Now we consider the two leading secondary modes $c_{j+1}^+$, $c_{j+1}^-$ simultaneously.
From \eqref{ceq-stable}, \eqref{ceq-unstable}, \eqref{c2-bound}, \eqref{c2-bound-better}, and Proposition \ref{growthprop} we have the system
\begin{align*}
\partial_t c_{j + 1}^- &= -\sqrt{3} c_{j + 1}^- + \bigO( g^2 c_{j+1}^- ) + O( T^{A_j^-+1} e^{-4\sqrt{3} T} )\\
\partial_t c_{j + 1}^+ &= \sqrt{3} c_{j + 1}^+ + \bigO( g^2 c_{j+1}^+ ) + O( T^{A_j^-+1} e^{-4\sqrt{3} T} e^{\sqrt{3} t} ).
\end{align*}
If we make the ansatz
$$ c_{j+1}^- = e^{-2\sqrt{3} T} e^{-\sqrt{3} t} \tilde a_{j+1}^-(t); \quad c_{j+1}^+ = e^{-2\sqrt{3} T} e^{\sqrt{3} t} \tilde a_{j+1}^+(t)$$
to eliminate the constant coefficient terms, then the system becomes
\begin{align*}
\partial_t \tilde a_{j + 1}^- &= \bigO( g^2 \tilde a_{j+1}^- ) + \bigO( g^2 e^{2\sqrt{3} t} \tilde a_{j+1}^+ )  + O( T^{A_j^-+1} e^{-2\sqrt{3} T} e^{\sqrt{3} t} )\\
\partial_t \tilde a_{j + 1}^+ &= \bigO( g^2 e^{-2\sqrt{3} t} \tilde a_{j+1}^- ) + \bigO( g^2 \tilde a_{j+1}^+ ) + O( T^{A_j^-+1} e^{-2\sqrt{3} T} )
\end{align*}
with initial conditions $\tilde a_{j+1}^\pm(0) = a_{j+1}^\pm + O( T^{A_j^-} e^{-\sqrt{3} T} )$.
Writing $a_{j+1} := \left( \begin{array}{l} a_{j+1}^- \\ a_{j+1}^+ \end{array} \right)$ and $\tilde a_{j+1}(t) := \left( \begin{array}{l} \tilde a_{j+1}^-(t) \\ \tilde a_{j+1}^+(t) \end{array} \right)$,
we can write this as
\begin{equation}\label{aj0}
\partial_t \tilde a_{j+1}(t) = A(t) \tilde a_{j+1}(t) + O( T^{A_j^-+1} e^{-2\sqrt{3} T} e^{\sqrt{3} t} ); \quad \tilde a_{j+1}(0) = a_{j+1} + O( T^{A_j^-} e^{-\sqrt{3} T} )
\end{equation}
where $A(t)$ is an explicit real $2 \times 2$ matrix (depending only on $\sigma$ and $t$) which (by \eqref{gbound}) has bounds of the form
$$ A(t) = \sigma^2  \left( \begin{array}{ll}
O( e^{-2\sqrt{3} t} ) & O(1) \\
O( e^{-4\sqrt{3} t} ) & O( e^{-2\sqrt{3} t} )
\end{array} \right ).$$
Unfortunately, the non-decaying coefficient $O(1)$ here prevents a direct application of Gronwall's inequality from being effective.
However, because this coefficient is located in a ``nilpotent'' part of the matrix, we can proceed using the following variant of Gronwall's lemma.

\begin{lemma}[Gronwall-type inequality]\label{nilgronwall}  Let $x(t), y(t)$ be vector-valued functions obeying the differential inequalities
\begin{align*}
|\partial_t x(t)| &\lesssim \delta e^{-\alpha t} |x(t)| + \delta |y(t)| + |F(t)| \\
|\partial_t y(t)| &\lesssim \delta e^{-\beta t} |x(t)| + \delta e^{-\gamma t} |y(t)| + |G(t)|
\end{align*}
for some $\alpha,\beta,\gamma > 0$, some $0<\delta < 1$, all $0 \leq t \leq T$, and some forcing terms $F(t), G(t)$.  Then we have
\begin{align*}
|x(t)| &\lesssim (1+t) |x(0)| + t |y(0)| + \int_0^t (1+t e^{-\beta s}) |F(s)| + t |G(s)|\ ds\\
|y(t)| &\lesssim |x(0)| + |y(0)| + \int_0^t e^{-\beta s} |F(s)| + |G(s)|\ ds
\end{align*}
for all $0 \leq t \leq T$,
where the implicit constants are allowed to depend on $\alpha,\beta,\gamma$.  If $F=G=0$ and $\delta$ is
sufficiently small depending on $\alpha,\beta,\gamma$, we also have the lower bound
$$ |y(t)| \geq \frac{1}{2} |y(0)| - O( |x(0)| ).$$
\end{lemma}

\begin{proof} Throughout this proof we assume that $t$ lies in the range $0 \leq t \leq T$, and implied constants can depend on
$\alpha,\beta,\gamma$.
From the equation for $\partial_t x(t)$ and the usual Gronwall inequality we have
$$ |x(t)| \lesssim |x(0)| + \int_0^t |y(s)| + |F(s)|\ ds.$$
Writing $Y(t) := \sup_{0 \leq s \leq t} |y(s)|$, we conclude that
\begin{equation}\label{xb}
|x(t)| \lesssim |x(0)| + \int_0^t |F(s)|\ ds + t Y(t).
\end{equation}
On the other hand, from the equation for $\partial_t y(t)$ and Gronwall's inequality we have
$$ |y(t)| \lesssim |y(0)| + \int_0^t e^{-\beta s} |x(s)| + |G(s)|\ ds;$$
inserting \eqref{xb} we conclude
$$ |y(t)| \lesssim |y(0)| + \int_0^t e^{-\beta s} (|x(0)| + \int_0^s F(s')\ ds') + |G(s)|\ ds + \int_0^t e^{-\beta s} Y(s)s \ ds.$$
By Fubini's theorem we have $\int_0^T e^{-\beta s}(|x(0)|+ \int_0^s F(s')\ ds')\ ds \lesssim |x(0)| + \int_0^T e^{-\beta s} F(s)\ ds$.  Taking suprema in $t$ we conclude that
$$ Y(t) \lesssim |y(0)| + |x(0)| + \int_0^t e^{-\beta s} |F(s)| + |G(s)|\ ds + \int_0^t e^{-\frac{\beta}{2} s} Y(s)\ ds$$
and hence by Gronwall's inequality again
$$ |y(t)| \lesssim Y(t) \lesssim |y(0)| + |x(0)| + \int_0^t e^{-\beta s} |F(s)| + |G(s)|\ ds.$$
The upper bounds on $x$ and $y$ now follows from \eqref{xb}.

Now suppose that $F=G=0$ and $\delta$ is small.
The triangle inequality gives us,
\begin{align*}
|y(t)| & \geq |y(0) - \int_0^t |\partial_s y(s)| ds.
\end{align*}
With the given bound for $|\partial_s y(s)|$ and the upper bounds we just proved for
$y(s), x(s)$, we conclude,
\begin{align*}
|y(t)| & \geq |y(0)| - \int_0^t C \delta e^{-\beta s}|x(s) ds - \int_0^t C \delta e^{-\gamma s}|y(s) ds \\
& \geq |y(0)| - \int_0^t C \delta (1+t)(x(0) + y(0)) ds,
\end{align*}
and the final claim in the Lemma follows by taking $\delta$ small enough.
\end{proof}

Let $G_{j+1}(t)$ be the transfer  matrix associated to $A(t)$, i.e.  $G_{j+1}(t)$ is the real $2 \times 2$ matrix
solving the ODE
$$ \partial_t G_{j+1}(t) = A(t) G_{j+1}(t); \quad G_{j+1}(0) = \id.$$
Then from the Lemma \ref{nilgronwall}, the coefficients of $G_{j+1}$ enjoy the bounds
\begin{equation}\label{gj-transfer}
G_{j+1}(t) = \left( \begin{array}{ll}
O( 1+t ) & O(t) \\
O( 1 )   & O(1)
\end{array} \right ).
\end{equation}
Since $\sigma$ is small, we can also use the last part of Lemma \ref{nilgronwall} and conclude that the coefficient in the bottom right corner has magnitude at least $1/2$.  (This will be important later when we ``invert'' $G_{j+1}(T)$.)

Now we return to \eqref{aj0}, and use the ansatz
$$ \tilde a_{j+1}(t) = G_{j+1}(t) a_{j+1} + E_{j+1}(t)$$
to obtain an equation for the error $E_{j+1}$:
$$
\partial_t E_{j+1}(t) = A(t) E_{j+1}(t) + O( T^{A_j^-+1} e^{-2\sqrt{3} T} e^{\sqrt{3} t} ); \quad E_{j+1}(0) = O( T^{A_j^-} e^{-\sqrt{3} T} )$$
and then by Lemma \ref{nilgronwall} again we obtain the bounds
$$ E_{j+1}(t) = O( T^{A_j^-+2} e^{-2\sqrt{3} T} e^{\sqrt{3} t} )+ O(T^{A_j^-+1} e^{-\sqrt{3} T})  = O( T^{A_j^-+2} e^{-\sqrt{3} T} ).$$
We thus conclude that
\begin{equation}\label{cj1-arg}
 \left( \begin{array}{l}
e^{2\sqrt{3} T} e^{\sqrt{3} t} c_{j+1}^-(t) \\
e^{2\sqrt{3} T} e^{-\sqrt{3} t} c_{j+1}^+(t)
\end{array} \right) = G_{j+1}(t) a_{j+1} + O( T^{A_j^-+2} e^{-\sqrt{3} T}  ).
\end{equation}

\subsection{Hitting the Ricochet Target}

Our estimates are now sufficiently accurate to show that $(M^-_j, d^-_j, R^-_j)$ can cover $(M^0_j, d^0_j, R^0_j)$.
Consider an arbitrary point $x^0_j$ in $M^0_j$, which in coordinates would take the form
\begin{align*}
c_{\leq j-1} = c_{j+1}^- &= 0\\
c_{j+1}^+ = z_{j+1}^+ e^{-\sqrt{3} T} \\
c_{\geq j+2} = z_{\geq j+2} e^{-2\sqrt{3} T}
\end{align*}
for some $z_{j+1}^+ \in \R$, $z_{\geq j+2} \in \C^{N-j-1}$ of magnitude at most $r_j^0$.  We need to locate a point $x^-_j$ in $M^-_j$, which in coordinates
takes the form
\begin{align*}
c_{\leq j-2} = c_{j-1}^+ &= 0\\
c_{j-1}^- &= \sigma \\
c_{j+1}^- &= a_{j+1}^- e^{-2\sqrt{3} T} \\
c_{j+1}^+ &= a_{j+1}^+ e^{-2\sqrt{3} T} \\
c_{\geq j+2} &= a_{\geq j+2} e^{-2\sqrt{3} T}
\end{align*}
for some $a_{j+1}^\pm \in \R$, $a_{\geq j+2} \in \C^{N-j-1}$ of magnitude at most $r_j^-/4$ (say) to be chosen later, such that given any data $c(0)$ which is within $R_j^- = T^{A_j^-}$ of $x^-_j$, thus in coordinates
\begin{align*}
c_{\leq j-2}(0)&= O(T^{A_j^-} e^{-2\sqrt{3} T} )\\
c_{j-1}^-(0) &= \sigma + O( T^{A_j^-} e^{-\sqrt{3} T}) \\
c_{j-1}^+(0) &= O( T^{A_j^-} e^{-4 \sqrt{3} T} ) \\
c_{j+1}^-(0) &= a_{j+1}^- e^{-2\sqrt{3} T} + O( T^{A_j^-} e^{-3\sqrt{3} T} ) \\
c_{j+1}^+(0) &= a_{j+1}^+ e^{-2\sqrt{3} T} + O( T^{A_j^-} e^{-3\sqrt{3} T} )\\
c_{\geq j+2}(0) &= a_{\geq j+2} e^{-2\sqrt{3} T} + O( T^{A_j^-} e^{-3\sqrt{3} T} ),
\end{align*}
the evolution of this data  after time $T$ will lie within $R^0_j = T^{A^0_j}$ of $x^0_j$ in the $d^0_j$ metric.  Thus we aim to show,
\begin{equation}\label{bullet}
\begin{split}
&e^{2\sqrt{3} T} |c_{\leq j-2}(T)| + e^{\sqrt{3} T} |c_{j-1}^-(T)| +  e^{3\sqrt{3} T} |c_{j-1}^+(T)| + e^{3\sqrt{3} T} |c_{j+1}^-(T)| \\
&\quad
+  e^{2\sqrt{3} T} |c_{j+1}^+(T) - e^{-\sqrt{3} T} z_{j+1}^+| +
e^{3\sqrt{3} T} |c_{\geq j+2}(T) - e^{-2\sqrt{3} T} z_{\geq j+2}| < T^{A^0_j}.
\end{split}
\end{equation}
To establish this, we of course apply the bounds obtained in this section.  From Proposition \ref{growthprop} we have
$$ |c_{\leq j-2}(T)| = O( T^{A_j^-} e^{-2\sqrt{3} T} ); \quad
c_{j-1}^-(T) = O( e^{-\sqrt{3} T} ); \quad c_{j-1}^+(T) = O( T^{2A_j^- +1} e^{-3\sqrt{3} T})$$
and hence the contribution of the trailing modes $c_{\leq j-1}$
to \eqref{bullet} will be acceptable (recall that $A_j^0$ is ten times larger than $A_j^-$).
From \eqref{cj2-bound} we have
$$ c_{\geq j+2}(t) = e^{-2\sqrt{3} T} G_{\geq 2}(T) a_{\geq j+2} + O( T^{A_j^-} e^{-3\sqrt{3} T} ).$$
Thus if we set $a_{\geq j+2} := G_{\geq 2}(T)^{-1} z_{\geq j+2}$,
then the contribution of the leading peripheral modes $c_{\geq j+2}$ to \eqref{bullet} is acceptable.
Note that since $|z_{\geq j+2}| \leq r_j^0$, then $|a_{\geq j+2}| \leq r_j^-$
thanks to \eqref{gdither} and the construction of $r_j^-$ large compared to $r_j^0$.

Finally, we need to deal with the leading secondary modes.  According to
Proposition \ref{growthprop} the contribution of $c_{j+1}^-(t)$ to the left
hand side of \eqref{bullet} is acceptable.  As for
$c_{j+1}^+$, from \eqref{cj1-arg} we have
$$
 \left( \begin{array}{l}
e^{3\sqrt{3} T} c_{j+1}^-(T) \\
e^{\sqrt{3} T} c_{j+1}^+(T)
\end{array} \right) = G_{j+1}(T) a_{j+1} + O( T^{A_j^-+2} e^{-\sqrt{3} T}  ).
$$
Now recall that the matrix $G_{j+1}(T)$ has the form \eqref{gj-transfer}, with the bottom right coefficient having magnitude comparable to $1$.
Because of this, and the hypothesis that $|z_{j+1}^+| \leq r^0_j$, one can easily find coefficients $a_{j+1}^-, a_{j+1}^+$ of magnitude at most
$r^-_j$ (which is large compared to $r^0_j$) such that
$$
G_{j+1}(T) \left( \begin{array}{l} a_{j+1}^- \\ a_{j+1}^+ \end{array} \right) =
\left( \begin{array}{l} \ldots \\ z_{j+1}^+ \end{array} \right)$$
where the exact value of the coefficient $\ldots$ is not important to us since we already bounded
$c_{j+1}^-(T)$ above.  We thus have
$$ c_{j+1}^+(T) = e^{-\sqrt{3} T} z_{j+1}^+ + O( T^{A_j^-+2} e^{-2\sqrt{3} T}  )$$
and conclude that the contribution of this term to \eqref{bullet} is also acceptable.  This concludes the proof
that $(M^-_j, d^-_j, r^-_j)$ covers $(M^0_j, d^0_j, r^0_j)$.

\subsection{Flowing from the Ricochet Target to the Outgoing Target}\label{outgoing-sec}

Again we fix $3 \leq j \leq N-2$.
We now show that $(M^0_j, d^0_j, r^0_j)$ covers $(M^+_j, d^+_j, r^+_j)$.  Broadly speaking,
this will resemble a time-reversed version
of the arguments in Section \ref{incoming-sec}, though there are a number of technical
differences.  On the one hand, the trailing secondary modes
are now much less important and do not require as delicate a treatment as in Section \ref{incoming-sec}.  On the other hand, the time reversal changes the role of $\sigma$; instead of starting with a stable trailing mode at size $\sigma$ and ensuring that it decays in a controlled manner, we now are starting with an unstable leading mode $c^+_{j+1}$ of small size (like $e^{-\sqrt{3} T}$)
 and growing it so that it reaches $\sigma$ almost exactly at time $T$.

\begin{table}[ht]
\caption{The evolution of a solution which is within the ricochet target 
$(M^0_j,d^0_j,R^0_j)$
at time $t=0$, and with a well-chosen (and small) leading unstable mode $c_{j+1}^+$.
The $a_{\geq 2}, a_{j+1}^+$ are
explicit numbers defined below in \eqref{starpage23}, \eqref{starpage25}.
The function $\tilde g(t)$ and the transfer matrix $\tilde G_{\geq j+2}(t)$ are
explicit objects (depending only on $\sigma$) that will be defined below.
The $t=0$ configuration of the modes within the ricochet target 
is illustrated in Figure \ref{fig:MjRicochetTarget}. At $t=T$,
the modes are within the outgoing target as illustrated in Figure
\ref{fig:MjOutgoingTarget}.}\label{outtable}
\begin{tabular}{|l|l|l|l|}
\hline
Mode & $t=0$ & $0 < t < T$ & $t=T$  \\
\hline
$c_{\leq j-2}$ & $O( T^{A_j^0} e^{-2\sqrt{3} T} )$         & $O( T^{A_j^0} e^{-2\sqrt{3} T} )$                  & $O( T^{A_j^0} e^{-2\sqrt{3} T} )$ \\
\hline
$c_{j-1}^-$    & $O( T^{A_j^0} e^{-\sqrt{3} T} )$          & $O( T^{A_j^0+1} e^{-\sqrt{3} T}
e^{-\sqrt{3} t} )$  & $O( T^{A_j^0+1} e^{-2\sqrt{3} T} )$  \\
\hline
$c_{j-1}^+$    & $O( T^{A_j^0} e^{-3\sqrt{3} T} )$         & $O( T^{A_j^0+1} e^{-3\sqrt{3} T} e^{\sqrt{3} t})$ & $O( T^{A_j^0+1} e^{-2\sqrt{3} T} )$ \\
\hline
$c_{j+1}^-$    & $O( T^{A_j^0} e^{-3\sqrt{3} T} )$
               & $O( T^{2A_j^0+3} e^{-3\sqrt{3} T} e^{-\sqrt{3} t} )$
               & $O( T^{2A_j^0+3} e^{-4\sqrt{3} T} )$ \\
\hline
$c_{j+1}^+$    & $\sigma e^{-\sqrt{3} T} a^+_{j+1}$
               & $e^{-\sqrt{3} T} \tilde g(t)$
               & $\sigma$               \\
&\quad $+ O( T^{A_j^0} e^{-2\sqrt{3} T} )$ & \quad  $+ O( T^{A_j^0+2} e^{-2\sqrt{3} T}
e^{\sqrt{3} t} )$ & \quad $+ O(  T^{A_j^0+2} e^{-\sqrt{3} T} )$ \\
\hline
$c_{\geq j+2}$ & $e^{-2\sqrt{3} T} a_{\geq j+2}$
               & $e^{-2\sqrt{3} T} \tilde G_{\geq j+2}(t) a_{\geq j+2}$
               & $e^{-2\sqrt{3} T} \tilde G_{\geq j+2}(T) a_{\geq j+2}$ \\
              & \quad $+ O( T^{A^0_j}e^{-3\sqrt{3} } )$ & \quad $+ O( T^{A^0_j} e^{-3\sqrt{3} } )$ &
	      \quad $+ O( T^{A^0_j} e^{-3\sqrt{3} } )$\\
\hline
\end{tabular}
\end{table}

Once again, we begin with upper bounds on the flow, though now we also need a smallness condition on the unstable mode $c_{j+1}^+(0)$
(to stop the evolution from moving too far away from $\T_j$ by time $T$).

\begin{proposition}[Upper bounds, outbound leg]\label{growthprop-out}
Let $b(t)$ be a solution to \eqref{eq:b-eq} such that $b(0)$ is within $(M^0_j, d^0_j,R^0_j)$.
Let $c(t)$ denote the coordinates of $b(t)$ as in Proposition \ref{coord}.  Assume the smallness condition
$c_{j+1}^+(0) = O( \sigma e^{-\sqrt{3} T} )$.  Then we have the bounds
\begin{align*}
|c_*(t)|       &= O( T^{A_j^0} e^{-2\sqrt{3} T} ) \\
|c_{j-1}^-(t)| &= O( T^{A_j^0} e^{-\sqrt{3} T} e^{-\sqrt{3} t} (1 + T e^{-\sqrt{3}(T-t)}) )
\leq O( T^{A_j^0+1} e^{-\sqrt{3} T} e^{-\sqrt{3} t}) \\
|c_{j-1}^+(t)| &= O( T^{A_j^0+1} e^{-3\sqrt{3} T} e^{\sqrt{3} t} ) \\
|c_{j+1}^-(t)| &= O( T^{2A_j^0+3} e^{-3\sqrt{3} T} e^{-\sqrt{3} t} )\\
|c_{j+1}^+(t)| &= O( \sigma e^{-\sqrt{3} T} e^{\sqrt{3} t} )
\end{align*}
for $0 \leq t \leq T$.
\end{proposition}



\begin{proof}
From the hypotheses we easily verify that the bounds hold at time $t=0$.  As in the proof of Proposition \ref{growthprop}, we use
the continuity method.  We again let $C_0$ be a large constant (depending only on $N$)
to be chosen later, and suppose that
$0 \leq T' \leq T$ is a time for which the bounds
\begin{align*}
|c_*(t)|       &= O( C_0 T^{A_j^0} e^{-2\sqrt{3} T} ) \\
|c_{j-1}^-(t)| &= O( C_0 T^{A_j^0} e^{-\sqrt{3} T} e^{-\sqrt{3} t} (1 + T e^{-\sqrt{3}(T-t)}) ) \\
|c_{j-1}^+(t)| &= O( C_0 T^{A_j^0+1} e^{-3\sqrt{3} T} e^{\sqrt{3} t} ) \\
|c_{j+1}^-(t)| &= O( C_0 T^{2A_j^0+3} e^{-3\sqrt{3} T} e^{-\sqrt{3} t} )\\
|c_{j+1}^+(t)| &= O( C_0 \sigma e^{-\sqrt{3} T} e^{\sqrt{3} t} )
\end{align*}
are known to hold for all $0 \leq t \leq T'$.  As before, it suffices to then deduce the same bounds without the $C_0$ factor.

From the above bounds we have
$$ |c(t)| = O( C_0 T^{A_j^0+1} e^{-\sqrt{3} T} e^{-\sqrt{3} t} + C_0 \sigma e^{-\sqrt{3} T} e^{\sqrt{3} t} )$$
which gives the bound \eqref{cl2-bound} (taking $\sigma$ small compared to $C_0$.  Lemma \ref{ubound} now applies.  From \eqref{cstar-bound} we have
$$ |c_*(t)| \lesssim |c_*(0)| \lesssim T^{A_j^0} e^{-2\sqrt{3} T} $$
which gives the desired control on the peripheral modes $c_*$.  From \eqref{ctotal-bound} we have
$$ |c_{j+1}(t)| \lesssim e^{\sqrt{3} t} |c_{j+1}(0)| \lesssim \sigma e^{-\sqrt{3} T} e^{\sqrt{3} t} $$
which gives the desired control on the unstable leading mode $c_{j+1}^+$.  Next, we apply \eqref{cminus-bound} to obtain
$$ |c_{j-1}^-(t)| \lesssim e^{-\sqrt{3} t} c_{j-1}^-(0) +
\int_0^t e^{\sqrt{3}(t-s)} |c_{j-1}^+(s)| |c(s)|^2\ ds$$
which after substituting the initial bound $c_{j-1}^-(0) = O( T^{A_j^0} e^{-\sqrt{3} T})$ and the above bounds on $c_{j-1}^+(s)$, $c(s)$ gives
$$ |c_{j-1}^-(t)| \lesssim T^{A_j^0} e^{-\sqrt{3} T} e^{-\sqrt{3} t} (1 + T e^{-\sqrt{3}(T-t)})$$
which is the desired bound in the stable trailing mode $c_{j-1}^-$.  Then, we apply \eqref{cplus-bound} to obtain
$$ |c_{j-1}^+(t)| \lesssim e^{-\sqrt{3} t} c_{j-1}^+(0) + \int_0^t e^{\sqrt{3}(t-s)} |c_{j-1}^-(s)| |c(s)|^2\ ds$$
which after substituting the initial bound $c_{j-1}^+(0) = O( T^{A_j^0} e^{-3\sqrt{3} T}  )$ and the above bounds
for $c_{j-1}^-$ and $c$ becomes
$$ |c_{j-1}^+(t)| \lesssim T^{A_j^0+1} e^{-3\sqrt{3} T} e^{\sqrt{3} t}$$
which is the desired bound on the unstable trailing mode $c_{j-1}^+$.  These bounds imply in particular that
$$ |c_{\neq j+1}(t)| \lesssim T^{A_j^0+1} e^{-\sqrt{3} T} e^{-\sqrt{3} t}$$
while from \eqref{cminus-bound} we have
$$ |c_{j+1}^-(t)| \lesssim e^{-\sqrt{3} t} c_{j+1}^-(0) +
\int_0^t e^{-\sqrt{3}(t-s)} |c_{j +1}^+(s)| |c_{\neq j + 1}^2(s)|\ ds.$$
Combining these bounds and also using the initial bound $c_{j+1}^-(0) = O( T^{A_j^0} e^{-3\sqrt{3} T} )$ and the bound already obtained for $c_{j+1}^+$,
we conclude
$$ |c_{j+1}^-(t)| \lesssim T^{2A_j^0+3} e^{-3\sqrt{3} T} e^{-\sqrt{3} t} $$
which is the desired bound on the leading stable mode $c_{j+1}^-$.
\end{proof}

Now we need more precise bounds on the leading modes.  Here we will assume the initial data takes the form
\begin{align*}
c_{j+1}^+(0) &= \sigma e^{-\sqrt{3} T} a_{j+1}^+ + O( T^{A_j^0} e^{-2\sqrt{3} T} )\\
c_{\geq j+2}(0) &= e^{-2\sqrt{3} T} a_{\geq j+2} + O( T^{A_j^0} e^{-3\sqrt{3} T} )
\end{align*}
for some data $a_{j+1}^+ \in \R$, $a_{\geq j+2} \in \C^{N-j-1}$ of magnitude at most $O(1)$ and $r_j^-/2$ respectively,
which we shall choose later.  Henceforth the time variable is restricted to the interval
 $0 \leq t \leq T$.

\paragraph{\bf Improved control of $c_{j+1}^+$}.
We begin by refining the control on the unstable leading mode.  From Proposition \ref{growthprop-out} we see that
any expression of the form $\bigO(c^3)$ splits as the sum of a term $\bigO( (c_{j+1}^+)^3 )$,
plus an error
of size $O( T^{3A_j^0+3} e^{-3\sqrt{3} T} e^{\sqrt{3} t} )$.  From \eqref{ceq-unstable} we thus have
$$
\partial_t c_{j + 1}^+ = \sqrt{3} c_{j + 1}^+ +
\bigO( (c_{j + 1}^+)^3 ) + O( T^{3A_j^0+3} e^{-3\sqrt{3} T} e^{\sqrt{3} t} ).$$
Let us now compare this against the function $\tilde g$, defined
as the solution of the associated equation
$$
\partial_t \tilde g = \sqrt{3} \tilde g + \bigO( \tilde g^3 ) $$
with initial data
\begin{align}
\label{starpage22}
\tilde g(T) & = \sigma,
\end{align}
 at time $T$.  If $\sigma$ is small, then an easy continuity argument\footnote{Alternatively, one could observe that $\tilde g$ is basically the time reflection of the function $g$ used in the preceding section.} \emph{backwards} in time
shows that
$$ |\tilde g(t)| \leq 2 \sigma e^{-\sqrt{3} T} e^{\sqrt{3} t}.$$
In particular we have
\begin{equation}\label{tgbound-2}
\int_0^T |\tilde g(t)|^2 \ dt = O(1).
\end{equation}
We now fix $a_{j+1}^+$ by requiring
\begin{align} \label{starpage23}
 \tilde g(0) &= \sigma e^{-\sqrt{3} T} a_{j+1}^+,
 \end{align}
thus $a_{j+1}^+ = O(1)$ as required.  If we then use the ansatz
$$ c_{j+1}^+ = \tilde g + E_{j+1}^+$$
then the error $E_{j+1}^+$ obeys the equation
$$ \partial_t E_{j+1}^+ = \sqrt{3} E_{j + 1}^+ + O( |\tilde g|^2 |E_{j+1}| )
+ O( |E_{j+1}|^3 ) + O( T^{3A_j^0+3} e^{-3\sqrt{3} T} e^{\sqrt{3} t} )$$
with initial data $E_{j+1}^+(0) = O( T^{A_j^0} e^{-2\sqrt{3} T} )$.  A simple application of the continuity method and Gronwall's inequality (and \eqref{tgbound-2}) then yields the bound
$$ E_{j+1}^+(t) = O( T^{A_j^0} e^{-2\sqrt{3} T} e^{\sqrt{3} t} )$$
for all $0 \leq t \leq T$.  We thus conclude that
\begin{equation}\label{cjp-asym}
c_{j+1}^+(t) = \tilde g(t) + O( T^{3A_j^0 + 3} e^{-2\sqrt{3} T} e^{\sqrt{3} t} )
\end{equation}
which then (in conjunction with the bounds in Proposition \ref{growthprop-out} implies that
\begin{equation}\label{c2-asym}
\bigO( c^2 ) = \bigO( \tilde g^2 ) + O( T^{2A_j^0+2} e^{-2\sqrt{3} T} ).
\end{equation}

\paragraph{\bf Improved control of $c_{\geq j+2}$}  

We now control the leading peripheral modes, by essentially the same argument used to control these modes in the previous section.
From \eqref{ck-eq} and \eqref{c2-asym} we have
$$ \partial_t c_{\geq j+2} = i c_{\geq j+2} + \bigO( \tilde g^2 c_{\geq j+2} ) + O( T^{2A_j^0+1} e^{-2\sqrt{3} T} |c_{\geq j+2}| ).$$
We approximate this by the corresponding linear equation
$$ \partial_t u = i u + \bigO( u \tilde g^2 )$$
where $u(t) \in \C^{M-j-1}$.  This equation has a fundamental solution
$\tilde G_{\geq j+2}(t): \C^{M-j-1} \to \C^{M-j-1}$ for all $t \geq 0$, thus $G_{\geq j+2}(t) u(0) = u(t)$.
From \eqref{tgbound-2} and Gronwall's inequality we have
\begin{equation}\label{tgdither}
 | \tilde G_{\geq j+2}(t) |, |\tilde G_{\geq j+2}(t)^{-1} | = O(1).
\end{equation}
We use the ansatz
$$ c_{\geq j+2}(t) =  e^{-2\sqrt{3} T} \tilde G_{\geq j+2}(t) a_{\geq j+2} + \tilde E_{\geq j+2}.$$
The error $\tilde E_{\geq j+2}$ then solves the equation
$$ \partial_t \tilde E_{\geq j+2} = i\tilde E_{\geq j+2} + \bigO( \tilde E_{\geq j+2} \tilde g^2 ) +
O( T^{2A_j^0+2} e^{-2\sqrt{3} T} |c_{\geq j+2}| )$$
with initial data $\tilde E_{\geq j+2}(0) = O( T^{A_j^0} e^{-3\sqrt{3} T} )$.  Applying the bound on $c_{\geq j+2}$ from Proposition \ref{growthprop-out}
we see that
$$ \partial_t |\tilde E_{\geq j+2}| = O( |\tilde E_{\geq j+2}| |\tilde g|^2 ) + O( T^{3A_j^0+2} e^{-4\sqrt{3} T} ).$$
From Gronwall's inequality and \eqref{tgbound-2} we conclude that
$$ |\tilde E_{\geq j+2}(t)| = O( T^{3A_j^0+2} e^{-3\sqrt{3} T} )$$
for all $0 \leq t \leq T$, and thus
\begin{equation}\label{tcj2-bound}
c_{\geq j+2}(t) = e^{-2\sqrt{3} T} \tilde G_{\geq j+2}(t) a_{\geq j+2} + O( T^{A_j^0} e^{-3\sqrt{3} T} ).
\end{equation}

\subsection{Hitting the Outgoing Target}

Our estimates are now sufficiently accurate to show that $(M^0_j, d^0_j, R^0_j)$ can cover $(M^+_j, d^+_j, R^+_j)$.
Consider an arbitrary point $x^+_j$ in $M^+_j$, which in coordinates would take the form
\begin{align*}
c_{\leq j-1} = c_{j+1}^- &= 0\\
c_{j+1}^+ = \sigma \\
c_{\geq j+2} = z_{\geq j+2} e^{-2\sqrt{3} T}
\end{align*}
for some $z_{\geq j+2} \in \C^{N-j-1}$ of magnitude at most $r_j^+$.  We now specify a
point $x^0_j$ in $M^0_j$, which in coordinates
has the form
\begin{align*}
c_{\leq j-1} = c_{j+1}^- &= 0\\
c_{j+1}^+ &= \sigma  e^{-\sqrt{3} T} a_{j+1}^+  \\
c_{\geq j+2} &= a_{\geq j+2} e^{-2\sqrt{3} T}
\end{align*}
with $a_{j+1}^+ = O(1)$ defined in \eqref{starpage23} and some
$a_{\geq j+2} \in \C^{N-j-1}$ of magnitude at most $r_j^0/2$ to be chosen later (see
\eqref{starpage25} below), such that given any data $c(0)$ which is within $R_j^0 = T^{A_j^0}$ of
$x^0_j$ in the metric $d_j^0$, thus in coordinates
\begin{align*}
c_{\leq j-2}(0) &= O(T^{A_j^0} e^{-2\sqrt{3} T} )\\
c_{j-1}^-(0) &= O( T^{A_j^0} e^{-\sqrt{3} T} )\\
c_{j-1}^+(0) &= O( T^{A_j^0} e^{-3\sqrt{3} T} )\\
c_{j+1}^-(0) &= O( T^{A_j^0} e^{-3\sqrt{3} T} )\\
c_{j+1}^+(0) &= \sigma  e^{-\sqrt{3} T} a_{j+1}^+ + O( T^{A_j^0} e^{-2\sqrt{3} T} ) \\
c_{\geq j+2}(0) &= a_{\geq j+2} e^{-2\sqrt{3} T} + O( T^{A_j^0} e^{-3\sqrt{3} T} )
\end{align*}
the evolution of this data  at time $T$ will lie within $R^+_j = T^{A^+_j}$ of
$x^+_j$ in the $d^+_j$ metric.  Thus we aim to show,
\begin{equation}\label{bullet-out}
\begin{split}
e^{2\sqrt{3} T} &|c_{\leq j-1}(T)| + e^{4\sqrt{3} T} |c_{j+1}^-(T)| + e^{\sqrt{3} T} |c_{j+1}^+(T) - \sigma| \\
&\quad
+ e^{3\sqrt{3} T} |c_{\geq j+2}(T) - e^{-2\sqrt{3} T} z_{\geq j+2}| < T^{A^+_j}.
\end{split}
\end{equation}
We argue as in the previous section.  From
Proposition \ref{growthprop-out} we have
\begin{align*}
c_{\leq j-2}(T) &= O( T^{A_j^0} e^{-2\sqrt{3} T} ) \\
c_{j-1}^-(T) &= O( T^{A_j^0+1} e^{-2\sqrt{3} T} ) \\
c_{j-1}^+(T) &= O( T^{A_j^0 +1} e^{-2\sqrt{3} T} ) \\
c_{j+1}^-(T) &= O( T^{2A_j^0+3} e^{-4\sqrt{3} T} )
\end{align*}
which shows that the contribution of the trailing modes and the leading
stable mode will be acceptable.  From \eqref{cjp-asym}, \eqref{starpage22} we have
$$ c_{j+1}^+(T) = \sigma + O( T^{3A_j^0 +3} e^{-\sqrt{3} T} )$$
so the contribution of the leading unstable mode will also be acceptable.  Finally, from \eqref{tcj2-bound} we have
$$ c_{\geq j+2}(T) = e^{-2\sqrt{3} T} \tilde G_{\geq j+2}(T) a_{\geq j+2} + O( T^{A_j^0} e^{-3\sqrt{3} T} )$$
so if we choose
\begin{align}
\label{starpage25}
 a_{\geq j+2} &= \tilde G_{\geq j+2}(T)^{-1} z_{\geq j+2}
 \end{align}
then from \eqref{tgdither} we see that $a_{\geq j+2}$ will have magnitude $O(r^+_j) \leq r^0_j / 2$, and the contribution of
the leading peripheral modes will also be acceptable.  This completes the proof that
$(M^0_j, d^0_j, r^0_j)$ covers $(M^+_j, d^+_j, r^+_j)$.

\subsection{Flowing from the Outgoing Target to the Next Incoming Target}\label{transit-sec}


Fix $3 \leq j < N-2$.
To conclude the proof of Theorem \ref{arnold} we need to show that $(M^+_j, d^+_j, r^+_j)$ covers $(M^-_{j+1}, d^-_{j+1}, r^-_{j+1})$.  This turns out
to be significantly simpler than the previous analysis because we will only need to flow\footnote{Indeed, if we could take $\sigma$ as large as $1/\sqrt{2}$ then we would not need to flow at all, and we could essentially match up the $j^{th}$ outgoing target with the $j+1^{th}$ incoming target.  However we took advantage of the smallness of $\sigma$ at too many places in the above argument, and so we are forced to add this bridging step as well.} for a time $O(\log \frac{1}{\sigma})$ rather than time $O(T)$.  This means that we can rely on much cruder Gronwall inequality-type tools
than in preceding sections as we do not have to be so careful about exponential or even polynomial losses in $T$.
On the other hand, the analysis here is more ``nonperturbative'' in that we are no longer close to a circle $\T_j$ or $\T_{j+1}$ but instead will
be traversing the intermediate region connecting the two.  Fortunately, we have an explicit\footnote{Actually, as in previous sections, the analysis can proceed without knowing the explicit form of this solution, only its qualitative properties, namely that it propagates from $\T_j$ to $\T_{j+1}$ using an unstable mode of the former and a stable mode of the latter.  But as the solution is so simple, we will take advantage of its explicitness.} solution (based on \eqref{eq:slider}) that we can follow (via standard perturbative theory) to carry us from one to the other without much difficulty.

We turn to the details.  Before we begin, there is an issue of which coordinate system to use: the local coordinates around $\T_j$, the local coordinates around $\T_{j+1}$, or the global coordinates $b_1,\ldots, b_N$.  It is arguably more natural to use global coordinates for the
transition from $\T_j$ to $\T_{j+1}$, but we will continue using the local coordinates around $\T_j$ since we have already built a fair amount of
machinery in these coordinates.  Also, these local coordinates have already quotiented out the phase invariance $x \mapsto e^{i\theta} x$ which would otherwise have required a small amount of attention.

Let us start with initial data within the outgoing target $(M^+_j, d^+_j, r^+_j)$.  In the local coordinates around $\T_j$, such an initial data takes the form
\begin{align*}
c_{\leq j-1}(0) &= O( T^{A^+_j} e^{-2\sqrt{3} T} ) \\
c_{j+1}^-(0) &= O( T^{A^+_j} e^{-4\sqrt{3} T} ) \\
c_{j+1}^+(0) &= \sigma + O( T^{A^+_j} e^{-\sqrt{3} T} ) \\
c_{\geq j+2}(0) &= e^{-2\sqrt{3} T} a_{\geq j+2} + O( T^{A^+_j} e^{-3\sqrt{3} T} )
\end{align*}
for some $a_{\geq j+2}$ of magnitude at most $r^+_j$.  Let us now consider the evolution of such data for times $0 \leq t \leq O( \log \frac{1}{\sigma} )$.  Because the unstable leading mode $c_{j+1}^+$ is already as large as $\sigma$, and is growing exponentially, we will no longer be in
the perturbative regime covered by Lemma \ref{ubound}.  However, we must necessarily stay within the region \eqref{cbound}; note that the coordinate singularity $\{ b_j = 0 \}$ cannot actually be reached via the flow, because as remarked earlier the support of $b$ is an invariant of the flow.
The bound \eqref{cbound} will still allow us to use Gronwall-type arguments for a time period of $O( \log \frac{1}{\sigma} )$, incurring (quite tolerable) losses which are polynomial in $1/\sigma$.  (This is in contrast to the analysis of the previous sections, where such a crude argument would cost unacceptable factors of $e^T$.)

For the rest of this section we shall restrict the time variable $t$ to the 
range $0 \leq t \leq 10 \log\frac{1}{\sigma}$ (say).
From \eqref{ck-eq}, \eqref{cbound} (and \eqref{cpm-eq} for the trailing secondary mode) we have the crude estimates
$$ \partial_t c_{\neq j+1} = O( |c_{\neq j+1}| )$$
so by Gronwall we have the crude bounds
\begin{equation}\label{cj-crude}
c_{\neq j+1}(t) = O( \frac{1}{\sigma^{O(1)}} T^{A^+_j} e^{-2\sqrt{3} T} ).
\end{equation}
We will return to establish more accurate bounds on the leading peripheral modes $c_{\geq j+2}$ shortly, but for now let us focus attention on the leading secondary modes $c_{j+1}^-, c_{j+1}^+$.  The stable leading mode $c_{j+1}^-$ can be
controlled by \eqref{ceq-stable}, which by \eqref{cj-crude} and \eqref{cbound} becomes
$$ \partial_t c_{j + 1}^- = O( |c_{j+1}^-| ) +  \frac{1}{\sigma^{O(1)}} T^{2A^+_j} e^{-4\sqrt{3} T}.$$
From Gronwall's inequality we conclude that
\begin{equation}\label{cjp-tiny}
 c_{j+1}^-(t) = O( \frac{1}{\sigma^{O(1)}} T^{2A^+_j} e^{-4\sqrt{3} T} ).
 \end{equation}
Now we turn to the most important non-primary mode, namely the unstable leading mode $c_{j+1}^+$.
Because we are no longer in the perturbative regime, the schematic equation \eqref{ceq-unstable} has a form which is a little bit too
crude for our purposes.  Instead we return to \eqref{cpm-eq}.  Taking advantage of the bounds \eqref{cjp-tiny}, \eqref{cj-crude} just obtained (as well as \eqref{cbound}), we can
take the $c_{j+1}^+$ component of this equation and obtain
$$ \partial_t c_{j+1}^+ = \sqrt{3} (1 - |c_{j+1}^+|^2) c_{j+1}^+ + O( \frac{1}{\sigma^{O(1)}} T^{2A^+_j} e^{-4\sqrt{3} T} ).$$
Now let us define $\hat g$ to be the solution to the scalar ODE
\begin{equation}\label{hatg} \partial_t \hat g = \sqrt{3} (1 - |\hat g|^2) \hat g
\end{equation}
with initial data $\hat g(0) = \sigma$.  We can easily compute $\hat g$ explicitly\footnote{For our analysis, the only property one needs of $\hat g$ is that it flows from $\sigma$ to $\sqrt{1-\sigma^2}$ in finite time (the fact that this time is $O( \log\frac{1}{\sigma} )$ is not essential to the argument).  This is immediate from an inspection of the vector field associated to the ODE \eqref{hatg}, but the argument via the explicit solution is equally brief.} as
$$ \hat g(t) = \frac{1}{\sqrt{1 + e^{-2\sqrt{3}(t-t_0)}}}$$
where the time $t_0$ is defined by the formula
$$ \frac{1}{\sqrt{1 + e^{2\sqrt{3} t_0}}} = \sigma.$$
This is of course closely related to the slider solution \eqref{eq:slider} (see also Remark \ref{slider-remark}).  Also observe that
$$ \hat g(2t_0) = \frac{1}{\sqrt{1 + e^{-2\sqrt{3} t_0}}} = \sqrt{1 - \sigma^2}$$
and that $2t_0 \leq 10 \log \frac{1}{\sigma}$ if $\sigma$ is small enough.

As in previous sections, we now use the ansatz
$$ c_{j+1}^+ = \hat g + E_{j+1}^+$$
where (thanks to the boundedness of both $\hat g$ and $c_{j+1}^+$) the error $E_{j+1}^+$ obeys the equation
$$ \partial_t E_{j+1}^+ = O( |E_{j+1}^+| ) + O( \frac{1}{\sigma^{O(1)}} T^{2A^+_j} e^{-4\sqrt{3} T} )$$
with initial data $E_{j+1}^+(0) = O( T^{A^+_j} e^{-\sqrt{3} T} )$.  From Gronwall we thus have
$$ E_{j+1}^+(t) = O( \frac{1}{\sigma^{O(1)}} T^{A^+_j} e^{-\sqrt{3} T} )$$
and hence
\begin{equation}\label{cjp-transit}
 c_{j+1}^+(t) = \hat g(t) + O( \frac{1}{\sigma^{O(1)}} T^{A^+_j} e^{-\sqrt{3} T} ).
\end{equation}
In particular this (together with \eqref{cjp-tiny}, \eqref{cj-crude}) implies that
\begin{equation}\label{c2-transit}
\bigO( c^2 ) = \bigO( \hat g^2 ) + O( \frac{1}{\sigma^{O(1)}} T^{A^+_j} e^{-\sqrt{3} T} ).
\end{equation}
Now we can return to improve the control on the leading peripheral modes $c_{\geq j+2}$.  From \eqref{c2-transit}, \eqref{cj-crude} and \eqref{ck-eq} we have
$$ \partial_t c_{\geq j+2} = i c_{\geq j+2} + \bigO( \hat g^2 c_{\geq j+2} ) + O( \frac{1}{\sigma^{O(1)}} T^{2A^+_j} e^{-3\sqrt{3} T} ).$$
As in previous sections we approximate this flow by the linear equation
$$ \partial_t u = i u + \bigO( u \hat g^2 )$$
where $u(t) \in \C^{N-j-1}$.  This equation has a fundamental solution
$\hat G_{\geq j+2}(t): \C^{N-j-1} \to \C^{N-j-1}$ for all $t \geq 0$; from the boundedness of $\hat g$ and Gronwall's inequality we have
\begin{equation}\label{tgdither-transit}
 | \hat G_{\geq j+2}(t) |, |\hat G_{\geq j+2}(t)^{-1} | \lesssim \frac{1}{\sigma^{O(1)}}.
\end{equation}
We use the ansatz
$$ c_{\geq j+2}(t) = e^{-2\sqrt{3} T} \hat G_{\geq j+2}(t) a_{\geq j+2} + E_{\geq j+2}$$
where the error $E_{\geq j+2}$ obeys an equation of the form
$$ \partial_t E_{\geq j+2} = O( |E_{\geq j+2}| ) + O( \frac{1}{\sigma^{O(1)}} T^{2A^+_j} e^{-3\sqrt{3} T} )$$
with initial data $E_{\geq j+2}(0) = O( T^{A^+_j} e^{-3\sqrt{3} T} )$.  From Gronwall's inequality we conclude
$$ E_{\geq j+2}(t) = O( \frac{1}{\sigma^{O(1)}} T^{2A^+_j} e^{-3\sqrt{3} T} )$$
and hence
\begin{equation}\label{cj2-transit}
c_{\geq j+2}(t) = e^{-2\sqrt{3} T} \hat G_{\geq j+2}(t) a_{\geq j+2} + O( \frac{1}{\sigma^{O(1)}} T^{2A^+_j} e^{-3\sqrt{3} T} ).
\end{equation}

We specialize the above bounds to the time $t = 2t_0 \leq 10 \log \frac{1}{\sigma}$, and conclude that
\begin{align*}
c_{\leq j-1}(2t_0) &= O( \frac{1}{\sigma^{O(1)}} T^{A^+_j} e^{-2\sqrt{3} T} ) \\
c_{j+1}^-(2t_0) &= O( \frac{1}{\sigma^{O(1)}} T^{2A^+_j} e^{-4\sqrt{3} T} ) \\
c_{j+1}^+(2t_0) &= \sqrt{1-\sigma^2} + O( \frac{1}{\sigma^{O(1)}} T^{A^+_j} e^{-\sqrt{3} T} ) \\
c_{\geq j+2}(2t_0) &= e^{-2\sqrt{3} T} \hat G_{\geq j+2}(2t_0) a_{\geq j+2} + O( \frac{1}{\sigma^{O(1)}} T^{2A^+_j} e^{-3\sqrt{3} T} )
\end{align*}
Using \eqref{rform}, we conclude that
$$ r = \sigma + O( \frac{1}{\sigma^{O(1)}} T^{A^+_j} e^{-\sqrt{3} T} ).$$
The phase $\theta$ is unspecified, but this will not concern us.  We can now move back to the global coordinates $b_1,\ldots,b_N$ and conclude that
\begin{align*}
b_{\leq j-1}(2t_0) &= O( \frac{1}{\sigma^{O(1)}} T^{A^+_j} e^{-2\sqrt{3} T} ) \\
b_j(2t_0) &= (\sigma + \Re O( \frac{1}{\sigma^{O(1)}} T^{A^+_j} e^{-\sqrt{3} T} ))  e^{i\theta} \\
b_{j+1}(2t_0) &= (\sqrt{1-\sigma^2} + \Re O( \frac{1}{\sigma^{O(1)}} T^{A^+_j} e^{-\sqrt{3} T} )) \omega^2 e^{i\theta} +
O( \frac{1}{\sigma^{O(1)}} T^{2A^+_j} e^{-4\sqrt{3} T} ) \\
b_{\geq j+2}(2t_0) &= e^{i\theta} e^{-2\sqrt{3} T} \hat G_{\geq j+2}(2t_0) a_{\geq j+2} + O( \frac{1}{\sigma^{O(1)}} T^{2A^+_j} e^{-3\sqrt{3} T} )
\end{align*}
where we have inserted real parts in front of some error terms for emphasis.

We now recast this in terms of the local coordinates around $\T_{j+1}$; to avoid confusion with the local coordinates $\T_j$, we shall denote these new coordinates with tildes, thus $\tilde r,\tilde \theta, \tilde c_{\leq j-1}, \tilde c_j^-, \tilde c_j^+, \tilde c_{j+2}^-, \tilde c_{j+2}^+, \tilde c_{\geq j+2}$.  Firstly, $b_{j+1}(2t_0)$ is certainly non-zero (it has magnitude close to $\sqrt{1-\sigma^2}$), and an inspection of the phase of $b_{j+1}(2t_0)$ shows that
$$ \tilde \theta(2t_0) = \theta + \frac{4\pi}{3} + O( \frac{1}{\sigma^{O(1)}} T^{2A^+_j} e^{-4\sqrt{3} T} ).$$
We then conclude that
\begin{align*}
\tilde c_{\leq j-1}(2t_0) &= O( \frac{1}{\sigma^{O(1)}} T^{A^+_j} e^{-2\sqrt{3} T} ) \\
\tilde c_j^-(2t_0) &= \sigma + O( \frac{1}{\sigma^{O(1)}} T^{A^+_j} e^{-\sqrt{3} T} ) \\
\tilde c_j^+(2t_0) &= O( \frac{1}{\sigma^{O(1)}} T^{2A^+_j} e^{-4\sqrt{3} T} ) \\
\tilde c_{\geq j+2}(2t_0) &=  \omega e^{-2\sqrt{3} T} \hat G_{\geq j+2}(2t_0) a_{\geq j+2} + O( \frac{1}{\sigma^{O(1)}} T^{2A^+_j} e^{-3\sqrt{3} T} ).
\end{align*}
From this and \eqref{tgdither-transit} it is an easy matter to show that $(M^-_{j+1}, d^-_{j+1}, R^-_{j+1})$ is covered by
$(M^0_j, d^0_j, R^0_j)$, by choosing $a_{\geq j+2}$ appropriately (note that any losses arising from \eqref{tgdither-transit} will be
acceptable since $r^+_j$ is assumed to be much larger than $r^-_{j+1}$ depending on $\sigma$).

\section{Construction of the Resonant Set $\Lambda$}\label{NumberTheory}

Fix $N \geq 2$.
We begin by constructing an abstract combinatorial model $\Sigma = \Sigma_1 \cup \ldots \cup \Sigma_N$
for the collection $\Lambda = \Lambda_1 \cup \ldots \cup \Lambda_N$ of frequencies.
Each element $x_0 \in \Sigma$  corresponds to a unique element $n_0 \in \Lambda$ and encodes
both how $n_0$ is related to other elements of $\Lambda$, and (at least approximately) where
$n_0$ is located in $\Z^2$.  The relationship between $\Sigma$ and $\Lambda$ will be made
explicit later when we construct an embedding of $\Sigma$ into $\Z^2$ and define $\Lambda$
to be this image.  Whereas $\Lambda$ lives in the frequency lattice $\Z^2$, $\Sigma$ will live in a more
abstract set ($\C^{N-1}$, to be precise).

To construct $\Sigma$, we define the \emph{standard unit square} $S \subset \C$ to be the four-element set of complex numbers
$$ S = \{ 0, 1, 1+i, i \}.$$
We split $S = S_1 \cup S_2$, where $S_1 := \{ 1,i \}$ and $S_2 := \{0, 1+i\}$.  Note that this is already an $N=2$ model for the $\Lambda$'s,
if we identify the frequency lattice $\Z^2$ with the Gaussian integers $\Z[i]$ in the usual manner $(n_1,n_2) \leftrightarrow n_1 + i n_2$.
To create $\Sigma$ we shall essentially take a large power of this unit square example.  More precisely, for any $1 \leq j \leq N$, we define
$\Sigma_j \subset \C^{N-1}$ to be the set of all $N-1$-tuples $(z_1,\ldots,z_{N-1})$ such that $z_1, \ldots, z_{j-1} \in S_2$ and $z_j, \ldots, z_{N-1} \in S_1$.  
In other words,
$$ \Sigma_j := S_2^{j-1} \times S_1^{N-j}.$$
Note that each $\Sigma_j$ consists of $2^{N-1}$ elements, and they are all disjoint.  We then set $\Sigma = \Sigma_1 \cup \ldots \cup \Sigma_N$;
this set consists of $N 2^{N-1}$ elements.  We refer to $\Sigma_j$ as the \emph{$j^{th}$ generation} of $\Sigma$.

For each $1 \leq j < N$, we define a \emph{combinatorial nuclear family connecting generations $\Sigma_j, \Sigma_{j+1}$} to be any
four-element set $F \subset
\Sigma_j \cup \Sigma_{j+1}$ of the form
$$ F := \{ (z_1,\ldots,z_{j-1}, w, z_{j+1}, \ldots, z_N): w \in S \}$$
where $z_1, \ldots, z_{j-1} \in S_2$ and $z_{j+1}, \ldots, z_N \in S_1$.  In other words, we have
$$ F = \{ F_0, F_1, F_{1+i}, F_i \} = = \{(z_1,\ldots,z_{j-1})\} \times S \times \{ (z_{j+1}, \ldots, z_N) \}$$
where $F_w = (z_1,\ldots,z_{j-1}, w, z_{j+1}, \ldots, z_N)$.
It is clear that $F$ is a four-element set consisting of two elements $F_1, F_i$ of $\Sigma_j$ (which we call the \emph{parents} in $F$)
and two elements $F_0, F_{1+i}$ of $\Sigma_{j+1}$ (which we call the \emph{children} in $F$).  For each $j$ there
are $2^{N-2}$ combinatorial nuclear families connecting the generations $\Sigma_j$ and $\Sigma_{j+1}$.  One easily verifies the following properties:

\begin{itemize}

\item (Existence and uniqueness of spouse and children) For any $1 \leq j < N$ and any $x \in \Sigma_j$ there exists a unique combinatorial nuclear family $F$
connecting $\Sigma_j$ to $\Sigma_{j+1}$ such that $x$ is a parent of this family (i.e. $x = F_1$ or $x = F_i$).  In particular each $x \in \Sigma_j$
has a unique spouse (in $\Sigma_j$) and two unique children (in $\Sigma_{j+1}$).

\item (Existence and uniqueness of sibling and parents) For any $1 \leq j < N$ and any $y \in \Sigma_{j+1}$ there exists a unique combinatorial
nuclear family $F$ connecting $\Sigma_j$ to $\Sigma_{j+1}$ such that $y$ is a child of the family (i.e. $y = F_0$ or $y = F_{1+i}$).
In particular each $y \in \Sigma_{j+1}$ has a unique sibling (in $\Sigma_{j+1}$) and two unique parents (in $\Sigma_j$).

\item (Nondegeneracy) The sibling of an element $x \in \Sigma_j$ is never equal to its spouse.

\end{itemize}

{\bf Example.} If $N=7$, the point $x = (0, 1+i, 0, i, i, 1)$ lies in the fourth generation $\Sigma_4$.  Its spouse is
$(0, 1+i, 0, 1, i, 1)$ (also in $\Sigma_4$) and its two children are $(0, 1+i, 0, 0, i, 1)$ and $(0, 1+i, 0, 1+i, i, 1)$
(both in $\Sigma_5$).  These four points form a combinatorial nuclear family connecting the generations $\Sigma_4$ and $\Sigma_5$.
The sibling of $x$ is
$(0, 1+i, 1+i, i, i, 1)$ (also in $\Sigma_4$, but distinct from the spouse) and its two parents
are $(0, 1+i, 1, i, i, 1)$ and $(0, 1+i, i, i, i, 1)$ (both in $\Sigma_3$).  These four points form a combinatorial nuclear family
connecting the generations $\Sigma_3$ and $\Sigma_4$. Elements of $\Sigma_1$ do not have siblings or parents,
and elements of $\Sigma_7$ do not have spouses or children.

Now we need to embed $\Sigma$ into the frequency lattice $\Z^2$.  We shall abuse notation and identify this lattice $\Z^2$ with the
Gaussian integers $\Z[i]$ in the usual manner.  We shall need a number of parameters:

\begin{itemize}

\item (Placement of initial generation) We will need a function $f_1: \Sigma_1 \to \C$ which assigns to each point $x$ in the first generation,
a location $f_1(x)$ in the complex plane (eventually we will choose $f$ to take values in the Gaussian integers).

\item (Angle of each nuclear family) For each $1 \leq j < N$ and each combinatorial nuclear family $F$ connecting the generations $\Sigma_j$ and $\Sigma_{j+1}$, we need an angle $\theta(F) \in \R/2\pi\Z$.

\end{itemize}

Given the placement function of the first generation, and given all the angles, we can then recursively define the placement function $f_j: \Sigma_j \to \C$ for later generations $2 \leq j \leq N$ by the following rule:

\begin{itemize}

\item If $1 \leq j < N$ and $f_j: \Sigma_j \to \C$ has already been constructed, we define $f_{j+1}: \Sigma_{j+1} \to \C$ by requiring
\begin{align*}
f_{j+1}( F_{1+i} ) &= \frac{1 + e^{i\theta(F)}}{2} f_j(F_1) +  \frac{1 - e^{i\theta(F)}}{2} f_j(F_i)\\
f_{j+1}( F_0 ) &= \frac{1 + e^{i\theta(F)}}{2} f_j(F_1) -  \frac{1 - e^{i\theta(F)}}{2} f_j(F_i)
\end{align*}
for all combinatorial nuclear families $F$ connecting $\Sigma_j$ to
$\Sigma_{j+1}$.

\end{itemize}

In other words, we require
$f_j(F_1), f_{j+1}(F_{1+i}), f_j(F_i), f_{j+1}(F_0)$ to form the four points of a rectangle in $\C$, with the long diagonals intersecting at
angle $\theta(F)$.

Note that this definition is well-defined thanks to the existence and uniqueness of parents.

{\bf A model example}:  Let $\rad$ be a large integer, and let the initial placement function be
$$ f_1(z_1, \ldots, z_{N-1}) := \rad z_1 \ldots z_{N-1} \in \{\rad, i\rad, -\rad, -i\rad \} \hbox{ for all } (z_1,\ldots,z_{M-1}) \in \Sigma_1.$$
Set all the angles $\theta(F)$ to equal $\pi/2$.  Then one can show inductively that
$$ f_j(z_1, \ldots, z_{N-1}) := \rad z_1 \ldots z_{N-1} \in \{0, (1+i)^{j-1} \rad, i (1+i)^{j-1} \rad, - (1+i)^{j-1} \rad, -i (1+i)^{j-1} \rad \}
\hbox{ for all } (z_1,\ldots,z_{N-1}) \in \Sigma_j.$$
Note that as the generations increase, an increasing majority of the points in that generation will get mapped to zero by the placement
function, but an increasingly small minority will get mapped to larger and larger frequencies.  For the final generation
$\Sigma_N$, there is a single element $(1+i,\ldots,1+i)$ which is mapped to the very large frequency $(1+i)^N \rad$, whereas all the other
$2^{N-1}-1$ frequencies are mapped to zero.  Notice that in each combinatorial nuclear
family $F = \{F_0, F_1, F_{1+i}, F_i\}$ connecting $\Sigma_j$ to $\Sigma_{j+1}$, the two parents $F_1$ and $F_i$ get mapped
to frequencies of equal magnitude (indeed $f_j(F_i) = i f_j(F_1)$), but of the two children, $F_0$ and $F_{1+i}$, one of the children
(the ``under-achiever'' $F_0$) gets mapped to 0, whereas the other child (the ``over-achiever'' $F_{1+i}$) gets mapped to a frequency of magnitude
$\sqrt{2}$ as large as that of its parents.  Our NLS solution will distribute mass evenly from $F_1$ and $F_i$ to $F_0$ and $F_{1+i}$, but
will distribute the energy from $F_1$ and $F_i$ almost entirely to $F_{1+i}$, thus sending
the energy to increasingly high frequencies.
Indeed by simply counting the number of nonzero frequencies in each generation, one easily verifies that
$$ \sum_{n \in f(\Sigma_{N-2})} |n|^{2s} = 2^{s(N-3) + 2} \rad^{2s}$$
and
$$ \sum_{n \in f(\Sigma_3)} |n|^{2s} = 2^{2s + N - 3} \rad^{2s}$$
and so there is a norm explosion by a factor of $2^{(s-1)(N-5)}$.

We define $f: \Sigma \to \C$ to be the function formed by concatenating the individual
functions $f_j: \Sigma_j \to \C$. The function $f$ will be referred to as the \emph{complete placement function}.

Figure \ref{fig:planting_m5} contains a sketch of the complete placement function's image
in the case $N = 5$.  (We emphasize that this function is not injective - for example
every generation after the first has several 4-tuples mapped to the origin.)

\begin{figure}[htp]
\centering
\includegraphics[width=12cm]{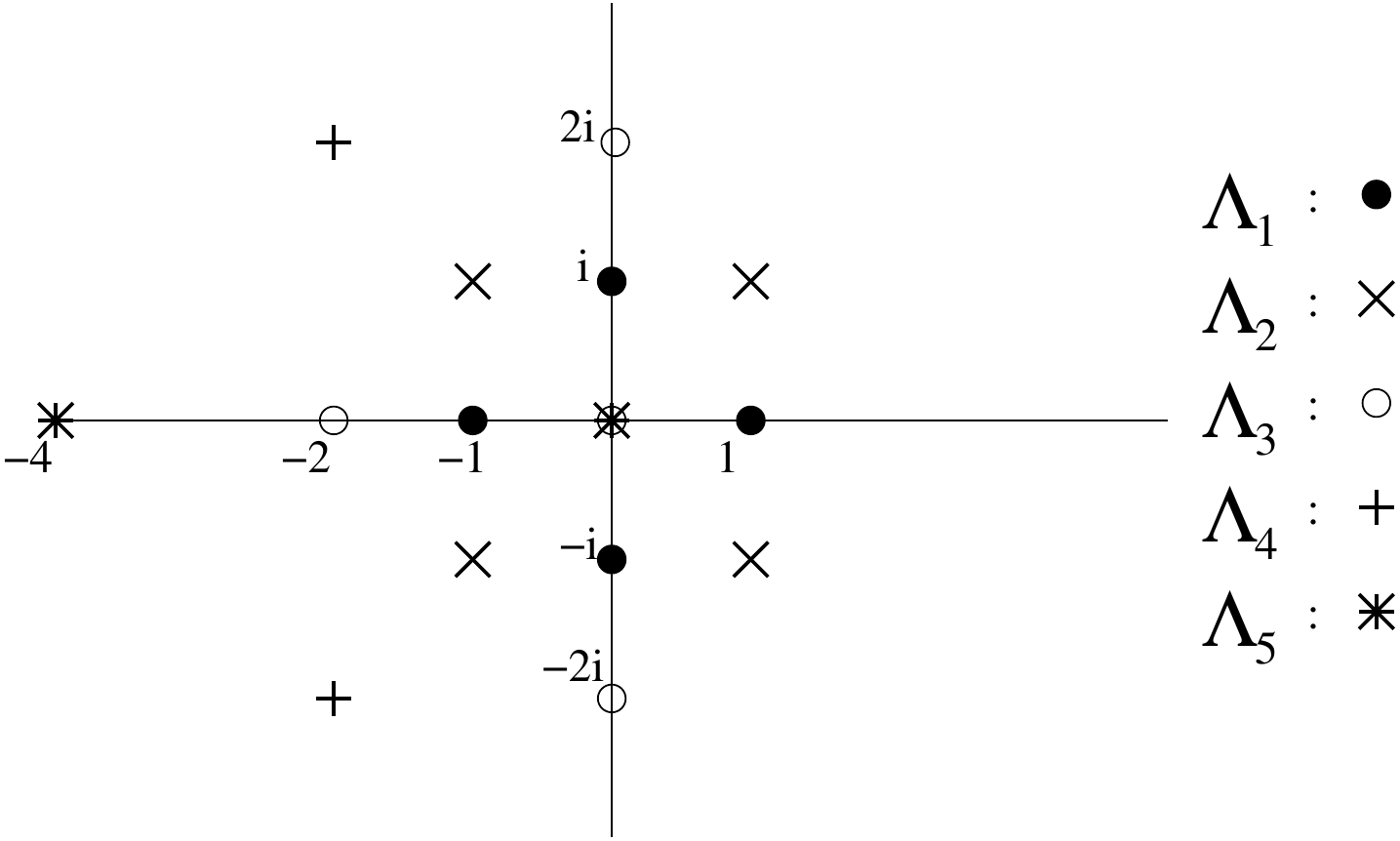}
\caption{The initial placement of frequencies in the case $N = 5$.}\label{fig:planting_m5}
\end{figure}

The model example turns out not to be directly usable for us because it is highly non-injective,
and also contains a large number of unwanted
rectangles.  But this can be fixed by a perturbation argument:

\begin{theorem}[Construction of a good placement function]  Let $N \geq 2$, $s > 1$, and let $\rad$ be a sufficiently large integer (depending on $N$).  Then
there exists an initial placement function $f_1: \Sigma_1 \to \C$ and choices of angles $\theta(F)$ for each nuclear family $F$ (and thus an associated complete placement function $f: \Sigma \to \C$) with the following properties:

\begin{itemize}

\item (Nondegeneracy) The function $f$ is injective.

\item (Integrality) We have $f(\Sigma) \subset \Z[i]$.

\item (Magnitude) We have $C(N)^{-1} \rad \leq |f(x)| \leq C(N) \rad$ for all $x \in \Sigma$.

\item (Closure and Faithfulness)  If $x_1, x_2, x_3 \in \Sigma$ are distinct elements of $\Sigma$ are such that $f(x_1), f(x_2), f(x_3)$ form the three
corners of a right-angled triangle, then $x_1, x_2, x_3$ belong to a combinatorial nuclear family.

\item (Norm explosion) We have
$$ \sum_{n \in f(\Sigma_{N-2})} |n|^{2s} > \frac{1}{2} 2^{(s-1)(N-5)} \sum_{n \in f(\Sigma_3)} |n|^{2s}.$$

\end{itemize}

\end{theorem}

Providing this theorem is true, we see that the frequency set $\Lambda := f(\Sigma)$, with generations $\Lambda_j := f(\Sigma_j)$, obey
all the required properties described in Section \ref{overview} above (See e.g. the statement of Proposition \ref{frequencies}.)

\begin{proof} To prove this theorem, we begin by a number of reductions which allow us to remove several of the constraints required on the
placement function.

First we remove the requirement that the construction work for $\rad$ sufficiently large, by observing that as soon as we obtain an example involving a single $\rad=\rad_0$, we also get examples for any integer multiple $\rad = k\rad_0$ of $\rad_0$ simply by multiplying $f_1$ (and hence $f$) by $k$.  Since $\rad$
only appears in the magnitude hypothesis, we also get the same claim for any $\rad > \rad_0$ since any such $\rad$ is comparable up to a factor of two to
an integer multiple of $\rad$.

Thus we only need to exhibit an example for a single $\rad = \rad(N)$.  But then the magnitude hypothesis would follow from the weaker hypothesis

\begin{itemize}
\item (Non-zero) $f(x) \neq 0$ for all $x \in \Sigma$.
\end{itemize}

and we can now drop the magnitude hypothesis.

Next, we can weaken the integrality hypothesis $f(\Sigma) \subset
\Z[i]$ to a rationality hypothesis
$f(\Sigma) \subset \Q [i]$, since by
dilating $f_1$ (and hence $f$) by a suitable large integer (the least common denominator of all the coefficients of $f(\Sigma)$).

Now observe that in order for $f(\Sigma)$ to lie in the complex rationals $\Q[i]$, it would suffice (by an easy induction argument)
for the initial placement function $f_1$
to take values in $\Q[i]$, and for all the phases $\theta(F)$ to be \emph{Pythagorean}, i.e. $e^{i \theta(F)} \in \Q[i]$.  (A typical example of
a Pythagorean phase is $\tan^{-1} \frac{3}{5}$).  So we may replace the rationality hypothesis by the hypothesis that
the initial placement function is complex-rational and all the phases are Pythagorean.

However, the complex rationals $\Q[i]$ are dense in $\C$, and the Pythagorean phases are dense in $\R/2\pi\Z$ (indeed they form an infinite
subgroup of this unit circle).  Thus by a perturbation argument we may dispense with these conditions altogether.  Note that
the remaining conditions required are all open conditions, in that they are stable under sufficiently small perturbations of the
parameters.  Thus any non-rational solution to the above problem can be perturbed into a rational one.

We have now eliminated the integrality and magnitude hypotheses.  We now work on eliminating the non-zero, injectivity, and closure/faithfulness
hypotheses.  The point is that in the parameter space (the space of initial placement functions $f_1$ and phases $\theta(F)$), the set of
solutions to the norm explosion property is certainly an open set, so it is either empty or has positive measure.
So it suffices to show that the sets where
each of the non-zero, injectivity, and closure/faithfulness fail is a measure zero set, and thus they have no impact as to whether a solution actually exists or not.

Let's begin with the injectivity condition.  We need to show that for each distinct $x, y \in \Sigma$, that we have $f(x) \neq f(y)$ for almost
every choice of parameters $f_1$, $\theta$. Without loss of generality we may take $x \in \Sigma_j$ and $y \in \Sigma_{j'}$ where $j \geq j'$.
We induct on $j$.  When $j=1$ the claim is clear since in that case $f(x) = f_1(x)$ and $f(y) = f_1(y)$, and since there are absolutely no
constraints on $f_1$ we see that $f_1(x) \neq f_1(y)$ for almost every choice of $f_1$  (the angles $\theta(F)$ are not relevant in this case).

Now suppose $j > 1$.  Since $x \in \Sigma_j$, there is a unique
combinatorial family $F := \{ x,x',p,p'\}$ where $p,p' \in \Sigma_{j-1}$ are the parents of $x$ and $x'$ is the sibling of $x$.
By induction hypothesis, we have $f(p) \neq f(p')$ for almost every choice of parameters.  Now note that $f(x)$ lies on the
circle with diameter $f(p)$, $f(p')$, with the location on this circle determined by the angle $\theta(F)$.  This angle $\theta(F)$ is a free
parameter and will not influence the value of $f(y)$, unless $y$ is equal to the sibling $x'$ of $x$.  But in the latter case $f(y) = f(x')$
will be diametrically opposite to $f(x)$ and thus unequal to $f(x)$ (since $f(p) \neq f(p')$).  In all other cases we see that for almost
every choice of $\theta(F)$, $f(x)$ will not be equal to $f(y)$, simply because a randomly chosen point on a circle will almost surely be unequal
to any fixed point.  This establishes injectivity almost everywhere.

A similar argument establishes the non-zero property almost everywhere: if $x$ is in the first generation $\Sigma_1$ then it is clear
that $f(x) = f_1(x) \neq 0$ for almost every $f_1$, and for $x$ in any later generation, $f(x)$ again 
ranges freely on the circle generated
by its two parents $f(p)$, $f(p')$ (which are distinct for almost every choice of parameters), determined by some angle parameter $\theta(F)$
and thus will be non-zero for almost all choices of $\theta(F)$.

Now we establish closure and faithfulness.  Suppose $x \in \Sigma_{j_x}, y \in \Sigma_{j_y}, z \in \Sigma_{j_z}$ with $j_x \geq j_y \geq j_z$
are distinct and
do not all belong to the same combinatorial nuclear family.  We need to show that $f(x), f(y), f(z)$ do not form a right-angled triangle (with either $f(x)$, $f(y)$, or $f(z)$ being the right angle) for almost all choice of parameters.

We induct on $j_x$.  When $j_x = 1$, we have $f(x) = f_1(x)$, $f(y) = f_1(y)$, $f(z) = f_1(z)$ and since $f_1$ is completely unconstrained it is clear
that these three points will not form a right-angled triangle for almost all choices of parameters.  Now we assume inductively that
$j_x > 1$ and that the claim has already been proven for smaller values of $j_x$.

As before, $x$ belongs to a family $F = \{x,x',p,p'\}$ consisting of $x$, its sibling, and its parents, 
ranges freely on the circle $C$ with diameter $f(p)$, $f(p')$ generated by its two parents, with the position on this circle determined by the angle $\theta(F)$; recall that
$f(p) \neq f(p')$ for almost all choices of parameters.  This angle $\theta(F)$
will not influence either of $f(y)$ or $f(z)$ unless $y$ or $z$ is equal to the sibling $x'$, in which case $f(x')$ is the point diametrically
opposite to $f(x)$.  To proceed further, we next claim that the circle $C$ contains no elements of $\Lambda_1 \cup \ldots \cup \Lambda_{j_x}$
other than $f(x), f(x'), f(p), f(p')$.  To see this for the first $j_x-1$ generations $\Lambda_1 \cup \ldots \cup \Lambda_{j_x-1}$ follows
from the induction hypothesis, because if $C$ contained a point $f(u)$ from one of those generations then $f(p)$, $f(p')$, $f(u)$ would
be a right-angled triangle, contradicting the induction hypothesis for almost all choices of parameters.  To see the claim for
the last generation $\Lambda_{j_x}$, we note that any point $f(v)$ from that generation ranges on another circle $C'$ spanned by the parents
$q, q'$ of $v$; since $f(q)$ and $f(q')$ do not lie on $C$ we know that $C'$ is not co-incident to $C$, and so for almost all choices of the angle
parameter determining $f(v)$ from $f(q)$ and $f(q')$ we know that $f(v)$ does not lie on $C$.

To summarize, we know that $f(x)$ varies freely on the circle $C$ as determined by the angle $\theta(F)$, that $f(x')$ is diametrically opposite
to $f(x)$, that $f(p)$, $f(p')$ are also
diametrically opposite points on $C$, and all other points in $\Lambda_1 \cup \ldots \cup \Lambda_{j_x}$ lie outside of $C$
and are not affected by $\theta(F)$ (for almost all choices of parameters).  Since $\{x,y,z\} \not \subset F =  \{x,x',p,p'\}$ by hypothesis,
this shows that for almost all choices of $\theta(F)$, $f(x)$, $f(y)$, $f(z)$ do not form a right-angled triangle.

From the preceding discussion we see that all we need to do now is establish a single example of parameters which exhibits the norm
explosion property.  But this was already achieved by the model example, and we are done.

\end{proof}

\section{Appendix: Cubic NLS on $\T^2$ has No Asymptotically Linear Solutions}

\label{appendix}

Consider $H^1$ solutions $u: \R \times \T^2 \to \C$ to  the defocusing periodic NLS
\begin{equation}\label{nlsappendix}
-iu_t + \Delta u = |u|^2 u
\end{equation}
on the torus $\T^2$.  We do not expect such solutions to scatter to a free solution $e^{-it\Delta} u_+$.  Indeed, the explicit solutions
to \eqref{nlsappendix}
\begin{equation}\label{explicit}
u(t,x) = A e^{i\kappa} e^{i n \cdot x} e^{i|n|^2 t} e^{i|A|^2 t}
\end{equation}
for $A \geq 0$, $\kappa \in \R/2\pi\Z$, $n \in \Z^2$ do not converge to a free solution,
due to the presence of the phase rotation $e^{i|A|^2 t}$ caused by the nonlinearity.
One may still hope that a solution to \eqref{nlsappendix} scatters ``modulo phase rotation''.
However, the following result shows that this is only the case for the
explicit solution \eqref{explicit} mentioned above.

\begin{theorem}[No non-trivial scattering modulo phase rotation]\label{noscat}  Let $u: \R \times \T^2 \to \C$ be an $H^1$ solution to
\eqref{nlsappendix} which \emph{scatters modulo phases in $H^1$} in the sense that there exists $u_+ \in H^1(\T^2)$ and a function $\theta: \R \to \R/2\pi\Z$ such that
$$ \| u(t) - e^{i\theta(t)} e^{-it\Delta} u_+ \|_{H^1(\R^2)} \to 0$$
as $t \to +\infty$.  Then $u$ is of the form \eqref{explicit} for some $A \geq 0$, $\kappa \in \R/2\pi\Z$, $n \in \Z^2$.
\end{theorem}

To prove this theorem, we first need some harmonic analysis lemmas.

\begin{lemma}[Compactness]\label{precomp}  For any $u_+ \in H^1(\T^2)$, the 
set $\{ e^{i\theta} e^{-it\Delta} u_+: \theta \in \R/2\pi\Z, t \in \R \}$ is precompact (i.e. totally bounded) in $H^1$.
\end{lemma}

\begin{proof} By taking Fourier transforms, we see that it suffices to show that the set $\{ e^{i\theta} e^{it|\cdot|^2} \widehat{u_+}: \theta \in \R/2\pi\Z, t \in \R \}$ is precompact in $l^2(\langle n\rangle^2\ dn)$.  But by monotone convergence, for every $\eps > 0$ there exists $ R> 0$ such that
$$ \sum_{n \in \Z^2: |n| \geq R} \langle n \rangle^2 |\widehat{u_+}(n)|^2 \leq \eps$$
from which we easily conclude that the set $\{ e^{i\theta} e^{it|\cdot|^2} \widehat{u_+}: \theta \in \R/2\pi\Z, t \in \R \}$ is covered by finitely many balls of radius $O(\eps)$, and the claim follows.
\end{proof}

\begin{lemma}[Diamagnetic inequality]\label{dia}  If $u \in H^1(\T^2)$, then $|u| \in H^1(\T^2)$ and $\| |u| \|_{H^1(\T^2)} \leq \| u \|_{H^1(\T^2)}$.
\end{lemma}

\begin{proof}  By a limiting argument (and Fatou's lemma) it suffices to verify this when $u$ is smooth.  Observe that for any $\eps > 0$ we have
\begin{align*}
2 |\sqrt{\eps^2 + |u|^2} \nabla \sqrt{\eps^2 + |u|^2}| &=
|\nabla (\eps^2 + |u|^2)| \\
&= 2 | \Re( \overline{u} \nabla u )| \\
&\leq 2 |u| |\nabla u|
\end{align*}
and hence
$$ |\nabla \sqrt{\eps^2+|u|^2}| \leq |\nabla u|.$$
(This can also be seen by observing that the map $u \mapsto \sqrt{\eps^2+|u|^2}$ is a contraction.)  Taking distributional limits as $\eps \to 0$ we obtain the claim.
\end{proof}

\begin{lemma}[$H^1$ has no step functions]\label{nostep}  Let $u \in H^1(\T^2)$ be such that $u(x)$ takes on at most two values.  Then $u$ is constant.
\end{lemma}

\begin{proof} Without loss of generality we may assume that $u$ takes on the values $0$ and $1$ only, thus $u^2 = u$.  Differentiating this we obtain $2 u \nabla u = \nabla u$, thus $(1-2u) \nabla u = 0$.  But since $u^2 = u$, $(1-2u)^2 = 1$, and so $\nabla u = 0$, and so $u$ is constant as required.  (Note that all these computations can be justified in a distributional sense via Sobolev embedding since $u$ lies in both $H^1$ and $L^\infty$, and thus lies in $L^p(\T^2)$ for every $1 \leq p \leq \infty$.)
\end{proof}

\begin{proof}[Proof of Theorem \ref{noscat}]  Let $u, u_+, \theta$ be as above.  We may of course assume that $u$ has non-zero mass. From Lemma \ref{precomp} (and the continuity in $H^1$ of the solution $t \mapsto u(t)$) we see that the set $\{ u(t): 0 \leq t \leq +\infty \}$ is precompact in $H^1$.  Thus we can find a sequence $t_m \to \infty$ such that $u(t_m)$ converges in $H^1$ to some limit $v_0$.  
Applying Lemma \ref{precomp} and passing to a subsequence, we can also assume that $e^{-it_m\Delta} u_+$ 
converges in $H^1$ to a limit $v_+$.
Since $u$ has non-zero mass, we see on taking limits that $v_+$ also has non-zero mass.

Now let $v^{(m)}: \R \times \T^2 \to \C$ be the time-translated solutions $v^{(m)}(t) := u(t+t_m)$, thus $v^{(m)}(0)$ 
converges in $H^1$ to $v_0$.  Let $v: \R \times \T^2 \to \C$ be the solution to NLS with initial data $v(0)=v_0$.  
By the $H^1$ local well-posedness theory we conclude that $v^{(m)}$ converges uniformly in $H^1$ to $v$ on every compact 
time interval $[-T,T]$.  
On the other hand, by hypothesis, for every $t$ we have
$$ \| v^{(m)}(t) - e^{i\theta(t+t_m)} e^{-i(t+t_m)\Delta} u_+ \|_{H^1} \to 0$$
as $m \to \infty$.  Since $e^{-it_m\Delta} u_+$ converges to $v_+$, we conclude (since $e^{-it\Delta}$ is unitary on $H^1$)
 that
$$ \| v^{(m)}(t) - e^{i\theta(t+t_m)} e^{-i t\Delta} v_+ \|_{H^1} \to 0$$
On taking limits, we conclude that
\begin{equation}\label{vad}
 v(t,x) = e^{i\alpha(t)} e^{-it\Delta} v_+(x)
\end{equation}
for some $\alpha(t) \in \R/2\pi\Z$.  In particular, since $v_+$ has non-zero mass, we have the identity
$$e^{i\alpha(t)} = \frac{1}{\|v_+\|_{L^2}^2} \int_{\T^2} v(t) \overline{e^{-it\Delta} v_+}\ dx.$$
From \eqref{nlsappendix} and Sobolev embedding we see that $v(t)$ is continuously
differentiable in $H^{-1}(\T^2)$, and so from the above identity we see that $\alpha$ is continuously differentiable in time.

Now we apply $(-i\partial_t + \Delta)$ to both sides of \eqref{vad}.  Using the NLS equation, we conclude that
$$ |v(t,x)|^2 v(t,x) =  \alpha'(t) e^{i\alpha(t)} e^{-it\Delta} v_+(x)$$
and thus by \eqref{vad} we have
$$ |v(t,x)|^2 =  \alpha'(t)$$
whenever $v(t,x) \neq 0$.  Thus we see that for each time $t$, $|v(t,x)|$ takes on at most two values.  
By Lemma \ref{dia} and Lemma \ref{nostep}, we conclude that $|v(t,x)|$ is constant in $x$; by \eqref{vad}, we 
see that the same is true for $|e^{-it\Delta} v_+|$.  
By mass conservation we conclude that $|e^{-it\Delta} v_+|$ is also constant in time.  
Since $v_+$ has non-zero mass, we can thus write
$$ e^{-it\Delta} v_+ = A e^{i \omega(t,x)}$$
for some phase $\omega: \R \times \T^2 \to \R/2\pi\Z$ and some constant amplitude $A > 0$.  Since $v_+$ was in $H^1$, we see 
that $\omega$ is in $H^1$ also.  Applying $(-i\partial_t + \Delta)$ to both sides, we conclude that
$$ 0 = A e^{-i\omega} ( \omega_t + i \Delta \omega - |\nabla \omega|^2 )$$
in the sense of distributions.  Since $A$ is non-zero, we conclude that
$$
 \omega_t + i \Delta \omega - |\nabla \omega|^2  = 0.
$$
Taking imaginary parts we conclude that $\Delta \omega = 0$, and in particular at time $t=0$ $\omega$ is a 
harmonic map from $\T^2$ to $\R/2\pi\Z$.  Using monodromy we can lift $\omega$ to a harmonic function 
from $\R^2$ to $\R$ of at most linear growth, and thus (by Liouville's theorem) $\omega$ must in fact be linear.  
Descending back to $\T^2$, we conclude that $\omega(0,x) = n \cdot x + \beta$ for some $n \in \Z^2$ and $\beta 
\in \R/2\pi\Z$.  Thus we have
$$ v_+(x) = A e^{i \beta} e^{in \cdot x}.$$
Since $e^{it_m\Delta} u_+$ converges to $v_+$, $e^{-it_m \Delta} v_+$ converges to $u_+$.  
But $e^{-it_m \Delta} v_+$ is a multiple of $A e^{i n \cdot x}$ by a phase; thus we have
$$ u_+(x) = A e^{i\gamma} e^{in \cdot x}.$$
Applying phase rotation and Galilean symmetry (noting that the conclusion of the theorem is invariant under 
these symmetries) we may assume $\gamma=n=0$.  Thus $u_+ = A$, and we have
$$ \| u(t) - e^{i\theta(t)}e^{-it \Delta} A \|_{H^1(\T^2)} \to 0 \hbox{ as } t \to \infty.$$
From mass and energy conservation we conclude
$$ \int_{\T^2} |u(t,x)|^2\ dx = A^2 |\T^2|$$
and
$$ \int_{\T^2} \frac{1}{2} |\nabla u(t,x)|^2 + \frac{1}{4} |u(t,x)|^4 \ dx = \frac{1}{4} A^4 |\T^2|.$$
On the other hand, from H\"older's inequality we have
$$ \int_{\T^2} \frac{1}{4} |u(t,x)|^4 \ dx \geq \frac{1}{4 |\T^2|} (\int_{\T^2} |u(t,x)|^2\ dx)^2 = \frac{1}{4} A^4 |\T^2|.$$
Thus we must have
$$ \int_{\T^2} \frac{1}{2} |\nabla u(t,x)|^2\ dx = 0,$$
thus $u$ is constant in space, and thus is of the form $u(t,x) = A e^{i\kappa(t)}$.
Applying \eqref{nlsappendix} we see that $u(t,x) = A e^{i \kappa(0)} e^{i |A|^2 t}$, and the claim follows.
\end{proof}

\end{document}